\documentclass[reqno]{amsart}
\usepackage{amssymb, amsmath, mathtools}
\usepackage[usenames, dvipsnames]{color}
\usepackage{enumerate}
\usepackage{hyperref}
\usepackage{todonotes}

\usepackage{verbatim}

\usepackage[T1]{fontenc} 

\numberwithin{equation}{section}

\newtheorem{theorem}{Theorem}[section]

\newtheorem{lemma}[theorem]{Lemma}
\newtheorem{proposition}[theorem]{Proposition}

\theoremstyle{definition}
\newtheorem{remark}[theorem]{Remark}

\theoremstyle{definition}
\newtheorem{definition}[theorem]{Definition}

\theoremstyle{definition}

\makeatletter
\def\dashint{\operatorname%
{\,\,\text{\bf--}\kern-.98em\DOTSI\intop\ilimits@\!\!}}
\makeatother

\def\bH{\mathbb{H}}

\def\bR{\mathbb{R}}

\def\bZ{\mathbb{Z}}

\def\cA{\mathcal{A}}

\def\cF{\mathcal{F}}

\def\cH{\mathcal{H}}

\def\cM{\mathcal{M}}

\def\cS{\mathcal{S}}
\def\cT{\mathcal{T}}

\def\cI{\mathcal{I}}
\def\cX{\mathcal{X}}

\begin{document}
\title[Trace theorems for time-fractional equations]{Sobolev spaces and trace theorems for time-fractional evolution equations}

\author[D. Kim]{Doyoon Kim}
\address[D. Kim]{Department of Mathematics, Korea University, 145 Anam-ro, Seongbuk-gu, Seoul, 02841, Republic of Korea}
\email{\href{mailto:doyoon_kim@korea.ac.kr}{\nolinkurl{doyoon_kim@korea.ac.kr}}}

\author[K. Woo]{Kwan Woo}
\address[K. Woo]{Department of Mathematics, Korea University, 145 Anam-ro, Seongbuk-gu, Seoul, 02841, Republic of Korea}
\email{\href{mailto:wkwan@korea.ac.kr}{\nolinkurl{wkwan@korea.ac.kr}}}

\thanks{D. Kim and K. Woo were supported by the National Research Foundation of Korea (NRF) funded by the Korean government (MSIT) (2019R1A2C1084683).}



\begin{abstract}
We establish trace and extension theorems for evolutionary equations with the Caputo fractional derivatives in (weighted) $L_p$ spaces.
To achieve this, we identify weighted Sobolev and Besov spaces with mixed norms that accommodate solution spaces and their initial values well-suited for equations involving time-fractional derivatives.
Our analysis encompasses both time-fractional sub-diffusion and super-diffusion equations.
We also provide observations on the initial behavior of solutions to time-fractional equations.
\end{abstract}

\maketitle

\section{Introduction}
	\label{intro}
This paper focuses on the $ L_p$ theory of the initial value problem for evolution equations involving time-fractional derivatives.
A representative example of the time-fractional equations we consider is given by:
\begin{align}
	\label{eq230510_1}
    \begin{cases}
    \partial_t^{\alpha}u - \Delta u = 0 &\quad \text{in}\quad  {(0, T) \times \Omega}, \\
    u(0, \cdot) = u_0 (\cdot)&\quad \text{on}\quad  \Omega,
    \end{cases} 
\end{align}
where $\Delta$ can be replaced with $a^{ij}D_{ij}$ (non-divergence form) or $D_i(a^{ij}D_ju)$ (divergence form).
Here, $\partial_t^{\alpha}$ represents the Caputo fractional derivative of order $\alpha \in (0, 1)$, defined as
\[
\partial_t^{\alpha}u (t, x) = \frac{1}{\Gamma(1-\alpha)} \frac{\mathrm{d}}{\mathrm{d}t} \int_0^t (t-s)^{-\alpha} \left( u \left( s, x \right) - u_0\left( x \right) \right) \,\mathrm{d}s.
\]

Time-fractional equations as in \eqref{eq230510_1} have been widely recognized as a valuable approach for describing dynamics governed by anomalous diffusion, as discussed, for instance, in \cite{MR1809268}.
Thus, the mathematical analysis of such equations becomes essential.
The main goal of this paper is to establish an appropriate framework that can discuss initial value problems for equations of type \eqref{eq230510_1} in Sobolev spaces.
In particular, we require trace and extension theorems that identify the optimal function spaces for initial conditions.
To achieve this, we first explore various properties of $\partial_t^{\alpha}u$ as a function in the weighted mixed norm space $L_{p, q, w}$ and $\bH_{p, q, w}^{-1}$. See Section \ref{subsec_div} for the definition of $\bH_{p, q, w}^{-1}$.
Based on these properties, we define solution spaces and prove trace and extension theorems for these spaces.
Trace theorem provides a rigorous background for discussing the initial value of a solution in Sobolev spaces, while extension theorem ensures that the initial data space considered in the trace theorem is optimal.
Notably, it turns out that if the time regularity $\alpha$ is too small, we cannot define $u(0, \cdot)$ in $L_p$ sense adequately; see Remark \ref{rmk_alpha}.
This issue can be overcome by appropriately choosing $q$ (integrability with respect to the time variable) and temporal weights.
This is also why the mixed norm and weight should be considered; see Remark \ref{rmk_homogeneous} for a more detailed description. 
Note that if the order of time smoothness of solution $u$ is a positive integer (\textit{e.g.} $\alpha = 1$), then the precise regularity of $u(0, \cdot)$ is well-known in the literature, for example, see \cite{MR0503903}.

There has been a huge body of literature on time-fractional equations, and the following is a brief overview of the literature closely related to initial value problems in Sobolev-type spaces with fractional derivatives in the time variable.
In \cite{MR2125407} and \cite{MR2278673}, the authors obtained an unweighted and unmixed $L_q$-theory for the Volterra type equations in non-divergence form.
 The solution spaces used in their work are the vector-valued Bessel potential spaces $\bar{H}_q^{\alpha}$ (see Remark \ref{rmk_yamamoto}-(ii) for the definition of $\bar{H}_q^{\alpha}$). 
This theory is based on the operator theoretical approach and holds under the assumption that $\alpha > 1/q$.
For equations in divergence form, discussions on the $L_2$ theory within the Sobolev-Slobodetskii spaces can be found in \cite{MR4200127} and references therein (also see Remark \ref{rmk_yamamoto}-(i)).
A totally different approach was employed in \cite{MR4186022} to obtain weighted mixed norm estimates for solutions to \eqref{eq230510_1} under minimal assumptions on the coefficient $a^{ij} = a^{ij}(t, x)$. 
The focus in \cite{MR4186022} is the zero initial value problems, but it also addresses non-zero initial values (without trace theorem), which is related to the extension theorem for a specific class of initial values. 
Related results can be found in \cite{MR4242958} and \cite{MR4464544}.
Studies on traces of weighted vector-valued Sobolev-Slobodetskii as well as Bessel potential spaces can be found in \cite{MR2863860} and \cite{MR3273638}.
In a recent study \cite{MR4604539}, the authors established trace theorems for weighted Triebel-Lizorkin spaces, using temporal weights $t^{\mu}$ with the assumption that $\alpha > (1+\mu)/q$ and $\mu \in (-1, q-1)$.
The results presented in \cite{MR4604539} encompass many previous findings on trace theorems.
Additionally, in \cite[Section 6]{MR4604539}, the authors applied their results to derive trace estimates for solutions in Bessel potential spaces to various evolution equations, such as time-fractional diffusion equations.
Note that, except for \cite{MR4186022} (and \cite{MR4242958, MR4464544}), the aforementioned papers use the intersection of two Banach space-valued function spaces as a solution space to evolution equations in the form
$\cA^{\alpha,p}(J;X_0) \cap L^p(J;X_1)$,
where $\cA^{s,p}$ represents the Sobolev-Slobodetskii space, the Bessel
potential space, or the Triebel-Lizorkin space, and $(X_0,X_1)$ is a pair of Banach spaces.
In particular, when $J = (0,T)$, for instance, the Triebel-Lizorkin space $\cA^{\alpha,p}(J) = F_{p,q}^\alpha(J)$ is given by a restriction of $F_{p,q}^\alpha(\bR)$.
This restriction is unavoidable when defining function spaces on a finite time interval $J = (0, T)$ because the definition of $F^\alpha_{p,q}$ requires the whole space, i.e., $\bR$, especially when $\alpha$ is non-integer.

The results presented in this paper exhibit several methodological differences compared to previous studies on traces of Sobolev-type spaces with time-fractional regularity and initial value problems for time-fractional evolution equations.
For example, the literature \cite{MR1238939, MR2125407, MR2278673, MR4604539} makes use of semigroup and operator theory to address maximal $L_p$ regularity for time-fractional equations, which is suitable when regarding solutions to the equations as Banach space-valued functions.
In this approach, the operator $\partial_t^{\alpha}$ can be interpreted as $(\partial_t)^{\alpha}$ in a semigroup sense, and various properties of the domain of $(\partial_t)^{\alpha}$, such as embedding and interpolation results, can be utilized.
In contrast, we directly construct solution spaces (see Definition \ref{def0102_01}) for time-fractional evolution equations on a given time interval $(0, T)$ using fractional calculus.
We view our solutions as functions of $(t,x)$ defined on the domain $(0,T) \times \Omega \subset \bR^{d+1}$.
Our constructed spaces align well with the approach used in  \cite{MR3581300, MR3899965, MR4186022, MR4387945, MR4345837, MR4464544} for obtaining $L_{p}$ estimates of solutions to time-fractional parabolic equations (with rough coefficients).
In particular, if equations in divergence form contain singular drifts as in \cite{MR4387945}, it becomes necessary to consider the case $\partial_t^{\alpha}u -\Delta u = D_ig_i + \sum_{k=1}^m f_k$ with $g_i \in L_p$ and $f_k \in L_{p_k}$ ($p_k < p$).
For such cases, to use an operator theoretical approach, one may need to answer if solution spaces can be represented by an intersection of two Banach space-valued function spaces, especially when the spatial boundary is irregular.
Our construction easily applies to this case.
Furthermore, the trace theorems (see Theorems \ref{thm0111_01} and \ref{trace0110_1}) for the constructed spaces are proved using fundamental tools such as integration by parts and Hardy's inequality.
Thus, our approach can be seen as an alternative for researchers who prefer to work with function spaces for time-fractional equations without relying on semigroup and operator theory (an alternative approach for solution spaces related to those in this paper).
We not only concentrate on trace and extension theorems but also provide comprehensive explanations of the initial behavior of a solution, depending on the relationship between $\alpha$, $q$, and temporal weights.
See remarks in Section \ref{ftn_space}.
Furthermore, unlike the Bessel potential spaces $\bar{H}_q^{\alpha}(0, T)$, where the initial value of functions cannot be well-defined if $\alpha \le 1/q$, as illustrated by H\"ormander's example in \cite{MR1044427}, the initial value $u(0, x)$ of a solution $u(t, x)$ is well-defined even when $\alpha = 1/q$ (or $\alpha = (1+\mu)/q$ for the weighted case) in our setting.
It is worth mentioning that \cite{MR4604539} focuses on trace theorems (without extension theorems) for anisotropic weighted Tribel-Lizorkin spaces, and our trace theorems when $T \to \infty$ correspond to a special case of the trace results mentioned in \cite[Theorem 1.2]{MR4604539}.
Regarding extension theorems (see Theorems \ref{thm_subdiff} and \ref{thm_supdiff}), if the time regularity is $\alpha \in (0, 1)$, there are some results in \cite{MR3273638}, but they differ from Theorem \ref{thm_subdiff} in this paper because the function spaces for the spatial variables in \cite{MR3273638} are different from ours.
Note that we also prove Theorem \ref{thm_supdiff} for the case where the time-regularity is $1 + \alpha \in (1, 2)$.

Let us present one special case of the main results in this paper.
Our solution spaces for the time-fractional equations in non-divergence form are $\bH_{p, q, w}^{ \alpha, 2}$ introduced in Definition \ref{def0102_01}.
We prove that the weighted Besov spaces $B_{p, q, w_2}^{2\theta}$ (Definition \ref{def_besov}) are the exact function spaces for initial values for the time-fractional equations in $\bH_{p, q, w}^{ \alpha, 2}$.
More precisely, if $\theta = 1 - (1+\mu)/(q\alpha) \in (0, 1)$ (in particular, $\alpha > (1+\mu)/q$), we prove that the following trace operator is bounded (trace theorem)
\[
\cT: \bH_{p, q, w}^{ \alpha, 2} \to B_{p, q, w_2}^{2\theta}, \quad \textrm{where} \quad \cT u(t, x) := u(0, x)
\]
and prove that $B_{p, q, w_2}^{2\theta}$ is the smallest space satisfying the above (extension theorem).
As a key step in proving the trace theorem, we take an appropriate integral representation of $u \in \bH_{p, q, w}^{ \alpha, 2}$ with the help of a certain type of mollification $u^{(\varepsilon)}$, which is also used to obtain Sobolev inequality \cite{MR4387945, MR4345837} and regularity of solutions on the boundary of the spatial domain (i.e., lateral trace) \cite{arXiv:2105.05131}. 
If there is no spatial weight, i.e., $w_2(x) = 1$, we also give an alternative proof for the trace theorem (see Remark \ref{rmk_unweighted}).
For the extension theorem, we solve a homogeneous equation with a non-zero initial value $u_0 \in B_{p, q, w_2}^{2\theta}$.
To this end, we derive precise time decay estimates for the fundamental solution $P_{\alpha}$ of the time-fractional heat equation.
We also obtain trace estimates of $\partial_t^{k+\alpha}u$ (non-negative integer $k$), and solve the initial value problem for super-diffusion equation $\partial_t^{1+\alpha}u - \Delta u = 0$.

The remainder of the paper is organized as follows.
In Section \ref{ftn_space}, we introduce spaces with fractional regularity in time and their basic properties.
In Section \ref{sol_space}, we introduce solution spaces and initial trace spaces for the time-fractional equations and state main theorems; Theorems \ref{thm0111_01}, \ref{cor0111_01} (trace theorems) and Theorems \ref{thm_subdiff} and \ref{thm_supdiff} (extension theorems).
Then we prove the main theorems in Section \ref{sec_trace}.
Also, in Section \ref{sec_trace}, we present another version of the trace theorem (Theorem \ref{thm230416}).

In this paper, $\bR^d$ stands for the $d$-dimensional Euclidean space and $x = (x_1, \ldots, x_d)$ denotes a point in $\bR^d$.
We set $\bR = \bR^1$ and $\bR_+ = (0, \infty)$.
For functions $u = u(t, x)$, we use the following standard notation
\[
D_i = \partial / \partial x_i, \quad Du = (D_1u, \ldots, D_d u).
\]
For a multi-index $\gamma = (\gamma_1, \ldots, \gamma_d)$ where $\gamma_i$'s are non-negative integers, we denote
\[
D^{\gamma} = D_1^{\gamma_1} \ldots D_d^{\gamma_d}, \quad |\gamma| = \gamma_1 + \ldots + \gamma_d.
\]
By $\partial_t^n u$, we mean a partial derivative of order $n$ with respect to $t$. 
Frequently, we use $u_t$ instead of $\partial_tu$.
For a domain $\Omega \subset \bR^d$ and $(0, T) \times \Omega \subset \bR^{d+1}$, we set 
\begin{itemize}
\item $\cS = \cS(\bR^d)$: the space of Schwartz functions.

\item $\cS' = \cS'(\bR^d)$: the space of tempered distributions.

\item $C_0^{\infty}(\Omega)$: the collection of all infinitely differentiable functions with compact support in $\Omega$.

\item $C_0^{\infty}\left( \left( 0, T\right) \times \Omega \right)$: the collection of all infinitely differentiable functions with compact support in $\Omega_T$.

\item $C_0^\infty([0, T] \times \Omega )$: the collection of all infinitely differentiable functions defined on $[0, T] \times \Omega$ having support in an intersection of a ball in $\bR^{d+1}$ with $[0, T] \times \Omega$.

\end{itemize}
We similarly define, for instance, $C_0^{\infty}\left( [0, T ) \times \Omega \right)$ and $C_0^{\infty}( \overline{ \left( 0, T\right) \times \Omega} )$.
By $\cF$ and $\cF^{-1}$, we denote the Fourier and the inverse Fourier transform in $\bR^d$;
\[
\cF f  := \frac{1}{(2\pi)^{d/2}} \int_{\bR^d} e^{- i x\cdot \xi}f(x)\,\mathrm{d}x, \quad \cF^{-1} f := \frac{1}{(2\pi)^{d/2}} \int_{\bR^d} e^{i x\cdot \xi}f(x)\,\mathrm{d}x,
\]
where we sometimes use $\hat{f}$ instead of $\cF f$.
For $a, b \in \bR$, we set $a \wedge b := \min\{a, b\}$ and $a \vee b := \max\{a, b \}$.
If we write $N = N(A, B, \ldots)$, this means that the constant $N$ depends only on $A, B, \ldots$.
Finally, for $a, b \in \bR$, we write $a \simeq_{\nu, \ldots} b$ if there exists a constant $N = N(\nu, \ldots) > 0$ such that $N^{-1}a \le b \le Na$.

\section{Function spaces with fractional regularity in time}
\label{ftn_space}

This section presents preliminary work on defining Sobolev-type solution spaces for time-fractional evolution equations to be discussed in Section \ref{sol_space}. One of the solution spaces defined in Section \ref{sol_space} is:
\[
\bH_p^{\alpha,2} = \{u \in L_{p}: u \in \bH_{p}^{\alpha,0}, Du, D^2 u \in L_{p} \},
\]
where the norm is given by
\[
\| u \|_{\bH_{p}^{\alpha,2}} = \left\| \left|u\right| + \left|Du\right| + \left|D^2u\right| + |\partial_t^{\alpha} u | \right\|_{L_{p}}. 
\]
To construct such spaces, in this section, we focus on studying the fractional derivative $\partial_t^\alpha$ in the context of $L_p$ spaces and examining its properties.
Here is a summary of the contents in this section:

\begin{itemize}
\item We start by defining function spaces with fractional regularity in time: $\bH_{p, 0}^{\alpha,0}$ and $\bH_{p}^{\alpha,0}$ for the non-divergence case (see Definition \ref{def0513_1}), $\cH_{p, 0}^{\alpha, -1}$ and $\cH_{p}^{\alpha, -1}$ for the divergence case (see Definition \ref{def0513_6}).

\item Next, we investigate properties related to the initial values of $\bH_{p, 0}^{\alpha,0}$ and $\bH_{p}^{\alpha,0}$ (respectively, $\cH_{p, 0}^{\alpha, -1}$ and $\cH_{p}^{\alpha, -1}$) that depend on the relationship between $\alpha$ and $p$. These properties are discussed in Lemmas \ref{lem0513_1} and \ref{lem0513_2} (respectively, Lemmas \ref{lem04131517_1} and \ref{lem0526_1}).

\item Finally, we define $\partial_t^{\alpha}u$ and the norms of $\bH_{p}^{\alpha,0}$ and $\cH_{p}^{\alpha, -1}$ using the construction outlined above (see Definitions \ref{def_0223_1} and \ref{def_0223_6}).

\item Additionally, we present essential properties of the fractional derivative $\partial_t^\alpha$ in the remaining lemmas, propositions, and remarks throughout this section.
\end{itemize}

For $\alpha \in (0, 1)$, we denote
\[
I^\alpha f(t,x) = \frac{1}{\Gamma(\alpha)} \int_0^t (t-s)^{\alpha-1} f(s,x) \, \mathrm{d}s, \quad I^1f(t, x) = \int_0^t f(s, x)\,\mathrm{d}s,
\]
and $I^n f = I^1\left( I^{n-1}f \right) $ for $n = 2, 3, \cdots$. We set
\begin{equation}
							\label{eq0516_04}
J^\alpha \varphi(t,x) = \frac{1}{\Gamma(\alpha)} \int_t^T (r-t)^{\alpha-1} \varphi(r,x) \, \mathrm{d}r.
\end{equation}

For $p \in (1,  \infty)$ and a positive integer $d$, we let $A_p(\bR^d, \mathrm{d}x) = A_p(\bR^d)$ be the set of all locally integrable non-negative functions $w$ on $\bR^d$ such that
$$
[ w ]_{A_p} := \sup_{x_0 \in \bR^d, r > 0} \left(\dashint_{B_r(x_0)}w(x)\, \mathrm{d} x \right)
\left(\dashint_{B_r(x_0)}\left( w \left( x \right) \right)^{-1/(p-1)}\,\mathrm{d} x \right)^{p-1}<\infty,
$$
where $B_r(x_0) = \{x \in \bR^d : |x - x_0| < r \}$.

Let $p, q \in (1, \infty)$, $\Omega \subset \bR^d$ be a domain and $\Omega_T := (0, T) \times \Omega$ for $T \in (0, \infty)$.
Note that throughout the paper, we assume that $T$ is a positive real number unless otherwise specified (for instance, $T$ can be $\infty$ in Lemma \ref{lem0224}).
For $w(t, x) = w_1(t)w_2(x) = t^{\mu}w_2(x)$, where $(t, x) \in \bR \times \bR^d$, $\mu \in (-1, q-1)$ and $w_2 \in A_p(\bR^d)$, we set $L_{p, q, w}(\Omega_T)$ to be the set of all measurable functions $f$ defined on $\Omega_T$ such that
$$
\|f\|_{L_{p, q, w}}(\Omega_T) := \left( \int_0^T \left( \int_{\Omega} |f(t, x)|^p w_2(x)\,\mathrm{d}x \right)^{q/p} w_1(t)\,\mathrm{d}t \right)^{1/q} < \infty.
$$
We set $L_{p, q, w}(\Omega_T) = L_p(\Omega_T)$ if $p = q$ and $w = 1$, i.e., $\mu = 0$ and $w_2(x)=1$.
We similarly define $L_{p, w_2}(\Omega)$ and $L_{q, w_1}(0, T)$.
One can use $w_2 \in A_p(\Omega)$ instead of $w_2 \in A_p(\bR^d)$ if the domain $\Omega$ is a space of homogeneous type. See \cite[Section 2]{MR3812104} for more details.

Before presenting function spaces incorporating fractional derivatives, we introduce a lemma on the inequality of the Hardy type, which serves as a useful tool throughout the paper.

\begin{lemma}
\label{lem0224}
Let $\alpha \in (0, 1]$, $q \in (1, \infty)$, and $\mu < q-1$. Then for any function $f$ defined on $(0, T)$ with $T \in (0, \infty]$, we have
\begin{equation}
\label{eq0224_03}
\int_0^{T} | t^{-\alpha}I^{\alpha}f |^q t^{\mu}\,\mathrm{d}t \le N  \int_0^{T} |f |^q t^{\mu}\,\mathrm{d}t,
\end{equation}
where $N = N(\alpha, q, \mu) > 0$.
\end{lemma}

\begin{proof}
Note that when $\alpha = 1$, \eqref{eq0224_03} is a version of classical Hardy's inequality. If $\alpha \in (0, 1)$, it is a direct consequence of \cite[Theorem 10]{MR1544927} by replacing $f$ with $f1_{(0, T)}$ in \cite[Theorem 10]{MR1544927}. 
\end{proof}

Now we introduce function spaces with fractional derivatives.

\subsection{Non-divergence case}
	\label{subsec_nondiv}

We present a set of definitions and properties for time fractional derivatives that are associated with time-fractional evolution equations in non-divergence form.

\begin{definition}
							\label{def0513_1}
Let $\alpha \in (0, 1)$, $p, q \in (1, \infty)$ and $w(t, x) = w_1(t)w_2(x) = t^{\mu}w_2(x)$, where $\mu \in (-1, q-1)$ and $w_2 \in A_p(\bR^d)$.
We define $\bH_{p, q, w, 0}^{\alpha,0}(\Omega_T)$ and $\bH_{p, q, w}^{\alpha,0}(\Omega_T)$ as follows.

\begin{enumerate}[(i)]
\item By $u \in \bH_{p, q, w, 0}^{\alpha,0}(\Omega_T)$, we mean that $u \in L_{p, q, w}(\Omega_T)$ and there exists $f \in L_{p, q, w}(\Omega_T)$ such that
\begin{equation}
							\label{eq0513_01}
\int_{\Omega_T}I^{1-\alpha} u \, \varphi_t \, \mathrm{d}x \, \mathrm{d}t = - \int_{\Omega_T} f \, \varphi \, \mathrm{d}x \, \mathrm{d}t
\end{equation}
for all $\varphi \in C_0^\infty\left([0,T) \times \Omega\right)$. 
Clearly, $\partial_t I^{1-\alpha} u = f$.
When $p=q$ and $w=1$, we set $\bH_{p, q, w, 0}^{\alpha,0}(\Omega_T) = \bH_{p, 0}^{\alpha,0}(\Omega_T)$.

\item By $u \in \bH_{p, q, w}^{\alpha,0}(\Omega_T)$, we mean that $u \in L_{p, q, w}(\Omega_T)$ and there exists $u_0 \in L_{p, w_2}(\Omega)$ such that
\[
u-u_0 \in \bH_{p, q, w, 0}^{\alpha,0}(\Omega_T).
\]
We again suppress $q$ and $w$ if $p=q$ and $w=1$ (i.e., $\mu = 0$ with $w_2(x) = 1$).
\end{enumerate}
In our notation, the superscript $0$ in $\bH_{p, q, w, 0}^{\alpha,0}(\Omega_T)$ and $\bH_{p, q, w}^{\alpha,0}(\Omega_T)$ indicates that no regularity of $u$ is required with resect to the spatial variables.
On the other hand, the subscript $0$ in $\bH_{p, q, w, 0}^{\alpha,0}(\Omega_T)$ signifies that the initial value of each element of $\bH_{p, q, w, 0}^{\alpha,0}(\Omega_T)$ is zero, whenever it is well-defined (see Lemma \ref{lem0513_2}, Theorem \ref{thm0111_01}, and Theorem\ref{cor0111_01}).
\end{definition}

The norms for $\bH_{p, q, w, 0}^{\alpha,0}(\Omega_T)$ and $\bH_{p, q, w}^{\alpha,0}(\Omega_T)$ will be introduced in Definition \ref{def_0223_1}.

\begin{lemma}
							\label{lem0513_1}
Let $\alpha \in (0,(1+\mu)/q)$ and $u_0 \in L_{p, w_2}(\Omega)$.
Then, $u_0 \in \bH_{p, q, w, 0}^{\alpha,0}(\Omega_T)$ as a function on $\Omega_T$.
Thus, $\bH_{p, q, w, 0}^{\alpha,0}(\Omega_T) = \bH_{p, q, w}^{\alpha,0}(\Omega_T)$ for $\alpha \in (0,(1+\mu)/q)$.
\end{lemma}

\begin{proof}
Clearly,
\[
I^{1-\alpha} u_0 = \frac{t^{1-\alpha}}{(1-\alpha)\Gamma(1-\alpha)}u_0(x) \quad \textrm{and} \quad \partial_t I^{1-\alpha} u_0 = \frac{t^{-\alpha}}{\Gamma(1-\alpha)}u_0(x),
\]
and we see that $\partial_t I^{1-\alpha} u_0 \in L_{p, q, w}(\Omega_T)$ because $\alpha q < 1 + \mu$.
By integration by parts and the fact that $I^{1-\alpha} u_0|_{t=0} = 0$, we have
\[
\int_{\Omega_T} I^{1-\alpha} u_0 \, \varphi_t \, \mathrm{d}x \, \mathrm{d}t = - \int_{\Omega_T} \partial_t I^{1-\alpha} u_0 \, \varphi \, \mathrm{d}x \, \mathrm{d}t
\]
for all $\varphi \in C_0^\infty\left([0,T) \times \Omega\right)$.
Therefore, by Definition \ref{def0513_1}, $u_0 \in \bH_{p, q, w, 0}^{\alpha,0}(\Omega_T)$.

If $u \in \bH_{p, q, w}^{\alpha,0}(\Omega_T)$, there exists $u_0 \in L_{p, w_2}(\Omega)$ such that $u - u_0 \in \bH_{p, q, w, 0}^{\alpha,0}(\Omega_T)$.
Since we also have $u_0 \in \bH_{p, q, w, 0}^{\alpha,0}(\Omega_T)$, it follows that
\[
u = (u-u_0)+u_0 \in \bH_{p, q, w, 0}^{\alpha,0}(\Omega_T).
\]
The lemma is proved. 
\end{proof}

\begin{lemma}
							\label{lem0513_2}
Let $\alpha \in [(1+\mu)/q,1)$ and $u_0 \in L_{p, w_2}(\Omega)$.
If $u_0 \in \bH_{p, q, w, 0}^{\alpha,0}(\Omega_T)$, then $u_0 = 0$.
Thus, for $u \in \bH_{p, q, w}^{\alpha,0}(\Omega_T)$, there exists a unique $u_0 \in L_{p, w_2}(\Omega)$ such that $u-u_0 \in \bH_{p, q, w, 0}^{\alpha,0}(\Omega_T)$ when $\alpha \in [(1+\mu)/q,1)$.
\end{lemma}

\begin{proof}
If $u_0 \in \bH_{p, q, w, 0}^{\alpha,0}(\Omega_T)$, by Definition \ref{def0513_1}, the derivative $\partial_t I^{1-\alpha} u_0$ exists in $L_{p, q, w}(\Omega_T)$.
On the other hand,
\[
\partial_t I^{1-\alpha} u_0 = \frac{t^{-\alpha}}{\Gamma(1-\alpha)}u_0(x)
\]
in the distribution sense.
Because $t^{-\alpha}w_1(t)^{1/q} = t^{-\alpha + \mu/q} \notin L_q(0,T)$ (recall that $w(t, x) = w_1(t)w_2(x) = t^{\mu}w_2(x)$), we have $\partial_t I^{1-\alpha}u_0 \in L_{p, q, w}(\Omega_T)$ only when $u_0 = 0$.

To prove the uniqueness of $u_0$, for $u \in \bH_{p, q, w}^{\alpha,0}(\Omega_T)$, let $u_0, v_0 \in L_{p, w_2}(\Omega)$ satisfy
\[
u-u_0, \, u-v_0 \in \bH_{p, q, w, 0}^{\alpha,0}(\Omega_T).
\]
Then, $u_0 - v_0 = (u-u_0) - (u-v_0) \in \bH_{p, q, w, 0}^{\alpha,0}(\Omega_T)$.
From this, along with the fact that $u_0 - v_0 \in L_{p, w_2}(\Omega)$ and the first assertion of the lemma proved above, we get $u_0 = v_0$.
\end{proof}

Thanks to the above lemmas we are ready to define $\partial_t^\alpha u$ and the norm of $\bH_{p, q, w}^{\alpha,0}(\Omega_T)$.

\begin{definition}
\label{def_0223_1}
Let $p, q \in (1, \infty)$ and $w(t, x) = w_1(t)w_2(x) = t^{\mu}w_2(x)$, where $\mu \in (-1, q-1)$ and $w_2 \in A_p(\bR^d)$.
For $\alpha \in (0,(1+\mu)/q)$ and $u \in \bH_{p, q, w}^{\alpha,0}(\Omega_T)$,
we denote
\[
\partial_t^\alpha u = \partial_t I^{1-\alpha} u.
\]
For $\alpha \in [(1+\mu)/q,1)$ and $u \in \bH_{p, q, w}^{\alpha,0}(\Omega_T)$, by Lemma \ref{lem0513_2}, there exists a unique $u_0 \in L_{p, w_2}(\Omega_T)$ such that $u - u_0 \in \bH_{p, q, w, 0}^{\alpha,0}(\Omega_T)$.
In this case, we denote
\[
\partial_t^\alpha u = \partial_t I^{1-\alpha}(u-u_0).
\]
Then, the norm of $\bH_{p, q, w}^{\alpha,0}(\Omega_T)$ for $\alpha \in (0,1)$ is defined by
\[
\|u\|_{\bH_{p, q, w}^{\alpha,0}(\Omega_T)} = \|u\|_{L_{p, q, w}(\Omega_T)} + \|\partial_t^\alpha u\|_{L_{p, q, w}(\Omega_T)}.
\]
\end{definition}

For the case $p = q = 2$ with $\mu = 0$, a similar notion of $\partial_t^{\alpha}u$ to Definition \ref{def_0223_1} can be found in \cite{MR2538276} and \cite{MR4200127}.

\begin{remark}
	\label{rmk_alpha}
\mbox{}
	
(i) Note that $\bH_{p, q, w, 0}^{\alpha,0}(\Omega_T)$ is a subspace $\bH_{p, q, w}^{\alpha,0}(\Omega_T)$ with $u_0 = 0$.
Thus, for $u \in \bH_{p, q, w, 0}^{\alpha,0}(\Omega_T)$, we have
\[
\partial_t^\alpha u = \partial_t I^{1-\alpha}(u-0) = \partial_t I^{1-\alpha} u.
\]
Regarding the norm for $\bH_{p, q, w, 0}^{\alpha,0}(\Omega_T)$, we use the same norm $\|\cdot\|_{\bH_{p, q, w}^{\alpha,0}(\Omega_T)}$ (and the same notation).

(ii) In Definition \ref{def_0223_1}, unlike the case where $\alpha \in [(1+\mu)/q,1)$, $u_0$ is not involved in defining $\partial_t^\alpha u$ for $\alpha \in (0, (1+\mu)/q)$.
This distinction arises because the initial value $u_0$ cannot be properly defined when $\alpha \in (0, (1+\mu)/q)$.
See Lemma \ref{lem0513_1}.
One can observe such inadequacy in the following example for the case when $p = q$ and $w(t, x) = t^{\mu}w_2(x)=1$.
For $\phi(x) (\neq 0) \in C_0^{\infty}(\Omega)$, set $u_n(t, x) = \phi(x)$, and $v_n(t, x) = t^{1/n}\phi(x)$ for $n = 1, 2, \cdots$.
By direct computation, we see that $\partial_t^{\alpha}u_n -  \partial_t^{\alpha}u_m = 0$ for all $n, m = 1, 2, \cdots$, and $\partial_t^{\alpha}v_n - \partial_t^{\alpha}v_m \to 0$ as $n, m \to \infty$, provided that $\alpha < 1/p$.
But, $u_n(0, x) = \phi(x) \neq 0 = v_n(0, x)$ for all $n = 1, 2, \cdots$.
On the other hand, in the case of $\alpha \in [(1+\mu)/q, 1)$, Lemma \ref{lem0513_2} guarantees the well-definedness of the initial trace, at least in a framework of Definition \ref{def0513_1}.
See Theorems \ref{thm0111_01} and \ref{cor0111_01} for initial trace results concerning functions with both fractional time derivatives and spatial derivatives (precisely, those in $\bH_{p, q, w}^{k+\alpha, 2}(\Omega_T)$ and $\cH_{p, q, w}^{k+\alpha, 1}(\Omega_T)$).
\end{remark}

The following lemma is used in the proof of Lemma \ref{lem0517_2}.
Recall the definition of $J^{\alpha}\varphi$ in \eqref{eq0516_04}.

\begin{lemma}
							\label{lem0517_1}
For $\varphi \in C_0^\infty\left( \left( 0,T \right) \times \Omega \right)$, we have $J^\alpha \varphi \in C_0^\infty([0,T) \times \Omega)$ and
\[
\partial_t J^\alpha \varphi = J^\alpha \partial_t \varphi.
\]
\end{lemma}

\begin{proof}
To check $\partial_t J^\alpha \varphi = J^\alpha \partial_t \varphi$, for $\phi \in C_0^\infty(0,T)$, we consider
\[
\int_0^T \phi'(t) J^\alpha \varphi \, \mathrm{d}t = \int_0^T I^\alpha \phi'(t) \varphi \, \mathrm{d}t = \int_0^T \partial_t I^\alpha \phi (t) \varphi \, \mathrm{d}t,
\]
where we know that $\partial_t I^\alpha \phi = I^\alpha \phi'$ because $\phi(0) = 0$.
Thus,
\[
\int_0^T \phi'(t) J^\alpha \varphi \, \mathrm{d}t = \int_0^T \partial_t I^\alpha \phi (t) \varphi \, \mathrm{d}t = - \int_0^T I^\alpha \phi(t) \,  \partial_t \varphi \, \mathrm{d}t = - \int_0^T \phi(t) J^\alpha \partial_t \varphi \, \mathrm{d}t,
\]
where we used the fact that $I^\alpha \phi|_{t=0} = 0$ and $\varphi(T,x) = 0$.
This shows that $\partial_t J^\alpha \varphi = J^\alpha \partial_t \varphi$.
\end{proof}

\begin{lemma}
							\label{lem0517_2}
Let $\alpha \in (0,1)$ and $u \in \bH_{p, q, w, 0}^{\alpha,0}(\Omega_T)$.
Then, $u = I^\alpha \partial_t I^{1-\alpha} u$
and
\[
\|u\|_{L_{p, q, w}(\Omega_T)} \leq NT^\alpha \|\partial_t^\alpha u\|_{L_{p, q, w}(\Omega_T)},
\]
where $N = N(\alpha, q, \mu) > 0$.
Thus, if $\alpha \in [(1+\mu)/q,1)$ and $u \in \bH_{p, q, w}^{\alpha,0}(\Omega_T)$ with $u - u_0 \in \bH_{p, q, w, 0}^{\alpha,0}(\Omega_T)$,
then
\begin{equation}
							\label{eq0513_11}
\|u\|_{L_{p, q, w}(\Omega_T)} \leq N T^\alpha \|\partial_t^\alpha u\|_{L_{p, q, w}(\Omega_T)} + N T^{(1+\mu)/q} \|u_0\|_{L_{p, w_2}(\Omega)},
\end{equation}
where $N = N(\alpha,q, \mu) > 0$.
\end{lemma}

\begin{proof}
For $\varphi \in C_0^\infty((0,T) \times \Omega)$, by Lemma \ref{lem0517_1} $J^\alpha \varphi(t,x) \in C_0^\infty([0,T) \times \Omega)$ and $\partial_t  J^\alpha \varphi = J^\alpha \partial_t \varphi$.
Hence, $u= I^\alpha \partial_t I^{1-\alpha} u$ holds because of the following calculation.
\[
\int_{\Omega_T} I^\alpha \partial_t I^{1-\alpha} u \, \varphi \, \mathrm{d}x \, \mathrm{d}t = \int_{\Omega_T} \partial_t I^{1-\alpha} u J^{\alpha} \varphi \, \mathrm{d}x \, \mathrm{d}t
= - \int_{\Omega_T} I^{1-\alpha} u \, \partial_t J^{\alpha} \varphi \, \mathrm{d}x \, \mathrm{d}t
\]
\[
 = - \int_{\Omega_T} I^{1-\alpha} u J^{\alpha} \partial_t  \varphi \, \mathrm{d}x \, \mathrm{d}t
= - \int_{\Omega_T} I^\alpha I^{1-\alpha} u \, \partial_t  \varphi \, \mathrm{d}x \, \mathrm{d}t = \int_{\Omega_T} u \, \varphi \, \mathrm{d}x \, \mathrm{d}t,
\]
which we essentially used that \eqref{eq0513_01} holds for $J^\alpha \varphi \in C_0^\infty([0,T) \times \Omega)$.

By the fact that $u= I^\alpha \partial_t I^{1-\alpha}u$ and \eqref{eq0224_03} of Lemma \ref{lem0224},
\[
\|u\|_{L_{p, q, w}(\Omega_T)} = \|I^\alpha \partial_t I^{1-\alpha} u\|_{L_{p, q, w}(\Omega_T)} \leq N T^\alpha \|\partial_t I^{1-\alpha} u\|_{L_{p, q, w}(\Omega_T)}.
\]

To prove \eqref{eq0513_11}, we simply write
\begin{equation*}
u = u-u_0 + u_0 = I^\alpha \partial_t I^{1-\alpha} (u-u_0) + u_0 = I^\alpha \partial_t^\alpha u + u_0.
\end{equation*}
The lemma is proved.
\end{proof}

\begin{lemma}
							\label{lem0225_1}
Let $\alpha \in (0,1)$ and $u \in \bH_{p, q, w, 0}^{\alpha,0}(\Omega_T)$.
Then
\[
\| t^{-\alpha} u \|_{L_{p, q, w}(\Omega_T)} \le N \| \partial_t^{\alpha}u \|_{L_{p, q, w}(\Omega_T)},
\]
where $N = N(\alpha, q, \mu) > 0$.
Furthermore, if $\alpha \in [(1+\mu)/q, 1)$ and $u \in \bH_{p, q, w}^{\alpha,0}(\Omega_T)$ with $u - u_0 \in \bH_{p, q, w, 0}^{\alpha,0}(\Omega_T)$, we have
\[
\| t^{-\alpha} (u - u_0) \|_{L_{p, q, w}(\Omega_T)} \le N(\alpha, q, \mu) \| \partial_t^{\alpha}u \|_{L_{p, q, w}(\Omega_T)}.
\]
\end{lemma}

\begin{proof}
By definition, the second assertion follows from the first one. Let $\alpha \in (0,1)$, $u \in \bH_{p, q, w, 0}^{\alpha,0}(\Omega_T)$ and denote $f = \partial_t I^{1-\alpha}u = \partial_t^\alpha u$.
Then by Lemma \ref{lem0517_2} $u = I^\alpha f$.
Since $f \in L_{p, q, w}(\Omega_T)$, by Minkowski inequality and \eqref{eq0224_03} of Lemma \ref{lem0224} we have
\[
\int_0^T \left(  \int_{\Omega} | I^\alpha f(t,x)|^p w_2(x)\,\mathrm{d}x \right)^{q/p} t^{-q\alpha + \mu}\, \mathrm{d}t \leq \int_0^T | I^{\alpha} \|f(t, \cdot)\|_{L_{p, w_2}(\Omega)} |^q t^{-q\alpha + \mu} \,\mathrm{d}t
\]
\[
\le N \int_0^T  \|f(t, \cdot)\|_{L_{p, w_2}(\Omega)}^q t^{\mu}\, \mathrm{d}t,
\]
where $N = N(\alpha, q, \mu)$.
We obtain the desired inequality by recalling that $I^\alpha f = u$.
\end{proof}

\begin{remark}
	\label{rmk_yamamoto} 
We here comment on the equivalence between $\bH_{p}^{\alpha, 0}$ and the function spaces with fractional regularity in time that appear in literature \cite{MR4200127} and \cite{MR2125407, MR2278673}.

(i) Lemma \ref{lem0517_2} shows that, for $u \in \bH_{p, q, w, 0}^{\alpha,0}(\Omega_T)$,
\[
\|\partial_t^\alpha u\|_{L_{p, q, w}(\Omega)} \simeq \|u\|_{\bH_{p, q, w}^{\alpha,0}(\Omega_T)}.
\]
Let us use the notation $\bH_{p,0}^{\alpha,0}(0,T)$ for functions without spatial variables with $w_1(t) = 1$.
Then, the above equivalence indicates that $\bH_{p,0}^{\alpha,0}(0,T)$, when $p=2$, is the same as $H_\alpha(0,T)$  defined in \cite{MR4200127}, where the authors state that
\[
\|\partial_t^\alpha u\|_{L_2(0,T)} \simeq \|u\|_{H_\alpha(0,T)}.
\]
Indeed, in \cite{MR4200127} the authors defines $\partial_t^\alpha$ as an inverse operator of the operator $I^\alpha$ from $L_2(0,T) \to H_\alpha(0,T)$.
Thus, for any $u \in \bH_{p,0}^{\alpha,0}(0,T)$, we have $\partial_t^\alpha u \in L_2(0,T)$ and by Lemma \ref{lem0517_2}, $I^\alpha \partial_t^\alpha u = u$, which means that $u \in H_\alpha(0,T)$.
On the other hand, if $u \in H_\alpha(0,T)$, then there exists $f \in L_2(0,T)$ such that $u = I^\alpha f$.
Then, $\partial_t^\alpha u = \partial_t^\alpha 
I^\alpha f = \partial_t I^{1-\alpha} I^\alpha f = f$.
We also see that \eqref{eq0513_01} is satisfied. Hence, by Definition \ref{def0513_1} $u \in \bH_{p,0}^{\alpha,0}(0,T)$.
Also note that Lemma \ref{lem0225_1} shows that, when $\alpha = 1/2$,
\[
\int_0^T \frac{|u(t)|^2}{t} \, \mathrm{d}t < \infty
\]
for $u \in \bH_{2,0}^{1/2,0}(\Omega_T)$,
which is required for $u \in H_{1/2}(0,T)$ in \cite{MR4200127}.

(ii) As noted in \cite[Example 2.1]{MR2278673}, if $p=q$, $w(t, x)= 1$ and $\alpha \neq 1/p$, $\bH_{p}^{\alpha, 0}(\Omega_T)$ is equivalent to
\[
\bar{H}_p^{\alpha}\left(  \left(0, T \right), L_p\left(  \Omega \right)\right) := \{ u|_{(0, T)} : u \in \bar{H}_p^{\alpha}\left(  \bR, L_p\left(  \Omega \right)\right) \},
\]
where $\bar{H}_p^{\alpha}\left(  \bR, L_p\left(  \Omega \right)\right)$ is an $L_p(\Omega)$-valued Bessel potential space.
The norm is given by
\[
\| u \|_{\bar{H}_p^{\alpha}\left(  \left(0, T \right), L_p\left(  \Omega \right)\right)} = \inf \{ \|v\|_{\bar{H}_p^{\alpha}\left(  \bR, L_p\left(  \Omega \right)\right)} : v|_{(0, T)} = u \}.
\]
The author in \cite{MR2278673} implicitly demonstrated the equivalence between $\bH_{p}^{\alpha, 0}(\Omega_T)$ and $\bar{H}_p^{\alpha}\left(  \left(0, T \right), L_p(\Omega)\right)$ by using results from operator analysis and the Sobolev embedding theorem.
In fact, one can obtain the above equivalence by using \cite[(7.4)]{MR1347689} with an extension of $u \in \bH_p^{\alpha, 0}(\Omega_T)$ to $\bR \times \Omega$.
For the case $\alpha = 1/p$, we can show that $\bH_{p}^{\alpha, 0}(\Omega_T)$ is equivalent to a proper subset of $\bar{H}_p^{\alpha}\left( (0, T), L_p(\Omega)\right)$.
At least, Lemma \ref{lem0525_1} below implies that they are not equivalent because one cannot define the initial trace of functions in $\bar{H}_p^{\alpha}\left((0, T), L_p(\Omega)\right)$.
See Remark \ref{rmk_230615}-(ii) for more discussion about the case $\alpha = 1/p$.
\end{remark}

\begin{proposition}
							\label{prop0514_1}
For $u \in \bH_{p, q, w, 0}^{\alpha,0}(\Omega_T)$, there is a sequence $\{u_k\}$ such that $u_k \in C^\infty\left([0,T] \times \Omega\right)$, $u_k(t,x)$ vanishes for large $|x|$ (if $\Omega$ is unbounded), $u_k(0,x) = 0$, and $u_k \to u$ in $\bH_{p, q, w, 0}^{\alpha,0}(\Omega_T)$ as $k \to \infty$.
\end{proposition}

\begin{proof}
We know $I^{1-\alpha}u \in L_{p, q, w}(\Omega_T)$.
By Definition \ref{def0513_1}, $\partial_t I^{1-\alpha} u \in L_{p, q, w}(\Omega_T)$.
Thus, $I^{1-\alpha} u \in W_{p, q, w}^{1,0}(\Omega_T)$.
Moreover, the equality \eqref{eq0513_01}  implies that $I^{1-\alpha} u(t,x)|_{t=0} = 0$.
Then, there exists a sequence $\{w_k\}$ such that $w_k \in C^\infty\left([0,T] \times \Omega\right)$, $w_k(t,x)$ vanishes for large $|x|$, $w_k(0,x) = 0$, and $w_k \to I^{1-\alpha}u$ in $W_{p, q, w}^{1,0}(\Omega_T)$ as $k \to \infty$.
In particular,
\[
w_k \to I^{1-\alpha} u, \quad 
\partial_t w_k \to \partial_t I^{1-\alpha} u
\]
in $L_p(\Omega_T)$.
Set
\[
u_k(t,x) = I^\alpha (\partial_t w_k)(t,x) = \frac{1}{\Gamma(\alpha)} \int_0^t (t-s)^{\alpha-1} \partial_t w_k(s,x) \, \mathrm{d}s.
\]
Then, $u_k \in C^\infty([0,T] \times \Omega)$, $u_k(t,x)$ vanishes for large $|x|$, and $u_k(0,x) = 0$.
Since $\partial_t w_k \to \partial_t I^{1-\alpha} u$ in $L_{p, q, w}(\Omega_T)$, we have $u_k = I^\alpha (\partial_t w_k) \to I^{\alpha} \partial_t I^{1-\alpha} u$ in $L_{p, q, w}(\Omega_T)$, where $I^{\alpha} \partial_t I^{1-\alpha} u = u$ by Lemma \ref{lem0517_2}.
We also have
\[
\partial_t I^{1-\alpha} u_k = \partial_t I^{1-\alpha} I^\alpha (\partial_t w_k) = \partial_t w_k \to 
\partial_t I^{1-\alpha} u
\]
in $L_p(\Omega_T)$.
Therefore, $\{u_k\}$ is a desired sequence.
\end{proof}

The following Lemmas \ref{lem0525_1} and \ref{lem1115_1} contain the behavior of $u \in \bH_{p, q, w}^{\alpha,0}(\Omega_T)$ at $t = 0$.
The complete description of initial traces for functions in $\bH_{p, q, w}^{\alpha,0}(\Omega_T)$ with additional temporal and spatial Sobolev regularities (that is, functions in $\bH_{p, q, w}^{k+\alpha,2}(\Omega_T)$) is provided in Theorem \ref{thm0111_01}.

\begin{lemma}[Trace inequality]
							\label{lem0525_1}
Let $\alpha \in [(1+ \mu)/q,1)$ and $u \in \bH_{p, q, w}^{\alpha,0}(\Omega_T)$ with $u - u_0 \in \bH_{p, q, w, 0}^{\alpha,0}(\Omega_T)$.
Then,
\begin{equation}
							\label{eq0513_03}
\|u_0\|_{L_{p, w_2}(\Omega)} \leq N \left( T^{-(1+\mu)/q}\|u\|_{L_{p, q, w}(\Omega_T)} + T^{\alpha - (1+\mu)/q} \| \partial_t^{\alpha}u\|_{L_{p, q, w}(\Omega_T)}\right),
\end{equation}
where $N = N(\alpha, q, \mu) > 0$. Furthermore, if $T \ge 1$, we have
$$
\|u_0\|_{L_{p, w_2}(\Omega)} \leq N(\alpha, q, \mu) \|u\|_{\bH_{p, q, w}^{\alpha, 0}(\Omega_T)}.
$$
\end{lemma}

\begin{proof}
Denote $f := \partial_t I^{1-\alpha}(u-u_0) = \partial_t^\alpha u$.
Observe that
\[
I^{1-\alpha} u_0 = I^{1-\alpha} u - I^{1-\alpha} (u-u_0) = I^{1-\alpha} u - \int_0^t f(s,x) \, \mathrm{d}s,
\]
where the second equality is due to the fact that \eqref{eq0513_01} holds for $\varphi \in C_0^\infty\left([0,T) \times \Omega\right)$ with $u$ replaced with $u-u_0$.
Since
\[
I^{1-\alpha} u_0 = \frac{t^{1-\alpha}}{(1-\alpha)\Gamma(1-\alpha)} u_0(x),
\]
we have
\[
\frac{t^{1-\alpha}}{(1-\alpha)\Gamma(1-\alpha)}u_0(x) = I^{1-\alpha}u - \int_0^t f(s,x) \, \mathrm{d}s.
\]
By taking the $L_{p, q, w}(\Omega_T)$-norms on both sides of the above equality, (when $T \geq 1$, we take $L_{p, q, w}((0,1) \times \Omega)$-norms), we have
\[
T^{(1-\alpha) + (1+\mu)/q} \|u_0\|_{L_{p, w_2}(\Omega)} \leq N \left( \|I^{1-\alpha} u\|_{L_{p, q, w}(\Omega_T)} + \left\| \int_0^t f(s,x) \, \mathrm{d}s \right\|_{L_{p, q, w}(\Omega_T)} \right)
\]
\[
\leq N(\alpha,q, \mu) \left( T^{1-\alpha}\|u\|_{L_{p, q, w}(\Omega_T)} +  T \|f\|_{L_{p, q, w}(\Omega_T)} \right),
\]
where the last inequality is due to \eqref{eq0224_03} of Lemma \ref{lem0224}.
Thus, the inequality \eqref{eq0513_03} is proved.
\end{proof}

\begin{lemma}
	\label{lem1115_1}
Let $\alpha \in ((1+\mu)/q,1)$ if $\mu \in [0, q-1)$ and let $\alpha \in (1/q, 1)$ if $\mu \in (-1, 0)$. For $u \in \bH_{p, q, w}^{\alpha,0}(\Omega_T)$ with $u_0 \in L_{p, w_2}(\Omega)$ and $u - u_0 \in \bH_{p, q, w, 0}^{\alpha,0}(\Omega_T)$, there is a version of $u$ (still denoted by $u$) such that
\begin{equation}
\label{eq0220_02_1}
\|u(t,\cdot) - u_0(\cdot)\|_{L_{p, w_2}(\Omega)} \leq N t^{\alpha-(1+\mu)/q} \|\partial_t^\alpha u\|_{L_{p, q, w}(\Omega_T)}
\end{equation}
for $t \in (0,T)$, where $N = N(\alpha,q, \mu) > 0$.
Thus, 
\begin{equation*}
\| u(t,\cdot) - u_0(\cdot)\|_{L_{p, w_2}(\Omega)} \to 0
\end{equation*}
as $t \searrow 0$.
Furthermore, we have
\begin{equation}
	\label{eq230606_1}
\|u(t, \cdot)\|_{L_{p, w_2}(\Omega)} \le N(\alpha, q, \mu, T) \| u \|_{\bH_{p, q, w}^{\alpha, 0}(\Omega_T)}.
\end{equation}
\end{lemma}

\begin{proof}
Note that \eqref{eq230606_1} follows from \eqref{eq0220_02_1} with the help of \eqref{eq0513_03} in Lemma \ref{lem0525_1} and the fact that $\alpha - (1+\mu)/q > 0$ under the assumptions on $\alpha, q$ and $\mu$.

Since $u - u_0 \in \bH_{p, q, w, 0}^{\alpha,0}(\Omega_T)$, by Lemma \ref{lem0517_2},
\[
u(t,x) - u_0(x) = I^\alpha \partial_t^\alpha u = \frac{1}{\Gamma(\alpha)}\int_0^t (t-s)^{\alpha-1} \partial_t^\alpha u(s,x) \, \mathrm{d}s
\]
for $(t, x) \in \Omega_T$.
By taking $L_{p, w_2}(\Omega)$-norms of both sides along with the Minkowski inequality and H\"{o}lder's inequality, we get
\[
\|u(t,\cdot) - u_0(\cdot)\|_{L_{p, w_2}(\Omega)} \leq \frac{1}{\Gamma(\alpha)}\int_0^t (t-s)^{\alpha-1} \|\partial_t^\alpha u(s,\cdot)\|_{L_{p, w_2}(\Omega)} \, \mathrm{d}s
\]
\[
\le N(\alpha) \left( \int_0^t (t-s)^{q'(\alpha - 1)} s^{-(q'-1)\mu}\,\mathrm{d}s \right)^{1/q'} \left( \int_0^t  \|\partial_t^\alpha u(s,\cdot)\|_{L_{p, w_2}(\Omega)}^q s^{\mu} \, \mathrm{d}s \right)^{1/q}
\]
\[
\leq N(\alpha,q, \mu) t^{\alpha - (1+\mu)/q} \left( \int_0^t \|\partial_t^\alpha u(s, \cdot)\|_{L_{p, w_2}(\Omega)}^{q}s^{\mu}\,\mathrm{d}s \right)^{1/q},
\]
where the last inequality is due to the assumptions on $\alpha$ and $\mu$, i.e. $\alpha > 1/q$ and $\mu < q-1$.
The lemma is proved.
\end{proof}

\begin{remark}
	\label{rmk_230613}
To see the necessity of the conditions on $\alpha$ in Lemma \ref{lem1115_1}, instead of \eqref{eq0220_02_1}, let us 
consider a weaker version
\begin{equation}
							\label{eq0705_01}
\sup_{0 < t < T}\|u(t,\cdot)\|_{L_{p,w_2}(\Omega)} \leq N \|\partial_t^\alpha u\|_{L_{p,q,w}(\Omega_T)},
\end{equation}
where we assume $u_0=0$ for simplicity.
Take $u(t,x) = \varphi(x) I^\alpha f(t)$, where $f$ and $\varphi$ are sufficiently smooth, so that $u(0,x) = 0$ and $\partial_t^\alpha u = \varphi
(x) f(t)$.
From the inequality \eqref{eq0705_01}, we have
\[
|I^\alpha f(t)| = N(\alpha) \left|\int_0^t (t-s)^{\alpha-1} f(s) \, \mathrm{d}s \right| \leq N \left(\int_0^T |f(s)|^q s^{\mu} \mathrm{d} s \right)^{1/q}
\]
for any $t \in (0,T)$, which means that, as a function of $s$, the $L_{q'}$-norm ($1/q+1/q'=1$) of $(t-s)^{\alpha-1} s^{-\mu/q}$ on the interval $(0,t)$ is uniformly bounded.
One can check that, by direct calculation, this is possible only when $\alpha$ satisfies the conditions (i.e., $\alpha > 1/q \vee (1 + \mu)/q$ with $\mu \in (-1, q-1)$) in Lemma \ref{lem1115_1}.
Especially, we need $\alpha > 1/q$ regardless of the value $\mu$.
The same conditions on $\alpha$ are imposed in Lemma \ref{lem230606_1} for the divergence case.
Nevertheless, it is worth noting that, by Lemma \ref{lem0513_2}, $u_0(x) = u(0, x)$ in Definition \ref{def_0223_1} is well-defined even under the condition $\alpha = (1+\mu)/q$ with $q \in (1, \infty)$ and $\mu \in (-1, q-1)$.
\end{remark}

\begin{remark}
	\label{rmk230604_1}
For later use, we note that if $u$ belongs to
\[
u \in W_{p, q, w}^{1, 0}(\Omega_T) := \{ u \in L_{p, q, w}(\Omega_T) : \partial_t u \in L_{p, q, w}(\Omega_T) \},
\]
which is a subset of $\bH_{p,q,w}^{\alpha,0}(\Omega_T)$, then there is a version of $u$ (still denoted by $u$) such that $u(t, \cdot)$ for $t \in [0, T]$ is well-defined in $L_{p, w_2}(\Omega)$ and satisfies
\begin{equation}
	\label{eq230604_1}
\|u(t, \cdot)\|_{L_{p, w_2}(\Omega)} \leq N(q, \mu) T^{-(1+\mu)/q} \left( \|u\|_{L_{p, q, w}(\Omega_T)} +  T \| \partial_t u\|_{L_{p, q, w}(\Omega_T)} \right).
\end{equation}
\end{remark}

\subsection{Divergence case}
	\label{subsec_div}
In this subsection, we introduce definitions and properties for time fractional derivatives associated with time-fractional evolution equations in divergence form.
The reader may skip this subsection since its content closely parallels that of Section \ref{subsec_nondiv}.

Let $p, q \in (1, \infty)$. For a given weight function $w(t,x) = t^{\mu}w_2(x)$ with $\mu \in (-1, q-1)$ and $w_2 \in A_p(\bR^d)$, we set
\[
w'(t, x) := w_1'(t)w_2'(x) = w_1^{1-q'}(t)w_2^{1-p'}(x) = t^{\mu(1-q')} w_2^{1-p'}(x)
\] 
with $1/p+1/p' = 1/q + 1/q' = 1$. Note that $\mu(1-q') \in (-1, q'-1)$ and $w_2^{1-p'} \in A_{p'}(\bR^d)$. 
We say $u \in \bH_{p, q, w}^{-1}(\Omega_T)$ if $u$ (as a distribution on $\Omega_T$) is a bounded linear functional on 
\[
\mathring{W}_{p',  q', w'}^{0,1}(\Omega_T) = \{v, Dv \in L_{p',  q', w'}(\Omega_T),v(t,x)|_{(0,T) \times \partial \Omega} = 0\}
\]
such that the evaluation of $u$ at $\varphi \in \mathring{W}_{p',  q', w'}^{0,1}(\Omega_T)$, denoted by $\langle u, \varphi \rangle_{\bH_{p, q, w}^{-1}(\Omega_T)}$, is given by   
\begin{equation}
							\label{eq0512_01}
 \langle u, \varphi \rangle_{\bH_{p, q, w}^{-1}(\Omega_T)} = \int_{\Omega_T} \left(F \varphi - G_i D_i \varphi \right) \, \mathrm{d}x \, \mathrm{d}t
\end{equation}
for some $G_i, F \in L_{p, q, w}(\Omega_T)$.
In this case we can write
\[
u = D_i G_i + F.
\]
If $u \in L_{p,q,w}(\Omega_T)$, then
\[
\langle u, \varphi \rangle_{\bH_{p, q, w}^{-1}(\Omega_T)} = \int_{\Omega_T} u \varphi  \, \mathrm{d}x \, \mathrm{d}t
\]
for $\varphi(t,x) \in \mathring{W}_{p',  q', w'}^{0,1}(\Omega_T)$, so that $L_{p, q, w}(\Omega_T) \subset \bH_{p, q, w}^{-1}(\Omega_T)$.

The norm of $\bH_{p, q, w}^{-1}(\Omega_T)$ is defined by
\[
\|u\|_{\bH_{p, q, w}^{-1}(\Omega_T)} = \inf \{ \|F\|_{L_{p, q, w}(\Omega_T)} + \|G_i\|_{L_{p, q, w}(\Omega_T)}: u = D_i G_i + F\}.
\]
Note that the linear functional $\varphi \to \langle u, \varphi\rangle_{\bH_{p, q, w}^{-1}(\Omega_T)}$ on $\mathring{W}_{p', q', w'}^{0,1}(\Omega_T)$ is uniquely determined if \eqref{eq0512_01} is determined for all $\varphi \in C_0^\infty((0,T) \times \Omega)$.
Then, we consider a linear functional
\begin{equation}
							\label{eq0514_01}
\varphi \to \langle u, J^{1-\alpha} \varphi \rangle_{\bH_{p, q, w}^{-1}(\Omega_T)}
\end{equation}
for $\varphi \in C_0^\infty((0,T) \times \Omega)$.
Because
\[
\langle u, J^{1-\alpha} \varphi \rangle_{\bH_{p, q, w}^{-1}(\Omega_T)} = \int_{\Omega_T} \left( (I^{1-\alpha}F) \varphi - (I^{1-\alpha} G_i) D_i \varphi \right) \, \mathrm{d}x \, \mathrm{d}t
\]
and $I^{1-\alpha}F, I^{1-\alpha}G_i \in L_{p, q, w}(\Omega_T)$ (see Lemma \ref{lem0224}), the linear functional in \eqref{eq0514_01} can be uniquely extended to all $\varphi \in \mathring{W}_{p', q', w'}^{0,1}(\Omega_T)$.
We denote this linear functional on $\mathring{W}_{p', q', w'}^{0,1}(\Omega_T)$ by $I^{1-\alpha} u$ so that $I^{1-\alpha} u \in \bH_{p, q, w}^{-1}(\Omega_T)$ and
\begin{equation}
							\label{eq0512_05}
\langle I^{1-\alpha} u, \varphi \rangle_{\bH_{p, q, w}^{-1}(\Omega_T)} = \int_{\Omega_T} \left( (I^{1-\alpha}F) \varphi - (I^{1-\alpha} G_i) D_i \varphi \right) \, \mathrm{d}x \, \mathrm{d}t.
\end{equation}

When the spatial domain $\Omega$ is concerned, we define $H_{p, w_2}^{-1}(\Omega)$ in a similar way to defining $\bH_{p, q, w}^{-1}(\Omega_T)$.
In particular, if $u_0 \in H_{p, w_2}^{-1}(\Omega)$ and
\begin{equation}
							\label{eq0512_04}
\langle u_0, \phi \rangle_{H_{p, w_2}^{-1}(\Omega)} = \int_\Omega \left( f \phi - g_i D_i \phi \right) \, \mathrm{d}x
\end{equation}
for $\phi \in C_0^\infty(\Omega)$, where $f, g_i \in L_{p, w_2}(\Omega)$, then, as an element of $\bH_{p, q, w}^{-1}(\Omega_T)$,
\begin{equation}
							\label{eq0512_03}
\langle u_0, \varphi \rangle_{\bH_{p, q, w}^{-1}(\Omega_T)} = \int_{\Omega_T} \left( f(x) \varphi(t,x) - g_i(x) D_i \varphi(t,x) \right) \, \mathrm{d}x \, \mathrm{d}t.
\end{equation}

\begin{definition}
	\label{def0513_6}
Let $\alpha \in (0, 1)$, $p, q \in (1, \infty)$ and $w(t, x) = w_1(t)w_2(x) = t^{\mu}w_2(x)$, where $\mu \in (-1, q-1)$ and $w_2 \in A_p(\bR^d)$.
We define $\cH_{p, q, w, 0}^{\alpha,-1}(\Omega_T)$ and $\cH_{p, q, w}^{\alpha,-1}(\Omega_T)$ as follows.
\begin{enumerate}[(i)]
\item By $u \in \cH_{p, q, w, 0}^{\alpha,-1}(\Omega_T)$ we mean that $u \in \bH_{p, q, w}^{-1}(\Omega_T)$ and there exist $g_i, f \in L_{p, q, w}(\Omega_T)$ such that
\begin{equation}
							\label{eq0601_01}
\langle I^{1-\alpha} u, \varphi_t \rangle_{\bH_{p, q, w}^{-1}(\Omega_T)} = \int_{\Omega_T} \left(g_i D_i \varphi - f \varphi \right) \, \mathrm{d}x \, \mathrm{d}t
\end{equation}
for $\varphi \in  C_0^\infty\left([0,T) \times \Omega\right)$.
In this case, in the distribution sense, we clearly have
\[
\partial_t I^{1-\alpha} u = D_i g_i + f \quad \text{in} \,\, \Omega_T,
\]
which, as a linear functional, further satisfies
\[
\langle \partial_t I^{1-\alpha}u, \varphi \rangle_{\bH_{p, q, w}^{-1}(\Omega_T)} = \int_{\Omega_T} \left(f \varphi - g_i D_i \varphi\right) \, \mathrm{d}x \, \mathrm{d}t = -\langle I^{1-\alpha} u, \varphi_t \rangle_{\bH_{p, q, w}^{-1}(\Omega_T)}
\]
for $\varphi \in \mathring{W}_{p', q', w'}^{0,1}(\Omega_T)$, so $\partial_t I^{1-\alpha}u \in \bH_{p, q, w}^{-1}(\Omega_T)$.

\item By $u \in \cH_{p, q, w}^{\alpha,-1}(\Omega_T)$ we mean that $u \in \bH_{p, q, w}^{-1}(\Omega_T)$ and there exists $u_0 \in H_{p, w_2}^{-1}(\Omega)$ such that
\[
u - u_0 \in \cH_{p, q, w, 0}^{\alpha,-1}(\Omega_T).
\]
\end{enumerate}
We again suppress $q$ and $w$ if $p=q$ and $w = 1$.
\end{definition}

\begin{lemma}
				\label{lem04131517_1}
Let $\alpha \in (0,(1+\mu)/q)$ and $u_0 \in H_{p, w_2}^{-1}(\Omega)$.
Then, $u_0 \in \cH_{p, q, w, 0}^{\alpha,-1}(\Omega_T)$.
Thus, $\cH_{p, q, w, 0}^{\alpha,-1}(\Omega_T) = \cH_{p, q, w}^{\alpha,-1}(\Omega_T)$ for $\alpha \in (0,(1+\mu)/q)$.
\end{lemma}

\begin{proof}
Let $u_0 \in H_{p, w_2}^{-1}(\Omega)$ satisfy \eqref{eq0512_04}.
Then, as noted in \eqref{eq0512_03}, $u_0 \in \bH_{p, q, w}^{-1}(\Omega_T)$.
Note that
\[
I^{1-\alpha}f(x) = \frac{t^{1-\alpha}}{(1-\alpha)\Gamma(1-\alpha)} f(x), \quad \partial_t I^{1-\alpha} f(x) = \frac{t^{-\alpha}}{\Gamma(1-\alpha)} f(x) \in L_{p, q, w}(\Omega_T),
\]
because $\alpha q < 1 + \mu$, and the same observation can be made for $g_i$.
From \eqref{eq0512_03} and \eqref{eq0512_05} along with the above calculations of $I^{1-\alpha}f$ and $I^{1-\alpha} g_i$, it follows that $I^{1-\alpha} u_0 \in \bH_{p, q, w}^{-1}(\Omega_T)$ and
\[
\langle I^{1-\alpha} u_0, \varphi \rangle_{\bH_p^{-1}(\Omega_T)} = \frac{1}{(1-\alpha)\Gamma(1-\alpha)}\int_{\Omega_T} \left( t^{1-\alpha} f(x) \varphi - t^{1-\alpha} g_i(x) D_i \varphi \right) \, \mathrm{d}x \, \mathrm{d}t
\]
for $\varphi \in \mathring{W}_{p', q', w'}^{0,1}(\Omega_T)$.
Let $\varphi \in C_0^\infty([0,T) \times \Omega)$.
By replacing $\varphi$ with $\varphi_t$ in the above equality, and using integration by parts and the fact that $t^{1-\alpha}|_{t=0} = 0$, we obtain that
\begin{equation*}
- \langle I^{1-\alpha} u_0, \varphi_t \rangle_{\bH_p^{-1}(\Omega_T)} = \frac{1}{\Gamma(1-\alpha)}\int_{\Omega_T} \left( t^{-\alpha} f(x) \varphi - t^{-\alpha} g_i(x) D_i \varphi \right) \, \mathrm{d}x \, \mathrm{d}t
\end{equation*}
for $\varphi \in C_0^\infty([0,T) \times \Omega)$.
This proves that $u_0 \in \cH_{p, q, w, 0}^{\alpha,-1}(\Omega_T)$.

It then follows as in the proof of Lemma \ref{lem0513_1} that  $\cH_{p, q, w, 0}^{\alpha,-1}(\Omega_T) = \cH_{p, q, w}^{\alpha,-1}(\Omega_T)$ for $\alpha \in (0,(1+\mu)/q)$. 
\end{proof}

\begin{lemma}
							\label{lem0526_1}
Let $\alpha \in [(1+\mu)/q,1)$ and $u_0 \in H_{p, w_2}^{-1}(\Omega)$.
If $u_0 \in \cH_{p, q, w, 0}^{\alpha,-1}(\Omega_T)$, then $u_0 = 0$.
Thus, for $u \in \cH_{p, q, w}^{\alpha,-1}(\Omega_T)$, there exists a unique $u_0 \in H_{p, w_2}^{-1}(\Omega)$ such that $u-u_0 \in \cH_{p, q, w, 0}^{\alpha,-1}(\Omega_T)$ when $\alpha \in [(1+\mu)/q,1)$.
\end{lemma}

\begin{proof}
Let $u_0 \in H_{p, w_2}^{-1}(\Omega)$ satisfy \eqref{eq0512_04}.
Then, from \eqref{eq0512_05} and \eqref{eq0512_03},
$I^{1-\alpha} u_0 \in \bH_{p, q, w}^{-1}(\Omega_T)$ and
\begin{equation}
							\label{eq0513_06}
\langle I^{1-\alpha} u_0, \varphi \rangle_{\bH_{p, q, w}^{-1}(\Omega_T)} = \frac{1}{(1-\alpha)\Gamma(1-\alpha)} \int_{\Omega_T} \left( t^{1-\alpha} f(x) \varphi - t^{1-\alpha} g_i(x) D_i \varphi \right) \, \mathrm{d}x \, \mathrm{d}t
\end{equation}
for $\varphi \in \mathring{W}_{p', q', w'}^{0,1}(\Omega_T)$,
where $t^{1-\alpha}f(x), t^{1-\alpha}g_i(x) \in L_{p, q, w}(\Omega_T)$.
Since $u_0 \in \cH_{p, q, w, 0}^{\alpha,-1}(\Omega_T)$, from the definition we have
\[
\langle \partial_t I^{1-\alpha} u_0, \varphi \rangle_{\bH_{p, q, w}^{-1}(\Omega_T)} = - \langle I^{1-\alpha} u_0, \varphi_t \rangle_{\bH_{p, q, w}^{-1}(\Omega_T)}
\]
for $\varphi \in C_0^\infty\left([0,T) \times \Omega\right)$.
From this, \eqref{eq0513_06}, and integration by parts it follows that
\begin{equation}
							\label{eq0513_07}
\langle \partial_t I^{1-\alpha} u_0, \varphi \rangle_{\bH_{p, q, w}^{-1}(\Omega_T)} = \frac{1}{\Gamma(1-\alpha)} \int_{\Omega_T} \left( t^{-\alpha} f\left( x \right) \varphi - t^{-\alpha} g_i\left( x \right) D_i \varphi \right) \, \mathrm{d}x \, \mathrm{d}t
\end{equation}
for $\varphi \in C_0^\infty\left([0,T) \times \Omega\right)$. 
Set
\[
\varphi_k(t,x) = \phi(x) t^{\mu/q} \eta_k(t), \quad \phi(x) \in C_0^\infty(\Omega), \quad \eta_k(t) \in C^\infty_0(0,T).
\]
Note that $ t^{\mu/q} \eta_k(t) \in C^\infty_0(0,T)$.
Then, from \eqref{eq0513_07} we see that
\[
\langle \partial_t I^{1-\alpha} u_0, \varphi_k \rangle_{\bH_{p, q, w}^{-1}(\Omega_T)} = \frac{1}{\Gamma(1-\alpha)} \langle u_0, \phi \rangle_{H_{p, w_2}^{-1}(\Omega)} \int_0^T t^{-\alpha + \mu/q} \eta_k(t) \, \mathrm{d}t.
\]
Since $t^{-\alpha + \mu/q} \notin L_q(0,T)$, it is possible to find $\eta_k \in C_0^\infty(0,T)$ such that $\|\eta_k\|_{L_{q'}(0,T)} \leq 1$, $1/q+1/q'=1$, but
\[
\int_0^T t^{-\alpha + \mu/q}  \eta_k(t) \, \mathrm{d}t  \geq k \,\, \nearrow \infty \quad \textrm{as} \quad k \to \infty.
\]
This means that $\langle \partial_t I^{1-\alpha} u_0, \varphi_k \rangle_{\bH_{p, q, w}^{-1}(\Omega_T)}$ is unbounded unless $u_0 = 0$.

On the other hand,
since $u_0 \in \cH_{p, q, w, 0}^{\alpha,-1}(\Omega_T)$, by the definition, there exist $\tilde{f}(t,x), \tilde{g}_i(t,x) \in L_{p, q, w}(\Omega_T)$ such that
\[
\langle \partial_t I^{1-\alpha} u_0, \varphi \rangle_{\bH_{p, q, w}^{-1}(\Omega_T)} = \int_{\Omega_T} \left(\tilde{f}(t,x) \varphi - \tilde{g}_i(t,x) D_i \varphi\right) \, \mathrm{d}x \, \mathrm{d}t
\]
for $\varphi \in \mathring{W}_{p', q', w'}^{0,1}(\Omega_T)$.
From this, we see that if $\varphi_k(t,x) \in  C_0^\infty\left((0,T) \times \Omega\right)$ with
\[
\|\varphi_k\|_{L_{p', q', w'}(\Omega_T)} + \|D\varphi_k\|_{L_{p', q', w'}(\Omega_T)} \leq M, \quad k = 1, 2, \ldots
\]
for some positive number $M$,
then $\langle \partial_t I^{1-\alpha} u_0, \varphi_k \rangle_{\bH_{p, q, w}^{-1}(\Omega_T)}$ is also bounded in $k$.
Considering $\varphi_k(t,x) = \phi(x) t^{\mu/q}\eta_k(t)$ given above, one verifies
\[
\|\varphi_k\|_{L_{p', q', w'}(\Omega_T)} + \|D\varphi_k\|_{L_{p', q', w'}(\Omega_T)}
\]
\[
 =  \left( \| \phi \|_{L_{p', w_2'}(\Omega)} + \| D\phi \|_{L_{p', w_2'}(\Omega)} \right) \left( \int_0^T |t^{\mu/q } \eta_k(t)|^{q'} w_1'(t)\,\mathrm{d}t \right)^{1/q'}
\]
\[
= \left( \| \phi \|_{L_{p', w_2'}(\Omega)} + \| D\phi \|_{L_{p', w_2'}(\Omega)} \right) \| \eta_k\|_{L_{q'}(0, T)} \le \| \phi \|_{L_{p', w_2'}(\Omega)} + \| D\phi \|_{L_{p', w_2'}(\Omega)}.
\]
We therefore conclude that $u_0 = 0$ if $u_0 \in \cH_{p, q, w, 0}^{\alpha,-1}(\Omega_T)$.

The uniqueness of $u_0 \in H_{p, w_2}^{-1}(\Omega)$ for $u \in \cH_{p, q, w}^{\alpha,-1}(\Omega_T)$ with $u-u_0 \in \cH_{p, q, w, 0}^{\alpha,-1}(\Omega_T)$ follows as in the proof of Lemma \ref{lem0513_2}.
\end{proof}

We now define $\partial_t^\alpha u$ and the norm of $\cH_{p, q, w}^{\alpha,-1}(\Omega_T)$ as in Definition \ref{def_0223_1}.

\begin{definition}
\label{def_0223_6}
Let $p, q \in (1, \infty)$ and $w(t, x) = w_1(t)w_2(x) = t^{\mu}w_2(x)$, where $\mu \in (-1, q-1)$ and $w_2 \in A_p(\bR^d)$.
For $\alpha \in (0,(1+\mu)/q)$ and $u \in \cH_{p, q, w}^{\alpha,-1}(\Omega_T)$, we denote
\[
\partial_t^\alpha u = \partial_t I^{1-\alpha} u.
\]
For $\alpha \in [(1+\mu)/q,1)$ and $u \in \cH_{p, q, w}^{\alpha,-1}(\Omega_T)$, by Lemma \ref{lem0526_1} there exists a unique $u_0 \in H_{p, w_2}^{-1}(\Omega)$ such that $u-u_0 \in \cH_{p, q, w, 0}^{\alpha,-1}(\Omega_T)$. In this case, we denote
\[
\partial_t^\alpha u = \partial_t I^{1-\alpha}(u-u_0).
\]
For both cases, the norm of $\cH_{p, q, w}^{\alpha,-1}(\Omega_T)$ is defined by
\begin{equation*}
\|u\|_{\cH_{p, q, w}^{\alpha,-1}(\Omega_T)} = \|u\|_{\bH_{p, q, w}^{-1}(\Omega_T)} + \|\partial_t^\alpha u\|_{\bH_{p, q, w}^{-1}(\Omega_T)}.
\end{equation*}
\end{definition}

\begin{remark}
If $u \in \cH_{p, q, w}^{\alpha,-1}(\Omega_T)$ satisfies
\[
\langle \partial_t I^{1-\alpha}(u-u_0), \varphi \rangle_{\bH_{p, q, w}^{-1}(\Omega_T)} = \int_{\Omega_T} (f \varphi - g_i D_i \varphi) \, \mathrm{d}x \, \mathrm{d}t
\]
for $\varphi \in \mathring{W}_{p', q', w'}^{0,1}(\Omega_T)$, then one can write
\[
\partial_t^\alpha u = D_i g_i + f.
\]
For $u \in \cH_{p, q, w, 0}^{\alpha,-1}(\Omega_T)$, we see that
\[
\partial_t^\alpha u = \partial_t I^{1-\alpha}u.
\]
Regarding the norm for $\cH_{p, q, w, 0}^{\alpha,-1}(\Omega_T)$, we use the same norm $\|\cdot\|_{\cH_{p, q, w}^{\alpha,-1}(\Omega_T)}$ (and the same notation).
\end{remark}

\begin{lemma}
							\label{lem0517_3}
Let $\alpha \in (0,1)$ and $u \in \cH_{p, q, w, 0}^{\alpha,-1}(\Omega_T)$.
Then $u = I^{\alpha} \partial_t I^{1-\alpha} u$ in $\bH_{p, q, w}^{-1}(\Omega_T)$ and
\[
\|u\|_{\bH_{p, q, w}^{-1}(\Omega_T)} \leq N T^\alpha \|\partial_t^\alpha u\|_{\bH_{p, q, w}^{-1}(\Omega_T)},
\]
where $N = N(\alpha,q, \mu) > 0$.
Thus, if $\alpha \in [(1+\mu)/q,1)$ and $u \in \cH_{p, q, w}^{\alpha,-1}(\Omega_T)$ with $u-u_0 \in \cH_{p, q, w, 0}^{\alpha,-1}(\Omega_T)$,
then
\[
\|u\|_{\bH_{p, q, w}^{-1}(\Omega_T)} \leq N T^\alpha \|\partial_t^\alpha u\|_{\bH_{p, q, w}^{-1}(\Omega_T)} + N T^{(1+\mu)/q} \|u_0\|_{H_{p, q, w}^{-1}(\Omega)},
\]
where $N = N(\alpha,q, \mu) > 0$.
\end{lemma}

\begin{proof}
For $u \in \cH_{p, q, w, 0}^{\alpha,-1}(\Omega_T)$, using Lemma \ref{lem0517_1} with \eqref{eq0512_05}, we have
\[
\langle I^{\alpha}\partial_t I^{1-\alpha} u, \varphi \rangle_{\bH_{p, q, w}^{-1}(\Omega_T)} = \langle \partial_t I^{1-\alpha} u, J^\alpha \varphi \rangle_{\bH_{p, q, w}^{-1}(\Omega_T)} = - \langle  I^{1-\alpha} u,  \partial_t J^\alpha \varphi \rangle_{\bH_{p, q, w}^{-1}(\Omega_T)}
\]
\[
= - \langle  I^{1-\alpha} u,  J^\alpha \partial_t \varphi \rangle_{\bH_{p, q, w}^{-1}(\Omega_T)} = - \langle I^{\alpha} I^{1-\alpha} u, \partial_t \varphi \rangle_{\bH_{p, q, w}^{-1}(\Omega_T)} = \langle u, \varphi \rangle_{\bH_{p, q, w}^{-1}(\Omega_T)}
\]
for $\varphi \in C_0^\infty((0,T) \times \Omega)$.
This proves $I^{\alpha} \partial_t I^{1-\alpha} u = u$ in $\bH_{p, q, w}^{-1}(\Omega_T)$.
Then, the inequalities in the lemma are proved as in the proof of Lemma \ref{lem0517_2}.
In particular, we use
\[
\|I^\alpha v\|_{\bH_{p, q, w}^{-1}(\Omega_T)} \leq N T^\alpha \|v\|_{\bH_{p, q, w}^{-1}(\Omega_T)},
\]
which holds since we have
$I^\alpha v = I^\alpha f + D_i I^\alpha g_i$ for $v = f + D_i g_i$, and 
\[
\|I^\alpha h\|_{L_{p, q, w}(\Omega_T)} \leq N T^\alpha \|h\|_{L_{p, q, w}(\Omega_T)}
\]
for $h \in L_{p, q, w}(\Omega_T)$ as in the proof of Lemma \ref{lem0517_2}.
\end{proof}

\begin{remark}
Lemma \ref{lem0517_3} shows that, for $u \in \cH_{p,q, w, 0}^{\alpha,-1}(\Omega_T)$,
\[
\|\partial_t^\alpha u\|_{\bH_{p, q, w}^{-1}(\Omega)} \simeq \|u\|_{\cH_{p, q, w}^{\alpha,-1}(\Omega_T)}.
\]
\end{remark}

\begin{lemma}
Let $\alpha \in (0,1)$ and $u \in \cH_{p, q, w, 0}^{\alpha,-1}(\Omega_T)$.
Then
\[
\| t^{-\alpha} u\|_{\bH_{p, q, w}^{-1}(\Omega_T)} \leq N \| \partial_t^\alpha u\|_{\bH_{p, q, w}^{-1}(\Omega_T)},
\]
where $N = N(\alpha, q, \mu) > 0$.
Furthermore, if $\alpha \in [(1+\mu)/q, 1)$ and $u \in \cH_{p, q, w}^{\alpha,-1}(\Omega_T)$ with $u - u_0 \in \cH_{p, q, w, 0}^{\alpha,-1}(\Omega_T)$, we have
\[
\| t^{-\alpha} (u - u_0) \|_{\bH_{p, q, w}^{-1}(\Omega_T)} \le N(\alpha, q, \mu) \| \partial_t^{\alpha}u \|_{\bH_{p, q, w}^{-1}(\Omega_T)}.
\]
\end{lemma}

\begin{proof}
By definition, the second assertion follows from the first one.
Let $g_i, f \in L_{p, q, w}(\Omega_T)$ satisfy \eqref{eq0601_01}.
Then, by Lemma \ref{lem0517_3},  for $\varphi \in C_0^\infty((0,T) \times \Omega)$,
\[
\langle t^{-\alpha} u, \varphi \rangle_{\bH_{p, q, w}^{-1}(\Omega_T)} = \langle u, t^{-\alpha} \varphi \rangle_{\bH_{p, q, w}^{-1}(\Omega_T)} = \langle \partial_t^\alpha u, J^\alpha (t^{-\alpha} \varphi) \rangle_{\bH_{p, q, w}^{-1}(\Omega_T)} 
\]
\[
= \int_{\Omega_T} \left( f J^\alpha (t^{-\alpha} \varphi) - g_i J^\alpha (t^{-\alpha} D_i \varphi) \right) \, \mathrm{d}x \, \mathrm{d}t
= \int_{\Omega_T} t^{-\alpha}\left( I^\alpha f \, \varphi - I^\alpha g_i D_i \, \varphi \right) \, \mathrm{d}x \, \mathrm{d}t,
\]
where as shown in the proof of Lemma \ref{lem0225_1}, by \eqref{eq0224_03} of Lemma \ref{lem0224} we have
\[
\|t^{-\alpha} I^\alpha f\|_{L_{p, q, w}(\Omega_T)} \leq N \|f\|_{L_{p, q, w}(\Omega_T)}, \quad \|t^{-\alpha} I^\alpha g_i\|_{L_{p, q, w}(\Omega_T)} \leq N \|g_i\|_{L_{p, q, w}(\Omega_T)}.
\]
Hence, we see that the desired inequality follows.
\end{proof}

\begin{proposition}
\label{prop0101_01}
If $u \in \cH_{p, q, w, 0}^{\alpha,-1}(\Omega_T)$, there exists a sequence $\{u_k\}$ such that $u_k \in C^\infty([0,T] \times \Omega)$, $u_k(t,x)$ vanishes for large $|x|$ (if $\Omega$ is unbounded), $u_k(0,x) = 0$, and $u_k \to u$ in $\cH_{p, q, w, 0}^{\alpha,-1}(\Omega_T)$ as $k \to \infty$.
\end{proposition}

\begin{proof}
We prove the case $\Omega = \bR^d$.
For $u \in \cH_{p, q, w, 0}^{\alpha,-1}(\Omega_T)$, let
\[
\langle u, \varphi \rangle_{\bH_{p, q, w}^{-1}(\Omega_T)} = \int_{\Omega_T}(F \varphi - G_i D_i \varphi ) \, \mathrm{d}x \, \mathrm{d}t,
\]
\[
\langle \partial_t I^{1-\alpha} u, \varphi \rangle_{\bH_{p, q, w}^{-1}(\Omega_T)} = \int_{\Omega_T}(f \varphi - g_i D_i \varphi ) \, \mathrm{d}x \, \mathrm{d}t
\]
for $\varphi \in W_{p', q', w'}^{0,1}(\Omega_T)$, where $F, G_i, f, g_i \in L_{p, q, w}(\Omega_T)$.

For $\xi = \xi(x) \in C_0^\infty(\bR^d)$, using the definition of $I^{1-\alpha}$ in \eqref{eq0512_05}, we see that $I^{1-\alpha}(\xi u) = \xi I^{1-\alpha} u$.
Then, $\langle \partial_t I^{1-\alpha} (\xi u), \varphi \rangle_{\bH_{p, q, w}^{-1}(\Omega_T)}$ is equal to
\[
\langle \xi \partial_t I^{1-\alpha} u, \varphi \rangle_{\bH_{p, q, w}^{-1}(\Omega_T)} = \int_{\Omega_T} \left((\xi f - g_i D_i \xi) \varphi - (\xi g_i) D_i \varphi \right) \, \mathrm{d}x \, \mathrm{d}t
\]
for all $\varphi \in C_0^\infty((0,T) \times \Omega)$, which further holds for all $\varphi \in \mathring{W}_{p', q', w'}^{0,1}(\Omega_T)$.
Hence, $\partial_t I^{1-\alpha} (\xi u) = \xi \partial_t I^{1-\alpha} u$
in $\bH_{p, q, w}^{-1}(\Omega_T)$.
This indicates that we can assume that $u$, $F$, $G_i$, $f$, and $g_i$ have compact support.

Let $\eta(t)$ and $\phi(x)$ be non-negative infinitely differentiable functions  with compact support on $\bR$ and $\bR^d$, respectively, such that $\eta(t) = 0$ for $t \in (-\infty,0) \cup (1,\infty)$ and $\int_\bR \eta(t) \, \mathrm{d}t = \int_{\bR^d} \phi(x) \, \mathrm{d}x = 1$.
We then set
\[
\varphi(t,x) = \eta(t) \phi(x), \quad \varphi_\varepsilon(t,x) = \frac{1}{\varepsilon^{d+2}}\varphi(t/\varepsilon^2, x/\varepsilon).
\]
For each $(t,x) \in \Omega_T$, we see that $\varphi_\varepsilon(t-s,x-y) \in C_0^\infty([0,T) \times \Omega)$ as a function of $(s,y)$.
Thus, for $w \in \bH_{p, q, w}^{-1}(\Omega_T)$,
\[
w^{(\varepsilon)}(t,x) = \langle w, \varphi_\varepsilon (t-\cdot,x-\cdot) \rangle_{\bH_{p, q, w}^{-1}(\Omega_T)}
\]
is well defined.
Then
\begin{multline}
							\label{eq0516_02}
u^{(\varepsilon)}(t,x) = \int_{\Omega_T} F(s,y) \varphi_\varepsilon(t-s,x-y) \, dy \, \mathrm{d}s + \int_{\Omega_T} G_i(s,y) (D_i\varphi_\varepsilon)(t-s,x-y) \, dy \, \mathrm{d}s
\\
= F^{(\varepsilon)}(t,x) + D_i  G_i^{(\varepsilon)}(t,x),
\end{multline}
and similarily,
\[
(\partial_t I^{1-\alpha}u)^{(\varepsilon)}(t,x) = f^{(\varepsilon)}(t,x) + D_i  g_i^{(\varepsilon)}(t,x),
\]
where, for a function $h$ on $\Omega_T = (0,T) \times \bR^d$, we denote
\[
h^{(\varepsilon)}(t,x) = \int_{\Omega_T} h(s,y) \varphi_\varepsilon(t-s,x-y) \, dy \, \mathrm{d}s
\]
\[
= \int_{\bR \times \bR^d} 1_{0<s<T} h(s,y) \varphi_\varepsilon(t-s,x-y) \, dy \, \mathrm{d}s.
\]

We now prove
\begin{equation}
							\label{eq0516_01}
\partial_t I^{1-\alpha} u^{(\varepsilon)}(t,x) = (\partial_t I^{1-\alpha} u)^{(\varepsilon)}(t,x)
\end{equation}
on $\Omega_T$.
If this is the case, $\{u^{(\varepsilon)}\}$ is a desired sequence.
Indeed, $u^{(\varepsilon)} \in C^\infty([0,T] \times \Omega)$ and
\[
u^{(\varepsilon)}(0,x) = F^{(\varepsilon)}(0,x) + D_i G_i^{(\varepsilon)}(0,x) = 0,
\]
because $\eta(t) = 0$ for $t \leq 0$.
Moreover, $u^{(\varepsilon)}(t,x)$ vanishes for large $|x|$ because $F$ and $G_i$ have compact support.
Since $(F^{(\varepsilon)}, G_i^{(\varepsilon)}, f^{(\varepsilon)}, g_i^{(\varepsilon)}) \to (F, G_i, f, g_i)$ in $L_{p, q, w}(\Omega_T)$,
we see that
\[
u^{(\varepsilon)} \to u, \quad (\partial_t I^{1-\alpha} u)^{(\varepsilon)} \to \partial_t I^{1-\alpha} u
\]
in $\bH_{p, q, w}^{-1}(\Omega_T)$, which together with \eqref{eq0516_01} means that $u^{(\varepsilon)} \to u$ in $\cH_{p, q, w, 0}^{\alpha,-1}(\Omega_T)$.

To prove \eqref{eq0516_01}, we first see that from \eqref{eq0516_02}
\begin{multline}
							\label{eq0516_03}
I^{1-\alpha} u^{(\varepsilon)}(t,x) = I^{1-\alpha} F^{(\varepsilon)}(t,x) + I^{1-\alpha} D_i G_i^{(\varepsilon)}(t,x)
\\
= (I^{1-\alpha} F)^{(\varepsilon)}(t,x) + D_i  (I^{1-\alpha} G_i)^{(\varepsilon)}(t,x)
\end{multline}
on $\Omega_T$,
where the second equality is due to the choice of $\eta(t)$ (see the proof of \cite[Proposition 3.2]{MR3899965}) so that, for a function $h$ on $\Omega_T$, we have
\begin{equation*}
(I^{1-\alpha} h)^{(\varepsilon)}(t,x) = I^{1-\alpha} h^{(\varepsilon)}(t,x).
\end{equation*}
By the equalities in \eqref{eq0516_03} we have
\[
\partial_t I^{1-\alpha} u^{(\varepsilon)} (t,x) 
= \partial_t (I^{1-\alpha} F)^{(\varepsilon)}(t,x) + \partial_t D_i (I^{1-\alpha} G_i)^{(\varepsilon)}(t,x),
\]
which finally proves \eqref{eq0516_01} because
\[
(\partial_t I^{1-\alpha} u)^{(\varepsilon)}(t,x) = \langle \partial_t I^{1-\alpha} u, \varphi_\varepsilon (t-\cdot,x-\cdot) \rangle_{\bH_{p, q, w}^{-1}(\Omega_T)}
\]
\[
= \langle I^{1-\alpha} u,  (\partial_t \varphi_\varepsilon) (t-\cdot,x-\cdot) \rangle_{\bH_{p, q, w}^{-1}(\Omega_T)}
\]
\[
= \int_{\Omega_T} (I^{1-\alpha}F)(s,y) (\partial_t \varphi_\varepsilon)(t-s,x-y) \, \mathrm{d}y \, \mathrm{d}s
\]
\[
+ \int_{\Omega_T} (I^{1-\alpha} G_i)(s,y) (\partial_t D_i \varphi_\varepsilon)(t-s,x-y) \, \mathrm{d}y \, \mathrm{d}s
\]
\[
=\partial_t (I^{1-\alpha} F)^{(\varepsilon)}(t,x) + \partial_t D_i (I^{1-\alpha} G_i)^{(\varepsilon)}(t,x) = \partial_t I^{1-\alpha} u^{(\varepsilon)}(t,x).
\]
For general $\Omega$, the claim is proved using a partition of unity with respect to the spatial variables. In particular, for $\zeta \in C_0^{\infty}(\Omega)$, we have $\zeta u \in \cH_{p, q, w}^{\alpha, -1}( \bR^d_T)$.
The proposition is proved.
\end{proof}

\begin{lemma}[Trace inequality]
	\label{lem1201_1}
Let $\alpha \in [(1+\mu)/q,1)$ and $u \in \cH_{p, q, w}^{\alpha,-1}(\Omega_T)$ with $u-u_0 \in \cH_{p, q, w, 0}^{\alpha,-1}(\Omega_T)$.
Then,
\begin{equation}
\label{eq0114_01}
\|u_0\|_{H_{p, w_2}^{-1}(\Omega)} \leq N T^{- (1+\mu)/q}\|u\|_{\bH_{p, q, w}^{-1}(\Omega_T)} + N T^{\alpha - (1+\mu)/q} \|\partial_t^\alpha u\|_{\bH_{p, q, w}^{-1}(\Omega_T)},
\end{equation}
where $N = N(\alpha,q, \mu) > 0$. Furthermore, if $T \geq 1$, we have
\[
\|u_0\|_{H_{p, w_2}^{-1}(\Omega)} \leq N(\alpha, q, \mu) \| u \|_{\cH_{p, q, w}^{\alpha, -1}(\Omega_T)}.
\]
\end{lemma}

\begin{proof}
Since $u, \partial_t I^{1-\alpha}(u-u_0) \in \bH_{p, q, w}^{-1}(\Omega_T)$, there exist $F, G_i, f, g_i \in L_{p, q, w}(\Omega_T)$ such that, for $\varphi \in \mathring{W}_{p', q', w'}^{0,1}(\Omega_T)$,
\[
\langle u, \varphi \rangle_{\bH_{p, q, w}^{-1}(\Omega_T)} = \int_{\Omega_T} \left( F \varphi - G_i D_i \varphi \right) \, \mathrm{d}x \, \mathrm{d}t,
\]
\[
\langle I^{1-\alpha}u, \varphi \rangle_{\bH_{p, q, w}^{-1}(\Omega_T)} = \int_{\Omega_T} \left( (I^{1-\alpha} F) \varphi - (I^{1-\alpha}G_i) D_i \varphi \right) \, \mathrm{d}x \, \mathrm{d}t,
\]
and
\[
\langle \partial_t I^{1-\alpha}(u-u_0), \varphi \rangle_{\bH_{p, q, w}^{-1}(\Omega_T)} = \int_{\Omega_T} \left(f \varphi - g_i D_i \varphi \right) \, \mathrm{d}x \, \mathrm{d}t.
\]
For $\varphi \in C_0^\infty\left((0,T) \times \Omega\right)$, set
\[
\Phi(t,x) = \int_t^T \varphi(s,x) \, \mathrm{d}s \in C_0^\infty\left([0,T) \times \Omega\right).
\]
Then, $\Phi_t(t,x) = - \varphi(t,x)$
and by definition, we have
\[
\langle I^{1-\alpha}(u-u_0), \Phi_t \rangle_{\bH_{p, q, w}^{-1}(\Omega_T)} = - 
\langle \partial_t I^{1-\alpha}(u-u_0), \Phi \rangle_{\bH_{p, q, w}^{-1}(\Omega_T)}.
\]
This shows that
\[
\langle I^{1-\alpha}(u-u_0), \varphi \rangle_{\bH_{p, q, w}^{-1}(\Omega_T)} = \int_{\Omega_T} \left(f \Phi - g_i D_i \Phi \right) \, \mathrm{d}x \, \mathrm{d}t
\]
\[
= \int_{\Omega_T} \left(\tilde{f} \varphi - \tilde{g}_i D_i \varphi\right) \, \mathrm{d}x \, \mathrm{d}t,
\]
where
\[
\tilde{f}(t,x) = \int_0^t f(s,x) \, \mathrm{d}s, \quad \tilde{g}_i(t,x) = \int_0^t g_i(s,x) \, \mathrm{d}s.
\]
Since $I^{1-\alpha}u_0 = I^{1-\alpha}u - I^{1-\alpha}(u-u_0)$ in $\bH_{p, q, w}^{-1}(\Omega_T)$,
it follows that
\begin{equation}
							\label{eq0513_08}
\langle I^{1-\alpha}u_0, \varphi \rangle_{\bH_{p, q, w}^{-1}(\Omega_T)} = \int_{\Omega_T} \left((I^{1-\alpha} F - \tilde{f}) \varphi - (I^{1-\alpha} G_i - \tilde{g}_i)D_i \varphi \right) \, \mathrm{d}x \, \mathrm{d}t
\end{equation}
for any $\varphi \in \mathring{W}_{p', q', w'}^{0,1}(\Omega_T)$.
If $\varphi = \phi(x) \in C_0^\infty(\Omega)$, by the definition of $I^{1-\alpha}$ and \eqref{eq0512_03}
\begin{multline}
							\label{eq0513_09}
\langle I^{1-\alpha}u_0, \varphi \rangle_{\bH_{p, q, w}^{-1}(\Omega_T)} = \frac{1}{(1-\alpha)\Gamma(1-\alpha)} \langle u_0, \phi \rangle_{H_{p, w_2}^{-1}(\Omega)} \int_0^T t^{1-\alpha} \, \mathrm{d}t
\\
= \frac{T^{2-\alpha}}{(2-\alpha)(1-\alpha)\Gamma(1-\alpha)} \langle u_0, \phi \rangle_{H_{p, w_2}^{-1}(\Omega)}.
\end{multline}
Note that
\[
\|I^{1-\alpha} F - \tilde{f}\|_{L_{p, q, w}(\Omega_T)} \leq NT^{1-\alpha} \|F\|_{L_{p, q, w}(\Omega_T)} + N T \|f\|_{L_{p, q, w}(\Omega_T)}, 
\]
\[
\|I^{1-\alpha} G_i - \tilde{g}_i\|_{L_{p, q, w}(\Omega_T)} \leq N T^{1-\alpha} \|G_i\|_{L_{p, q, w}(\Omega_T)} + N T\|g_i\|_{L_{p, q, w}(\Omega_T)},
\]
where $N = N(\alpha, q, \mu)$.
From these inequalities with \eqref{eq0513_08} and \eqref{eq0513_09}, we get \eqref{eq0114_01}. If $T \geq 1$, we choose $\varphi(t,x) = \phi(x) \eta(t)$, where $\phi(x) \in C_0^\infty(\Omega)$ and $\eta(t) \in C_0^\infty(0,1)$ with $\eta(t) \geq 0$, $\eta(t) = 1$ on $(1/4,3/4)$.
\end{proof}

Let $u \in \bH_{p, q, w}^{-1}(\Omega_T)$ satisfy
\[
\langle u, \varphi \rangle_{\bH_{p, q, w}^{-1}(\Omega_T)} = \int_{\Omega_T} \left( F \varphi - G_i D_i \varphi \right) \, \mathrm{d}x \, \mathrm{d}t,
\]
where $F, G_i \in L_{p, q, w}(\Omega_T)$.
Set
\[
\langle u(t,\cdot), \phi \rangle_{H_{p, w_2}^{-1}(\Omega)} = \int_\Omega \left(F(t,x) \phi(x) - G_i(t,x) D_i \phi(x)\right) \, \mathrm{d}x
\]
for $\phi \in C_0^\infty(\Omega)$, which is a well-defined for a.e. $t \in [0,T]$.
That is, $u(t,\cdot) \in H_{p, w_2}^{-1}(\Omega)$ for a.e. $t \in [0,T]$.

\begin{lemma}
	\label{lem230606_1}
Let $\alpha \in ((1+\mu)/q,1)$ if $\mu \in [0, q-1)$ and let $\alpha \in (1/q, 1)$ if $\mu \in (-1, 0)$. For $u \in \cH_{p, q, w}^{\alpha,-1}(\Omega_T)$ with $u_0 \in H_{p, w_2}^{-1}(\Omega)$ and $u - u_0 \in \cH_{p, q, w, 0}^{\alpha,-1}(\Omega_T)$, there is a version of $u$ (still denoted by $u$) such that
\begin{equation}
	\label{eq230606_03}
\| u(t,\cdot) - u_0(\cdot)\|_{H_{p, w_2}^{-1}(\Omega)} \leq N t^{\alpha - (1+\mu)/q} \|\partial_t^\alpha u\|_{\bH_{p, q, w}^{-1}(\Omega_T)}
\end{equation}
for $t \in (0,T)$, where $N = N(\alpha, q, \mu) > 0$.
Thus,
\[
\| u(t,\cdot) - u_0(\cdot)\|_{H_{p, w_2}^{-1}(\Omega)} \to 0
\]
as $t \searrow 0$.
Furthermore, we have
\begin{equation}
	\label{eq230606_2}
\|u(t, \cdot)\|_{H_{p, w_2}^{-1}(\Omega)} \le N(\alpha, q, \mu, T) \| u \|_{\cH_{p, q, w}^{\alpha, -1}(\Omega_T)}.
\end{equation}
\end{lemma}

\begin{proof}
Note that \eqref{eq230606_2} follows from \eqref{eq230606_03} with the help of \eqref{eq0114_01} in Lemma \ref{lem1201_1} and the fact that $\alpha - (1+\mu)/q > 0$ under the assumptions on $\alpha, q$ and $\mu$.

Let
\[
\langle \partial_t I^{1-\alpha}(u-u_0), \varphi \rangle_{\bH_{p, q, w}^{-1}(\Omega_T)} = \int_{\Omega_T} \left(f \varphi - g_i D_i \varphi \right) \, \mathrm{d}x \, \mathrm{d}t,
\]
where $f, g_i \in L_{p, q, w}(\Omega_T)$.
From the equality $u-u_0 = I^\alpha \partial_t I^{1-\alpha}(u - u_0)$ in the proof of Lemma \ref{lem0517_3}, it follows that
\[
\langle u-u_0, \varphi \rangle_{\bH_{p, q, w}^{-1}(\Omega_T)} = 
\langle I^\alpha\partial_t I^{1-\alpha}(u-u_0), \varphi \rangle_{\bH_{p, q, w}^{-1}(\Omega_T)}
\]
\[
= \int_{\Omega_T} \left((I^\alpha f) \varphi - (I^\alpha g_i) D_i \varphi \right) \, \mathrm{d}x \, \mathrm{d}t.
\]
Hence,
\[
\langle u(t,\cdot)-u_0(\cdot), \phi \rangle_{H_{p, w_2}^{-1}(\Omega)} = \int_\Omega \left((I^\alpha f)(t,x) \phi(x) - (I^\alpha g_i) D_i \phi(x) \right) \, \mathrm{d}x.
\]
Then, as in the proof of Lemma \ref{lem1115_1},
\begin{multline*}
\left| \langle u(t,\cdot)-u_0(\cdot), \phi \rangle_{H_p^{-1}(\Omega)} \right|
\\
\leq N t^{\alpha - (1+\mu)/q} \left( \|f\|_{L_{p, q, w}(\Omega_T)} + \|g_i\|_{L_{p, q, w}(\Omega_T)}\right) \|\phi\|_{W_{p', w_2'}^1(\Omega)},
\end{multline*}
where $N = N(\alpha,q, \mu)$.
This proves the desired inequality in the lemma.
\end{proof}

\begin{remark}
	\label{rmk230604_2}
For later use, we note the following.
If $u, u_t  \in \bH_{p, q, w}^{-1}(\Omega_T)$, that is, there exist $G_i, F, g_i, f \in L_{p,q,w}(\Omega_T)$ such that, as linear functionals on $\varphi \in \mathring{W}_{p',  q', w'}^{0,1}(\Omega_T)$,
\[
u = D_i G_i + F \quad \text{and} \quad u_t = D_i g_i + f,
\]
then we have a version of $u$ (still denoted by $u$) such that $u(t, \cdot)$ for $t \in [0, T]$ is well-defined in $H_{p, w_2}^{-1}(\Omega)$ and satisfies
\begin{equation*}
\|u(t, \cdot)\|_{H_{p, w_2}^{-1}(\Omega)} \leq N(q, \mu) T^{-(1+\mu)/q} \left( \|u\|_{\bH_{p, q, w}^{-1}(\Omega_T)} + T \| \partial_t u\|_{\bH_{p, q, w}^{-1}(\Omega_T)} \right).
\end{equation*}
To check this, we assume that the involved functions are sufficiently smooth and make use of the fundamental theorem of calculus.
Then, we employ the approximation argument presented in the proof of Proposition \ref{prop0101_01}.
\end{remark}

\section{Main result: solution spaces and initial traces}
	\label{sol_space}
This section introduces solution spaces for time-fractional evolution equations using the function spaces defined in section \ref{ftn_space}.
Then, as the main results of this paper, we present initial trace theorems along with extension theorems for the constructed solution spaces.

\subsection{Solution spaces}
Let $k$ be a positive integer. 
Note that if $\partial_t^m v \in L_{p, q, w}(\Omega_T)$ (respectively, $\in \bH_{p, q, w}^{-1}(\Omega_T)$) for $m = 0, 1, \ldots, k$, then by Remark \ref{rmk230604_1} (respectively, Remark \ref{rmk230604_2}), $\partial_t^mv(0, x)$, $m = 0, 1, \ldots, k-1$, is well-defined in $L_{p, w_2}(\Omega)$ (respectively, in $H_{p, w_2}^{-1}(\Omega)$). Here, $\partial_t^m v \in \bH_{p, q, w}^{-1}(\Omega_T)$ is defined by, for some $g_{m, i} \in L_{p, q, w}(\Omega_T)$, $i = 1, \ldots, d$, and $f\in L_{p, q, w}(\Omega_T)$, 
\[
\langle \partial_t^mv, \varphi \rangle_{\bH_{p, q, w}^{-1}(\Omega_T)} = \int_{\Omega_T} \left( f \varphi - g_{m, i}D_i\varphi \right) \,\mathrm{d}x\,\mathrm{d}t
\]
for any $\varphi \in \mathring{W}_{p',  q', w'}^{0,1}(\Omega_T)$.
With these observations at hand, we define function spaces whose elements have $(k+\alpha)$-order fractional time derivatives, as well as zero initial values, where $k=1,2,\ldots$.
Recall that we have already addressed the case $k=0$ in Definitions  \ref{def0513_1} and \ref{def0513_6}.
We also define function spaces having both spatial and fractional time regularities with zero initial values.

\begin{definition}
\label{def0307_01}
Let $\alpha \in (0, 1)$, $k$ be a non-negative integer, $1 < p, q < \infty$, and $w = w_1(t)w_2(x) = t^{\mu}w_2(x)$ with $\mu \in (-1, q-1)$ and $w_2 \in A_p(\bR^d)$.

(i) By $u \in \bH_{p, q, w, 0}^{k + \alpha, 0}(\Omega_T)$  with $k \ge 1$, we mean that $\partial_t^m u \in L_{p, q, w}(\Omega_T)$ with  $\partial_t^mu(0, x) = 0$ for $m=0, 1, \ldots, k-1$, and $\partial_t^ku \in \bH_{p, q, w, 0}^{\alpha,0}(\Omega_T)$.
Similarly, by $u \in \cH_{p, q, w, 0}^{k+\alpha, -1}(\Omega_T)$, we mean that $\partial_t^m u \in \bH_{p, q, w}^{-1}(\Omega_T)$ with  $\partial_t^mu(0, x) = 0$ for $m=0, 1, \ldots, k-1$, and $\partial_t^ku \in \cH_{p, q, w, 0}^{\alpha,-1}(\Omega_T)$.

(ii) By $u \in \bH_{p, q, w, 0}^{k + \alpha,2}(\Omega_T)$, we mean that $u \in \bH_{p, q, w, 0}^{k + \alpha, 0}(\Omega_T)$ and $Du, D^2 u \in L_{p, q, w}(\Omega_T)$.
Similarly, by $u \in \cH_{p, q, w, 0}^{k + \alpha, 1}(\Omega_T)$, $u \in \cH_{p, q, w, 0}^{k + \alpha, 0}(\Omega_T)$ and $u, D u \in L_{p, q, w}(\Omega_T)$.
\end{definition}

Using the above definition, we now introduce solution spaces for time-fractional evolution equations.

\begin{definition}
\label{def0102_01}
Let $\alpha \in (0, 1)$, $k$ be a non-negative integer, $1 < p, q < \infty$, and $w = w_1(t)w_2(x) = t^{\mu}w_2(x)$ with $\mu \in (-1, q-1)$ and $w_2 \in A_p(\bR^d)$.
\begin{enumerate}[(i)]
\item For $k = 0$, we define 
\[
\bH_{p, q, w}^{k + \alpha,2}(\Omega_T) = \bH_{p, q, w}^{\alpha,2}(\Omega_T) := \{u \in \bH_{p, q, w}^{\alpha,0}(\Omega_T) : Du, D^2 u \in L_{p, q, w}(\Omega_T)\}
\]
with a norm
\[
\| u \|_{\bH_{p, q, w}^{ \alpha,2}(\Omega_T)} := \left\|  \left|Du\right| + \left|D^2u\right| \right\|_{L_{p, q, w}(\Omega_T)} + \left\|  u \right\|_{\bH_{p, q, w}^{\alpha, 0}(\Omega_T)}.
\]
Similarly, we define
\[
\cH_{p, q, w}^{k+\alpha,1}(\Omega_T) = \cH_{p, q, w}^{\alpha,1}(\Omega_T) := \{ u \in \cH_{p, q, w}^{\alpha,-1}(\Omega_T) : u, Du \in L_{p, q, w}(\Omega_T)\}
\]
with a norm
\[
\| u \|_{\cH_{p, q, w}^{\alpha,1}(\Omega_T)} := \left\| \left| u \right| + \left| Du \right|  \right\|_{L_{p, q, w}(\Omega_T)} + \left\|u \right\|_{\cH_{p, q, w}^{\alpha, -1}(\Omega_T)}.
\]

\item For a positive integer $k$, we define $\bH_{p, q, w}^{k + \alpha,2}(\Omega_T)$ by
\[
 \{u \in L_{p, q, w}(\Omega_T): \partial_t^ku \in \bH_{p, q, w}^{\alpha,0}(\Omega_T), Du, D^2 u, \partial_t u, \ldots, \partial_t^{k-1}u \in L_{p, q, w}(\Omega_T)\}
\]
with a norm
\[
\| u \|_{\bH_{p, q, w}^{k + \alpha,2}(\Omega_T)} := \left\| \left|u\right| + \left|Du\right| + \left|D^2u\right| +\sum_{m=1}^{k-1} \left|\partial_t^m u  \right| \right\|_{L_{p, q, w}(\Omega_T)} + \left\| \partial_t^k u \right\|_{\bH_{p, q, w}^{\alpha, 0}(\Omega_T)}.
\]
Similarly, we define $\cH_{p, q, w}^{k+\alpha,1}(\Omega_T)$ by
\[
 \{u \in L_{p, q, w}(\Omega_T): \partial_t^ku \in \cH_{p, q, w}^{\alpha,-1}(\Omega_T), Du \in L_{p, q, w}(\Omega_T), \partial_t u, \ldots, \partial_t^{k-1}u \in \bH_{p, q, w}^{-1}(\Omega_T)\}
\]
with a norm
\[
\| u \|_{\cH_{p, q, w}^{k+\alpha,1}(\Omega_T)} := \left\| \left| u \right| + \left| Du \right|  \right\|_{L_{p, q, w}(\Omega_T)} + \sum_{m=0}^{k-1} \left\| \partial_t^m u \right\|_{\bH_{p, q, w}^{-1}(\Omega_T)} + \left\| \partial_t^{k} u \right\|_{\cH_{p, q, w}^{\alpha, -1}(\Omega_T)}.
\]
\end{enumerate}
In both cases, we denote $\partial_t^{k+\alpha}u := \partial_t^{\alpha}(\partial_t^k u) = \partial_t I^{1 - \alpha}(\partial_t^k u - u_k)$ for $u_k \in L_{p,w_2}(\Omega)$ and $u_k \in H^{-1}_{p,w_2}(\Omega)$, respectively.
Note that we do not consider the interchange of the integer and fractional order derivatives, that is,
\[
\partial_t^{\alpha}(\partial_t^k u) = \partial_t I^{1-\alpha} \left(\partial_t^ku - u_k \right) \neq  \partial_t^k \left( \partial_tI^{1-\alpha} \left( u - u_0 \right) \right) =  \partial_t^k(\partial_t^{\alpha} u).
\]
\end{definition}

\begin{remark}
	\label{rmk0102_02}
If $u$ is sufficiently smooth, a standard definition for $(k+\alpha)$-times fractional derivative of $u$ in the literature is
\begin{equation}
	\label{eq0605_1}
\partial_t^{k+1} I^{1-\alpha} \left( u(t) - \sum_{m=0}^k \frac{t^m}{m!} \partial_t^mu(0) \right).
\end{equation}
It is worth noting that the $(k+\alpha)$-times fractional derivative considered in this paper ($\partial_t^{k+\alpha}u := \partial_t^{\alpha}(\partial_t^k u)$ in Definition \ref{def0102_01}) can also be represented as in \eqref{eq0605_1}.
Indeed, for $u \in \bH_{p, q, w}^{k+\alpha, 2}(\Omega_T)$, if $\alpha \in [(1+\mu)/q, 1)$, there is a $u_k \in L_{p, w_2}(\Omega)$ such that $\partial_t^{k+\alpha}u = \partial_t I^{1-\alpha} \left( \partial_t^k u - u_k \right)$ in $L_{p, q, w}(\Omega_T)$. On the other hand, since $\partial_t^m u(0, x) \in L_{p, w_2}(\Omega)$, $m = 0, 1, \ldots, k-1$, exists at least in the trace sense (Remark \ref{rmk230604_1}), we also have $\{ u_m \}_{m=0}^k \subset L_{p, w_2}(\Omega)$ such that
\[
\partial_t^{k+\alpha} u = \partial_t I^{1-\alpha} \left( \partial_t^k u - u_k \right) = \partial_t I^{1-\alpha} \left( \partial_t^k \left( u - \sum_{m=0}^k \frac{t^m}{m!}u_m \right) \right)
\]
\[
= \partial_t^{k+1} I^{1-\alpha} \left( u - \sum_{m=0}^k \frac{t^m}{m!}u_m \right) \quad \textrm{in} \quad L_{p, q, w}(\Omega_T),
\]
where, for the last equality, one can directly check that $I^{1-\alpha} \partial_t^k v = \partial_t^k I^{1-\alpha}v$ with $v := u - \sum_{m=0}^k \frac{t^m}{m!}u_m \in \bH_{p, q, w, 0}^{k + \alpha, 0}(\Omega_T)$.
If $\alpha \in (0, (1+\mu)/q)$, we just set $u_k = 0$. 

Similarly, for $u \in \cH_{p, q, w}^{k+\alpha, 1}(\Omega_T)$ with $\alpha \in [(1+\mu)/q, 1)$, there exists $\{ u_m \}_{m=0}^k \subset H_{p, w_2}^{-1}(\Omega)$ (Remark \ref{rmk230604_2}) such that
\[
\partial_t^{k+\alpha} u = \partial_t^{k+1} I^{1-\alpha} \left( u - \sum_{m=0}^k \frac{t^m}{m!}u_m \right) \quad \textrm{in} \quad \bH_{p, q, w}^{-1}(\Omega_T),
\]
where $u - \sum_{m=0}^k \frac{t^m}{m!}u_m \in \cH_{p, q, w, 0}^{k + \alpha, -1}(\Omega_T)$. If $\alpha \in (0, (1+\mu)/q)$, we just set $u_k = 0$.
\end{remark}

\begin{proposition}
$\bH_{p, q, w}^{k + \alpha,2}(\Omega_T)$ and $\cH_{p, q, w}^{k+\alpha,1}(\Omega_T)$ are Banach spaces.
\end{proposition}

\begin{proof}
We only consider the case $\alpha \in [(1+\mu)/q, 1)$ since the proof for the case $\alpha \in (0, (1+\mu)/q)$ is similar.
Let $\{ u^n \}_{n=1}^{\infty}$ is a Cauchy sequence in $\bH_{p, q, w}^{k + \alpha,2}(\Omega_T)$. It is clear that there is a $u \in L_{p, q, w}(\Omega_T)$ such that $Du, D^2u, \partial_tu, \ldots, \partial_t^k u \in L_{p, q, w}(\Omega_T)$ and $u^n, Du^n, D^2u^n,  \partial_tu^n, \ldots, \partial_t^k u^n$ converge to $u, Du, D^2u, \partial_tu, \ldots, \partial_t^k u$ in $L_{p, q, w}(\Omega_T)$, respectively.
On the other hand, there is a (unique) $u^n_k \in L_{p, w_2}(\Omega)$ such that $\partial_t^ku^n - u^n_k \in \bH_{p, q, w, 0}^{\alpha, 0}(\Omega_T)$. By Lemma \ref{lem0525_1}, $\{ u^n_k \}_{n=1}^{\infty}$ is a Cauchy sequence in $L_{p, w_2}(\Omega)$, so there is a $u_k \in L_{p, w_2}(\Omega)$ satisfying $u^n_k \to u_k$ in $L_{p, w_2}(\Omega)$. Suppose $\partial_t^{k+\alpha}u^n \to f$ in $L_{p, q, w}(\Omega_T)$. Then, for all $\varphi \in C_0^\infty\left([0,T) \times \Omega\right)$,
\[
\int_{\Omega_T}I^{1-\alpha} (\partial_t^ku - u_k) \, \varphi_t \, \mathrm{d}x \, \mathrm{d}t = \lim_{n \to \infty} \int_{\Omega_T}I^{1-\alpha} (\partial_t^ku^n - u^n_k) \, \varphi_t \, \mathrm{d}x \, \mathrm{d}t = - \int_{\Omega_T} f \, \varphi \, \mathrm{d}x \, \mathrm{d}t,
\]
where we use the fact that $\| I^{1-\alpha} f \|_{L_{p, q, w}(\Omega_T)} \le N(\alpha, q, T) \| f \|_{L_{p, q, w}(\Omega_T)}$ and $\partial_t^ku, u_k \in L_{p, q, w}(\Omega_T)$. Hence, $\partial_t^{k+\alpha}u = f$ in $L_{p, q, w}(\Omega_T)$ and $u^n$ converges to $u \in \bH_{p, q, w}^{k+\alpha, 2}(\Omega_T)$.

The proof for $\cH_{p, q, w}^{k + \alpha, 1}(\Omega_T)$ is almost the same, except that we use Lemma \ref{lem1201_1} instead of Lemma \ref{lem0525_1}. The proposition is proved.
\end{proof}

\begin{lemma}
\label{lem0102_04}
If $u \in \bH_{p, q, w, 0}^{k+\alpha, 2}(\Omega_T)$ (respectively, $\cH_{p, q, w, 0}^{k+\alpha,1}(\Omega_T)$), there exists a sequence $\{u_n\}$ such that $u_n \in C^\infty([0,T] \times \Omega)$, $u_n(t,x)$ vanishes for large $|x|$ (if $\Omega$ is unbounded), $\partial_t^m u_n(0,x) = 0$ for $m = 0, 1, \ldots, k$, and $u_n \to u$ in $\bH_{p, q, w}^{k+\alpha, 2}(\Omega_T)$ (respectively, $\cH_{p, q, w}^{k+\alpha,1}(\Omega_T)$) as $n \to \infty$.
\end{lemma}

\begin{proof}
The proof is almost identical to the proofs of Propositions \ref{prop0514_1} and \ref{prop0101_01}.
For simplicity, we give a sketched proof. In particular, for the case of $u \in \cH_{p, q, w, 0}^{k+\alpha,1}(\Omega_T)$ with $\Omega = \bR^d$, as in the proof of Proposition \ref{prop0101_01}, we set for $w \in \bH_{p, q, w}^{-1}(\Omega_T)$, 
\[
w^{(\varepsilon)}(t, x) = \langle w, \varphi_\varepsilon (t-\cdot,x-\cdot) \rangle_{\bH_{p, q, w}^{-1}(\Omega_T)}, \quad (t,x) \in \Omega_T,
\]
where $\varphi_\varepsilon(t-s,x-y) \in C_0^\infty([0,T) \times \Omega)$ as a function of $(s,y)$.
Then, for $m = 0, 1, \ldots, k-1$, since $\partial_t^m u (0, x) = 0$ in the trace sense, one verifies that $\partial_t^{m+1} u^{(\varepsilon)} = (\partial_t^{m+1} u)^{(\varepsilon)}$.
Indeed, if $v, \partial_t v \in \bH_{p, q, w}^{-1}(\Omega_T)$ with $v(0, x) = 0$, there exists $g_i, f \in L_{p, q, w}(\Omega_T)$ such that
\[
\langle \partial_t v, \varphi \rangle_{\bH_{p, q, w}^{-1}(\Omega_T)} = -
\langle v, \varphi_t \rangle_{\bH_{p, q, w}^{-1}(\Omega_T)} = \int_{\Omega_T} \left(f \varphi - g_i D_i \varphi\right) \, \mathrm{d}x \, \mathrm{d}t
\]
for all $\varphi \in C_0^\infty\left([0,T) \times \Omega\right)$. From this fact and the choice of $\eta$ where $\varphi(t, x) = \eta(t)\phi(x)$ (see the proof of Proposition \ref{prop0101_01}), we have $\partial_t^{m+1} u^{(\varepsilon)} = (\partial_t^{m+1} u)^{(\varepsilon)}$ for $m = 0, 1, \ldots, k-1$. Then it follows that $\partial_t I^{1-\alpha } \partial_t^k u^{(\varepsilon)} = \partial_t I^{1-\alpha} (\partial_t^k u)^{(\varepsilon)} = (\partial_t I^{1-\alpha} \partial_t^k u)^{(\varepsilon)} \to \partial_t I^{1-\alpha} \partial_t^k u$ in $\bH_{p, q, w}^{-1}(\Omega_T)$. $Du^{(\varepsilon)} = (Du)^{(\varepsilon)} \to Du$ in $L_{p, q, w}(\Omega_T)$ is directly follows from the definition of $u^{(\varepsilon)}$.
\end{proof}

\begin{proposition}
\label{prop0102_03}
If $u \in \bH_{p, q, w}^{k+\alpha, 2}(\Omega_T)$ (respectively, $\cH_{p, q, w}^{k+\alpha,1}(\Omega_T)$), there exists a sequence $\{u_n\}$ such that $u_n \in C^\infty([0,T] \times \Omega)$, $u_n(t,x)$ vanishes for large $|x|$ (if $\Omega$ is unbounded), and $u_n \to u$ in $\bH_{p, q, w}^{k+\alpha, 2}(\Omega_T)$ (respectively, $\cH_{p, q, w}^{k+\alpha,1}(\Omega_T)$) as $n \to \infty$.
In particular, if $\alpha \in [(1+\mu)/q,1)$, for $v \in L_{p, w_2}(\Omega)$ (respectively, $H_{p, w_2}^{-1}(\Omega)$) satisfying $\partial_t^ku - v \in \bH_{p, q, w, 0}^{\alpha, 0}(\Omega_T)$ (respectively, $\cH_{p, q, w, 0}^{\alpha,-1}(\Omega_T)$), we have $\partial_t^ku_n(0, \cdot) \to v$ in $L_{p, w_2}(\Omega)$ (respectively, $H_{p, w_2}^{-1}(\Omega)$) as $n \to \infty$.
\end{proposition}

\begin{proof}
We prove only the case when $\Omega = \bR^d$.
For a general $\Omega$, one can use a partition of unity argument, for example, in \cite{MR0164252}.

 Let $u \in \cH_{p, q, w}^{k+\alpha, 1}(\Omega_T)$. As mentioned in Remark \ref{rmk0102_02}, there exists $\{ u_m \}_{m=0}^k \subset H_{p, w_2}^{-1}(\Omega)$ such that $u - \sum_{m=0}^k \frac{t^m}{m!}u_m \in \cH_{p, q, w, 0}^{k + \alpha, -1}(\Omega_T)$. 
Now for a given $\delta > 0$, by taking mollification with respect to $x$ variable (after multiplying cut-off function if necessary) to $u(t, x)$ and $\{ u_m(x) \}_{m=0}^k$, we obtain a $v \in \cH_{p, q, w}^{k+\alpha, 1}(\Omega_T)$ and $\{v_m(x)\}_{m=0}^k \subset C_0^{\infty}(\Omega) \subset C^{\infty}([0, T] \times \Omega)$ such that $\| u - v \|_{\cH_{p, q, w}^{k+\alpha, 1}(\Omega_T)} < \delta/2$ and $v - \sum_{m=0}^k \frac{t^m}{m!}v_m \in \cH_{p, q, w, 0}^{k + \alpha, 1}(\Omega_T)$. 
Here, for $u_m \in H_{p, w_2}^{-1}(\Omega)$, we define its mollification $u_m^{(\varepsilon)} (= v_m)$ by
\[
u_m^{(\varepsilon)}(x) := \langle u_m, \phi_{\varepsilon}(x - \cdot)\rangle_{H_{p, w_2}^{-1}(\Omega)},
\]
where $\phi(x)$ is a standard mollifier and $\phi_{\varepsilon}(x) = \varepsilon^{-d}\phi(x/\varepsilon)$.
Then by Lemma \ref{lem0102_04}, there is a sequence $\{w_n\}$ such that $w_n \in C^\infty([0,T] \times \Omega)$, $w_n(t,x)$ vanishes for large $|x|$, and $w_n \to v - \sum_{m=0}^k \frac{t^m}{m!}v_m$ in $\cH_{p, q, w, 0}^{k+\alpha,1}(\Omega_T)$ as $n \to \infty$. That is, for sufficiently large $n$, we have $\|u_n - v \|_{\cH_{p, q, w}^{k+\alpha, 1}(\Omega_T)} < \delta/2$ where $u_n := w_n + \sum_{m=0}^k \frac{t^m}{m!}v_m$, and then $\|u_n - u \|_{\cH_{p, q, w}^{k+\alpha, 1}(\Omega_T)} < \delta$. Therefore, $\{u_n \}$ is the desired sequence.

Note that by Lemma \ref{lem1201_1}$, \partial_t^k u_n(0, \cdot) \to v$ in $H_{p, w_2}^{-1}(\Omega)$ as $n \to \infty$, where $v$ satisfies $\partial_t^ku - v \in \cH_{p, q, w, 0}^{\alpha, -1}(\Omega_T)$.
The proof for the case $u \in \bH_{p, q, w}^{k+\alpha, 2}(\Omega_T)$ is almost the same.
The proposition is proved.
\end{proof}

\begin{remark}
The fractional derivative $\partial_t^{\alpha}$ we use in this paper is of the Caputo type, but there are many other notions of fractional derivatives.
For example, in \cite{MR3488533}, the authors use the Marchaud fractional derivative $\mathcal{D}_t^{\alpha}$ to investigate a De Giorgi-Nash-Moser H\"older regularity theory for solutions to equations in divergence form. For sufficiently smooth function $u$ defined on $[0, T]$, $\mathcal{D}_t^{\alpha}u$ is defined by
\[
\Gamma(1-\alpha) \mathcal{D}_t^{\alpha}u(t) := \frac{u(t) - u(0)}{t^{\alpha}} + \alpha \int_0^t \frac{u(t) - u(s)}{(t-s)^{\alpha + 1}}\,\mathrm{d}s, \quad 0 < t < T,
\]
where $\alpha \in (0, 1)$.
Note that, by using integration by parts, we verify that $\partial_t^{\alpha}u$ and $\mathcal{D}_t^{\alpha}u$ are equivalent if $u \in C^1[0, T]$.
Also, for $u \in \bH_{p, q, w}^{\alpha, 2}(\Omega_T)$ with $\alpha \ge (1+\mu)/q$,
we have
\[
\Gamma(1-\alpha) \mathcal{D}_{t, \varepsilon}^{\alpha}u := \frac{u(t, x) - u_0(x)}{t^{\alpha}} + \alpha \int_0^{t-\varepsilon} \frac{u(t, x) - u(s, x)}{(t-s)^{\alpha + 1}}\,\mathrm{d}s \to \Gamma(1-\alpha) \partial_t^{\alpha}u
\]
in $L_{p, q, w}(\Omega_T)$ as $\varepsilon \to 0$.
\end{remark}

\subsection{Spaces for initial values}
To establish the initial trace spaces for $\bH_{p, q, w}^{k+\alpha, 2}(\Omega_T)$ and $\cH_{p, q, w}^{k+\alpha, 1}(\Omega_T)$, we introduce (weighted) Besov and Bessel potential spaces.

\begin{definition}
\label{def_besov}
Take a $\psi \in \cS(\bR^d)$ such that $\operatorname{supp} \hat{\psi} = \{ 1/2 \le |\xi| \le 2\}$, $\hat{\psi} \ge 0$, and $\sum_{j \in \bZ} \hat{\psi}(2^{-j}\xi) = 1$ for $\xi \neq 0$ where $\hat{\psi} := \mathcal{F}\psi$.
Denote $\hat{\psi}_j(\xi) = \hat{\psi}(2^{-j}\xi)$ for $j = \pm 1, \pm2, \ldots$ and $\hat{\Psi}(\xi) = 1 - \sum_{j=1}^{\infty} \hat{\psi_j}(\xi)$. 

For $1 < p, q < \infty$ and $s \in \bR$, we define the weighted Besov space as follows:
\[
B_{p, q, w_2}^s = B_{p, q, w_2}^s(\bR^d) = \{ f \in \cS'(\bR^d): \|f \|_{B_{p, q, w_2}^s(\bR^d)} < \infty \},
\]
where
\[
 \|f \|_{B_{p, q, w_2}^s(\bR^d)} := \| \Psi * f \|_{L_{p, w_2}(\bR^d)} + \left( \sum_{j=1}^{\infty} \ 2^{jsq} \| \psi_j* f \|_{L_{p, w_2}(\bR^d)}^q \right)^{1/q}.
\]
Here, $\varphi * f$ for $\varphi \in \cS(\bR^d)$ and $f \in \cS'(\bR^d)$ is defined by $\varphi * f(x) = \langle f, \varphi(x - \cdot) \rangle$, where $\langle f, \varphi(x - \cdot) \rangle$ means the action of $f \in \cS'(\bR^d)$ on the test function $\varphi(x - \cdot) \in \cS(\bR^d)$.

We also define weighted Bessel potential space as follows.
\[
H_{p, w_2}^s = H_{p, w_2}^s(\bR^d) =  (1 - \Delta)^{-s/2}L_{p, w_2}(\bR^d),
\]
and $H_{p, w_2}^{0} = L_{p, w_2}$, where $(1 - \Delta)^{s/2}$ is defined by
\[
(1 - \Delta)^{s/2} = \cF^{-1} \left(\left(1 + |\xi|^2 \right)^{s/2} \cF u  \right).
\]
Note that the Fourier transform $\cF f = \hat{f}$ of $f \in \cS'(\bR^d)$ is defined by
\[
\langle \hat{f}, \varphi \rangle = \langle f, \hat{\varphi} \rangle
\]
for $\varphi \in \cS(\bR^d)$.
In particular, if $u \in H_{p, w_2}^{-1}$, there exists $g_i, f \in L_{p, w_2}$, $i = 1, 2, \ldots, d$, such that $u = D_ig_i + f$ in the distribution sense.
\end{definition}

\begin{remark}
\label{rmk_interpol}
The following results on $A_p$-weighted Besov spaces and $A_p$-weighted Bessel potential spaces are introduced in \cite[Theorem 2.8 and Theorem 3.5]{MR0676560}.

(i) For any $\nu, s \in \bR$, $(1 - \Delta)^{\nu/2}$ is an isometry from $B_{p, q, w_2}^s$ and $H_{p, w_2}^s$ to $B_{p, q, w_2}^{s-\nu}$ and $H_{p, w_2}^{s- \nu}$, respectively.

(ii) For any $s \in \bR$, $B_{p, q, w_2}^s$ can be characterized by the real interpolation of $H_{p, w_2}^{s_0}$ and $H_{p, w_2}^{s_1}$;
\[
\left( H_{p, w_2}^{s_0}, H_{p, w_2}^{s_1} \right)_{\theta, q} = B_{p, q, w_2}^s,
\]
where $\theta \in (0, 1)$ and $s = (1-\theta)s_0 + \theta s_1$ with $s_0, s_1 \in \bR$. (The norm equivalence is determined by a constant $ N(d, p, [w_2]_{A_p}) >0$.)
In particular, we have
\begin{equation}
	\label{eq2306080355}
\| \sigma \|_{B_{p, q, w_2}^s}^q \simeq_{d, p, [w_2]_{A_p}}  \inf_{U_0, U_1} \int_0^{\infty} \left( \varepsilon^{-\theta} \left(  \| U_0(\varepsilon, \cdot) \|_{H_{p, w_2}^{s_0}} + \varepsilon  \| U_1(\varepsilon, \cdot) \|_{H_{p, w_2}^{s_2}}  \right) \right)^q \, \frac{\mathrm{d}\varepsilon}{\varepsilon},
\end{equation}
where the infimum is taken for all sufficiently smooth $U_0$ and $U_1$ such that $\sigma(x) = U_0(\varepsilon, x) + U_1(\varepsilon, x)$.

See \cite[Section 1]{MR0503903} for the precise definition and fundamental properties of real interpolations, especially, the $K$-method.
\end{remark}

For the extension theorems of the solution spaces $\bH_{p, q, w}^{k+\alpha, 2}(T)$ and $\cH_{p, q, w}^{k+\alpha, 1}(T)$, $k \in \{0, 1\}$, we consider the initial value problems (\eqref{eq0203_01} and \eqref{eq0203_02} below) for the time-fraction heat equations in non-divergence/divergence form.
Let $P_{\beta} =  P_{\beta}(t, x)$ be the fundamental solution of the time-fractional heat equation $\partial_t^{\beta} - \Delta$, and $\widetilde{P}_{\beta}(t, x) := \int_0^t P_{\beta}(s, x)\,\mathrm{d}s$. The existence of such $P_{\beta}$ is proved in, for example, \cite[Section 6.2]{MR3581300}. 
For sufficiently smooth functions $u_0$ and $u_1$, 
\begin{equation}
\label{eq0131_03}
U(t, \cdot) := P_{\beta}(t, \cdot) \ast u_0(\cdot) + 1_{\beta > 1}\widetilde{P}_{\beta}(t, \cdot) \ast u_1(\cdot)
\end{equation}
is a solution to 
\begin{align}
\label{eq0203_01}
    \begin{cases}
    \partial_t I^{1-\alpha} \left( U - u_0\right) - \Delta U = 0 &\quad \text{in}\quad  {\bR^d_T} \\
    U(0, \cdot) = u_0 (\cdot)&\quad \text{on}\quad  \bR^d
    \end{cases} 
\end{align}
if $\beta = \alpha \in (0, 1)$, and to
\begin{align}
\label{eq0203_02}
    \begin{cases}
    \partial_t^2I^{1 - \alpha} \left( U - u_0 - tu_1 \right) - \Delta U = 0 &\quad \text{in}\quad  {\bR^d_T} \\
    U(0, \cdot) = u_0 (\cdot)&\quad \text{on}\quad  \bR^d\\
    U_t(0, \cdot) = u_1 (\cdot) &\quad \text{on}\quad  \bR^d
    \end{cases} 
\end{align}
if $\beta = 1 + \alpha \in (1, 2)$. In particular, see \cite{MR2047909, MR4186022} for \eqref{eq0203_01} and \cite[Section 1]{MR3240385}, \cite[Appendix C]{MR4464544} for \eqref{eq0203_02}.

\begin{remark}
\label{rmk_mittag}
It is known that the Fourier transform of $P_{\beta}$ with respect to $x$ can be represented by the Mittag-Leffler function $E_\beta(\cdot)$, that is, $\hat{P}_{\beta}(t, \xi) = E_{\beta}(-|\xi|^2t^{\beta})$.
Also, the following are true.

(i) \cite[Lemma 3.2]{MR3581300}:
$P_{\beta}(t, x) =  t^{-\beta d/2} P_{\beta}(1, t^{-\beta/2}x)$ and
\begin{multline}
\label{eq0209_03}
  P_{\beta}(1, x) \le N 1_{|x| \ge 1}\left( e^{-\sigma |x|^{\frac{2}{2-\beta}}}  \right) 
\\
+ N 1_{|x| < 1}  |x|^{-d} \left( |x|^{2} +  |x|^{2} \left| \log{ \left| x \right|} \right|1_{d=2} +  |x|1_{d=1} \right),
\end{multline}
where $N = (d, \beta) > 0$ and $\sigma = \sigma(d, \beta) > 0$.

(ii) Let $c, \delta \in \bR$ be $c < 1 + \beta$ and $\frac{\pi \beta}{2} < \delta \le \pi \wedge \pi \beta$. Then the two parametric Mittag-Leffler function $E_{\beta, c}$ has an integral representation
\begin{equation}
\label{eq0223_01}
E_{\beta, c}(-v) = \frac{1}{\pi \beta} \int_0^{\infty} \frac{r^{\frac{1-c}{\beta}}e^{r^{\frac{1}{\beta}} \cos(\delta / \beta)} \left[ r \sin(\psi - \delta) + v \sin{\psi} \right] }{r^2 + 2rv\cos(\delta) + v^2}\,\mathrm{d}r
\end{equation}
for $v > 0$, where $\psi = \psi(r) = r^{1/\beta} \sin(\frac{\delta}{\beta}) + \delta(1 + \frac{1-c}{\beta})$.
In particular, the above identity \eqref{eq0223_01} can be obtained by taking $\epsilon \searrow 0$ in the formula (14) in \cite{MR1967847}, since $c < 1 + \beta$.

(iii) \cite[Theorem 1.3.5]{MR1249271}: If $\beta = \alpha \in (0, 1)$, $E_{\alpha} = E_{\alpha, 1}$ has an integral representation
\begin{equation}
\label{eq0209_02}
E_{\alpha}(-v) = \frac{\sin{\alpha \pi}}{\pi} \int_0^{\infty} \frac{r^{\alpha - 1}}{r^{2\alpha} + 2r^{\alpha} \cos{\alpha \pi} + 1} e^{-rv^{1/\alpha}}\,\mathrm{d}r
\end{equation}
for $v  > 0$.

(iv) \cite[Lemma 3.1]{MR3980916}: If $\beta = 1 + \alpha \in (1, 2)$,
\begin{equation}
\label{eq0223_02}
\cF(\widetilde{P}_{\beta})(t, \xi) = \int_0^t \hat{P}_{\beta}(s, \xi)\,\mathrm{d}s = tE_{\beta, 2}(-|\xi|^2t^{\beta}). 
\end{equation}
\end{remark}

We now present the initial trace/extension theorems for the solution spaces.
For the remaining part of this paper, if $\Omega = \bR^d$, we omit $\Omega = \bR^d$ and use $T$ instead of $\Omega_T = (0, T) \times \bR^d$ in the notation of function spaces, for example, $\bH_{p, q, w}^{k+\alpha, 2}(T) = \bH_{p, q, w}^{k+\alpha, 2}(\bR^d_T) $ and $L_{p, w_2} = L_{p, w_2}(\bR^d)$.

\subsection{Main results}

\begin{theorem}[Trace theorem]
	\label{thm0111_01}
Let $\alpha \in (0, 1)$, $p, q \in (1, \infty)$, $T \in (0, \infty)$, and $w(t, x) = w_1(t)w_2(x) = t^{\mu}w_2(x)$, where $\mu \in (-1, q-1)$ and $w_2 \in A_p(\bR^d)$.
Also, let $k$ and $n$ be non-negative integers with $n \le k$ and $\theta_n := (k + \alpha -n  - \frac{1 + \mu}{q})/(k + \alpha)$.
If $\alpha > (1+\mu)/q$ (so that $\theta_n \in (0,1)$ for all $n=0,1,\ldots,k$), for each $n \in \{0,1,\ldots,k\}$, the operator
\[
\cT_n: \bH_{p, q, w}^{k+\alpha, 2}(T) \to B_{p, q, w_2}^{2\theta_n}
\]
with $\cT_nu = \partial_t^nu(0, x)$ for $u \in \bH_{p, q, w}^{k+\alpha, 2}(T) \cap C^{\infty}([0, T] \times \bR^d)$ is bounded and satisfies
\begin{equation}
\label{trace0110_1}
\|\cT_n u \|_{B_{p, q, w_2}^{2\theta_n}} \le N \| u \|_{ \bH_{p, q, w}^{k+\alpha, 2}(T)},
\end{equation}
where $N = N(d, \alpha, p, q, [w_2]_{A_p}, \mu, k, n, T) > 0$.
Furthermore, $N$ is uniformly bounded for $T \ge 1$.
\end{theorem}

\begin{theorem}[Trace theorem]
				\label{cor0111_01}
Let $\alpha \in (0, 1)$, $p, q \in (1, \infty)$, $T \in (0, \infty)$ and $w(t, x) = w_1(t)w_2(x) = t^{\mu}w_2(x)$, where $\mu \in (-1, q-1)$ and $w_2 \in A_p(\bR^d)$.
Also, let $k$ and $n$ be non-negative integers with $n \le k$ and $\theta_n := (k -n + \alpha - \frac{1 + \mu}{q})/(k + \alpha)$.
If $\alpha > (1+\mu)/q$ (so that $\theta_n \in (0,1)$ for all $n=0,1,\ldots,k$), for each $n \in \{0,1,\ldots,k\}$, the operator
\[
\cT_n: \cH_{p, q, w}^{k+\alpha, 1}(T) \to B_{p, q, w_2}^{2\theta_n - 1}
\]
with $\cT_nu = \partial_t^nu(0, x)$ for $u \in \cH_{p, q, w}^{k+\alpha, 1}(T) \cap C^{\infty}([0, T] \times \bR^d)$ is bounded and satisfies
\begin{equation}
\label{trace0110_3}
\|\cT_n u \|_{B_{p, q, w_2}^{2\theta_n-1}} \le N \| u \|_{ \cH_{p, q, w}^{k+\alpha, 1}(T)},
\end{equation}
where $N = N(d, \alpha, p, q, [w_2]_{A_p}, \mu, k,n, T) > 0$.
Furthermore, $N$ is uniformly bounded for $T \ge 1$.
\end{theorem}

\begin{remark}
See Theorem \ref{thm230416} (and Remark \ref{rmk230416}) for the above trace results with general $\Omega$, not just $\Omega = \bR^d$. 
\end{remark}

\begin{theorem}[Extension theorem: sub-diffusion case]
\label{thm_subdiff}
Let $\alpha \in (0, 1)$, $p, q \in (1, \infty)$, $T \in (0, \infty)$ and $w(t, x) = w_1(t)w_2(x) = t^{\mu}w_2(x)$, where $\mu \in (-1, q-1)$ and $w_2 \in A_p(\bR^d)$.
Also, let 
\[
\theta = 1 - \frac{1 + \mu}{q\alpha} \in (0, 1), \quad \left( \alpha > \frac{1+\mu}{q} \right).
\]
Then we have the following:
\begin{enumerate}[(i)]

\item For any $u_0 \in B_{p, q, w_2}^{2\theta}$, \eqref{eq0203_01} has a unique solution $U \in \bH_{p, q, w}^{\alpha, 2}(T)$  satisfying
\begin{equation}
\label{eq0202_03}
\| U \|_{\bH_{p, q, w}^{\alpha, 2}(T)} \le N \| u_0 \|_{B_{p, q, w_2}^{2\theta}},
\end{equation}
where $N = N(d, \alpha, p, q, [w_2]_{A_p}, \mu, T) > 0$.

\item For any $v_0 \in B_{p, q, w_2}^{2\theta-1}$, \eqref{eq0203_01} has a unique solution  $V \in \cH_{p, q, w}^{\alpha, 1}(T)$ satisfying
\begin{equation}
\label{eq0224_01}
\| V \|_{\cH_{p, q, w}^{\alpha, 1}(T)} \le N \| v_0 \|_{B_{p, q, w_2}^{2\theta-1}},
\end{equation}
where $N = N(d, \alpha, p, q, [w_2]_{A_p}, \mu, T) > 0$.
\end{enumerate}
In particular, $U(0, \cdot) = u_0$ and $V(0, \cdot) = v_0$ are well-defined in the trace sense by Theorem \ref{thm0111_01}, since $\alpha > (1+\mu)/q$. 
\end{theorem}

\begin{theorem}[Extension theorem: super-diffusion case]
\label{thm_supdiff}
Let $\alpha \in (0, 1)$, $p, q \in (1, \infty)$, $T \in (0, \infty)$ and $w(t, x) = w_1(t)w_2(x) = t^{\mu}w_2(x)$, where $\mu \in (-1, q-1)$ and $w_2 \in A_p(\bR^d)$.
Also, let 
$$
\theta_0 = 1 - \frac{1+\mu}{q(1+\alpha)} \in (0, 1) \quad \textrm{and} \quad \theta_1 = 1 - \frac{1}{1+\alpha} - \frac{1+\mu}{q(1+\alpha)} \in (0, 1)
$$
($\alpha > (1+\mu)/q$). Then we have the following:
\begin{enumerate}[(i)]

\item For any $u_0 \in B_{p, q, w_2}^{2\theta_0}$ and $u_1 \in B_{p, q, w_2}^{2\theta_1}$, the equation \eqref{eq0203_02} has a solution $U \in \bH_{p, q, w}^{1 + \alpha, 2}(T)$  satisfying
\begin{equation*}
\| U \|_{\bH_{p, q, w}^{1 + \alpha, 2}(T)} \le N \left( \| u_0 \|_{B_{p, q, w_2}^{2\theta_0}} +  \| u_1 \|_{B_{p, q, w_2}^{2\theta_1}} \right),
\end{equation*}
where $N = N(d, \alpha, p, q, [w_2]_{A_p}, \mu, T) > 0$.

\item For any $v_0 \in B_{p, q, w_2}^{2\theta_0-1}$ and $v_1 \in B_{p, q, w_2}^{2\theta_1 -1}$, the equation \eqref{eq0203_02} has a solution  $V \in \cH_{p, q, w}^{1 + \alpha, 1}(T)$ satisfying
\begin{equation*}
\| V \|_{\cH_{p, q, w}^{1 + \alpha, 1}(T)} \le N \left( \| v_0 \|_{B_{p, q, w_2}^{2\theta_0-1}} + \| v_1 \|_{B_{p, q, w_2}^{2\theta_1-1}}\right),
\end{equation*}
where $N = N(d, \alpha, p, q, [w_2]_{A_p}, \mu, T) > 0$.
\end{enumerate}
In particular, $U(0, \cdot) = u_0$, $U_t(0, \cdot) = u_1$ and $V(0, \cdot) = v_0$, $V_t(0, \cdot) = v_1$ are well-defined in the trace sense by Theorem \ref{thm0111_01}, since $\alpha > (1+\mu)/q$.
\end{theorem}

Note that Theorem \ref{thm_subdiff} (respectively, Theorem \ref{thm_supdiff}) provides the optimality of the trace estimates for the solution spaces $\bH_{p, q, w}^{\alpha, 2}(T)$ and $\cH_{p, q, w}^{\alpha, 1}(T)$ (respectively, $\bH_{p, q, w}^{1 + \alpha, 2}(T)$ and $\cH_{p, q, w}^{1+ \alpha, 1}(T)$) in Theorems \ref{thm0111_01} and \ref{cor0111_01}.

\begin{remark}
	\label{rmk_230615}

\mbox{}
	
(i)
An unweighted version (i.e., $w = 1$) of Theorem \ref{thm_subdiff} (and Lemma \ref{lem0201}, Proposition \ref{prop0201}) can be found in \cite{MR4242958}. However, one of their assumptions must be restricted.
More precisely, the authors of \cite{MR4242958} claim that their results hold for all $\alpha \in (0, 1)$, but in fact, it is only possible for, in their setting, $\alpha \in (1/q, 1)$.
See (ii) below for more discussion about this.
On the other hand, with respect to the spatial variables, \cite{MR4242958} deals with more general operators than $\Delta$, so-called $\phi(\Delta)$.

(ii) One may try to obtain \eqref{eq0202_03} and \eqref{eq0224_01} for $\alpha \in (0, (1+\mu)/q]$. However, such estimates cannot be established if $\alpha \in (0, (1+\mu)/q)$.
Indeed, suppose that \eqref{eq0202_03} holds for $\alpha \in (0, (1+\mu)/q)$. If $u_0 \in C_0^{\infty}(\bR^d)$, then $U(t, \cdot) = P_{\alpha}(t, \cdot) \ast u_0(\cdot)$ is a classical solution, that is, sufficiently smooth solution.
Then for any $u_0 \in C_0^{\infty}(\bR^d)$, we must have
$$
\|u_0 \|_{L_{p, w_2}} \le N \| U \|_{\bH_{p, q, w}^{\alpha, 2}(T)} \le N  \| u_0 \|_{B_{p, q, w_2}^{2\theta}},
$$
 where $N$ is independent of $u_0$ and the first inequality is due to Lemma \ref{lem0517_2}.
This gives a contradiction since $\theta < 0$ for $\alpha \in (0, (1+\mu)/q)$. On the other hand, if $\alpha = (1+\mu)/q$, i.e., $\theta = 0$, the situation is more delicate. It is because, for example, the inequality $\|u_0 \|_{L_p}  \le N  \| u_0 \|_{B_{p, q}^{0}}$ may or may not hold depending on the relationship between the numbers $p$, $q$, and $2$.

(iii) In \cite{MR4186022} and \cite{MR4464544}, the authors proved the solvability of the homogeneous equations \eqref{eq0203_01} and \eqref{eq0203_02} for given initial values. They assumed $u_0 \in \cX_0$ (and $u_1 \in \cX_1$ for the $\beta \in (1, 2)$ case), but in view of maximal regularity theory, $\cX_0$ and $\cX_1$ are not appropriate initial spaces.
For example, if $q\alpha = 2(1+\mu)$, the initial value $u_0$ considered in \cite{MR4186022} is $u_0 \in \cX_0 = B_{p, q, w_2}^{1 + \varepsilon} \subsetneq B_{p, q, w_2}^{1}$ ($\varepsilon > 0$), 
but by Theorem \ref{thm_subdiff} it is possible to take functions from $B_{p,q,w_2}^1$ as initial values.
Furthermore, when it comes to the non-trivial weighted case, i.e., $w_2 \neq 1$, 
it can be observed that, in general, the set $\cX_0$ is strictly included in $B_{p, q, w_2}^{2 - \frac{2(1+\mu)}{q \alpha}}$ even when $q\alpha \neq 2(1+\mu)$ (in fact, $\cX_0$ is too small), which means that the solvability results with non-zero initial conditions in  \cite{MR4186022} and \cite{MR4464544} are not sufficiently general. 
In light of this, our Theorems \ref{thm_subdiff} and \ref{thm_supdiff} address and bridge this gap.
\end{remark}

\section{Proof of Trace and Extension theorem}
\label{sec_trace}

\subsection{Proof of Theorems \ref{thm0111_01} and \ref{cor0111_01}}
In this subsection, we prove the (initial) trace embedding for $\bH_{p, q, w}^{k+\alpha, 2}(\Omega_T)$ and $\cH_{p, q, w}^{k+\alpha, 1}(\Omega_T)$.
We fix a non-negative integer $k$ and $\alpha \in (0, 1)$. Recall that $w(t, x) = w_1(t)w_2(x) = t^{\mu}w_2(x)$ where $\mu \in (-1, q-1)$ and $w_2 \in A_p(\bR^d)$.
 
To prove Theorem \ref{thm0111_01}, we present an integral representation of $u \in C_0^{\infty}([0, \infty) \times \Omega)$ satisfying
\begin{equation}
	\label{eq230717_1}
\partial_t^{k+\alpha}u(t, x) = \partial_t^{k+1}I^{1-\alpha}\left( u(t, x) - \sum_{m=0}^k \frac{t^m}{m!}\partial_t^mu(0, x)\right)= f(t, x)
\end{equation}
(see Remark \ref{rmk0102_02}).
Take a non-negative $\eta = \eta(t) \in C^{\infty}(\bR)$ such that $\operatorname{supp}\eta \subset (-1, -1/2)$ and $\|\eta\|_{L_1(\bR)} = 1$.
For $\varepsilon > 0$, set $ \eta_{\varepsilon}(t) := \varepsilon^{-1/\beta} \eta(t\varepsilon^{-1/\beta})$ where $\beta := k + \alpha$.
Then for $(t, x) \in [0, \infty) \times \Omega$ and
\[
u^{(\varepsilon)}(t,x) := \int_{\bR_+ } \eta_\varepsilon(t-s) u(s,x) \, \mathrm{d}s,
\]
we write
$$
u(t, x) = u^{(\varepsilon)}(t, x) - \left( u^{(\varepsilon)}(t, x) - u(t, x) \right)
$$
$$
 = u^{(\varepsilon)}(t, x) - \int_{\bR_+} \left( u(s, x) - u(t, x)\right) \eta_{\varepsilon}(t-s)\,\mathrm{d}s
=: u^{(\varepsilon)}(t, x) - v(t, x; \varepsilon).
$$
Let $\gamma(\lambda)$ be a path from $t \in \bR_+ $ to $s \in \bR_+ $ defined by
$$
\gamma(\lambda) = (1-\lambda^{1/\beta})t + \lambda^{1/\beta} s, \quad \lambda \in [0,1].
$$
Since
$$
u(s, x) - u(t, x) = \int_0^1 (\nabla_{t, x} u)(\gamma(\lambda)) \cdot  \gamma'(\lambda) \, \mathrm{d}\lambda,
$$
we see that
\[
u(s, x) - u(t, x) = u\left(\gamma(1), x\right) - u\left( \gamma (0), x \right)  =  \int_0^1 u_t \left( \gamma \left( \lambda \right), x\right) \gamma'(\lambda)\,\mathrm{d}\lambda
\]
with
\[
\gamma'(\lambda) = \frac{\mathrm{d}}{\mathrm{d}\lambda}\left( \lambda^{1/\beta}(s-t) + t \right) = \frac{1}{\beta}\lambda^{1/\beta - 1}(s-t).
\]
Then,
\[
v(t, x) = v(t, x;\varepsilon) 
= \frac{1}{\beta}\int_{\bR_+}\int_0^1 u_s(\gamma(\lambda), x) \lambda^{ 1/\beta - 1} (s-t) \, \mathrm{d}\lambda \,   \eta_{\varepsilon}(t-s)\,\mathrm{d}s
\]
\[
 = \frac{1}{\beta} \int_0^{\varepsilon}  \lambda^{-1} \int_{\bR_+ } u_s(s, x) \,  \zeta \left( \frac{t-s}{\lambda^{1/\beta}} \right)\,\mathrm{d}s \, \mathrm{d}\lambda,
\]
where $\zeta(t) = -t \eta(t)$ and the last equality is due to the change of variables $\gamma(\lambda) \to s$, and then $\varepsilon \lambda \to \lambda$.
For a non-negative integer $n$ such that $n \le k$, by taking $n$-times differentiation with respect to $t$ and by integration by parts with the fact that $\operatorname{supp} \zeta \subset (-1, -1/2)$, we have
\begin{equation*}
\partial_t^n u (t, x) = \partial_t^n u^{(\varepsilon)}(t, x) -
\frac{1}{\beta} \int_0^{\varepsilon} \lambda^{-1}\int_{\bR_+ }   \partial_s^{n+1}u(s, x) \,  \zeta \left( \frac{t-s}{\lambda^{1/\beta}} \right)\,\mathrm{d}s \, \mathrm{d}\lambda
\end{equation*}
for $(t, x) \in [0, \infty) \times \Omega$.
Using the fact that $I^{1-\alpha} \partial_t g = \partial_t I^{1-\alpha} g$ for a sufficiently smooth $g$ with $g(0,x) = 0$ and \eqref{eq230717_1}, we have
\[
 \int_{\bR_+ } \partial_s^{n+1}u(s, y) \,  \zeta \left( \frac{t-s}{\lambda^{1/\beta}} \right)\,\mathrm{d}s = \int_{t+\frac{1}{2}\lambda^{1/\beta}}^{t+ \lambda^{1/\beta}} \partial_s I^{\alpha} I^{1-\alpha} \partial_s^{n+1}  u(s, x)  \zeta \left( \frac{t-s}{\lambda^{1/\beta}} \right)\,\mathrm{d}s
\]
\[
=\lambda^{-1/\beta} \int_{t + \frac{1}{2}\lambda^{1/\beta}}^{t + \lambda^{1/\beta}} I^{\alpha}\partial_s^{n+1}I^{1-\alpha} \left( u(s, x) - \sum_{m=0}^n \frac{s^m}{m!}\partial_s^m u (0, x) \right)   \zeta' \left( \frac{t-s}{\lambda^{1/\beta}} \right)\,\mathrm{d}s
\]
\[
= \lambda^{-1/\beta} \int_{t + \frac{1}{2}\lambda^{1/\beta}}^{t + \lambda^{1/\beta}}
\left( I^{\alpha}I^{k-n}f \left(s, x\right) + R_{n}\left(s, x\right) \right)
  \zeta' \left( \frac{t-s}{\lambda^{1/\beta}} \right)\,\mathrm{d}s, \quad (I^0f := f)
\]
where $R_{n}(s, x) = 0$ for $n = k$ and
\begin{equation}
	\label{eq230717_2}
R_{n}(s, x) = \sum_{m = n+1}^k N(m, n) s^{m-n} \partial_t^mu(0, x)
\end{equation}
for $0 \le n < k$.
Therefore, by evaluating $t = 0$ in the above, we have the following representation of $\partial_t^n u(0, x)$:
\begin{multline}
\label{equ0101_01}
\partial_t^n u(0, x) = \partial_t^n u^{(\varepsilon)}(0, x)
\\
- \frac{1}{\beta} \int_0^{\varepsilon}  \lambda^{-1-1/\beta} \int_{ \frac{1}{2}\lambda^{1/\beta}}^{ \lambda^{1/\beta}}  \left( I^{\alpha}I^{k-n}f \left(s, x\right) + R_{n}\left(s, x\right) \right)  \zeta' \left( \frac{-s}{\lambda^{1/\beta}} \right) \,\mathrm{d}s \, \mathrm{d}\lambda.
\end{multline}
Moreover, \eqref{equ0101_01} also holds for $u \in C_0^{\infty}([0, T] \times \Omega)$ by extending $u$ to $\{t > T\} \times \Omega$ properly and following the above argument with $\varepsilon < T^{\beta}$.

\begin{proof}[\textbf{Proof of Theorem \ref{thm0111_01}}]
By Proposition \ref{prop0102_03}, it is enough to prove \eqref{trace0110_1} for $u \in C_0^{\infty}([0, T] \times \bR^d)$ with $\partial_t^{k+\alpha}u = f$.
If we have $\partial_t^n u(0, \cdot) = U_{n, 0}(\varepsilon, x) + U_{n, 1}(\varepsilon, x)$ with sufficiently smooth $U_{n, 0}$ and $U_{n, 1}$, then by \eqref{eq2306080355} in Remark \ref{rmk_interpol} with $\sigma(x) = \partial_t^n u(0, \cdot)$, $s_0 = 0$ and $s_1 = 2$, it holds that
\[
\| \partial_t^n u(0, \cdot) \|_{B_{p, q, w_2}^{2\theta_n}} \le N \left( \int_0^{\infty} \left( \varepsilon^{-\theta_n} \| U_{n, 0}(\varepsilon, \cdot) \|_{L_{p, w_2}} \right)^q \, \frac{\mathrm{d}\varepsilon}{\varepsilon} \right)^{1/q}
\]
\[
+ N \left( \int_0^{\infty} \left( \varepsilon^{1-\theta_n} \| U_{n, 1}(\varepsilon, \cdot) \|_{H_{p, w_2}^2} \right)^q \, \frac{\mathrm{d}\varepsilon}{\varepsilon} \right)^{1/q} := N(A_n^{1/q} + B_n^{1/q}),
\]
where $N = N(d, p, [w_2]_{A_p}) > 0$.
It means that to obtain the trace estimate \eqref{trace0110_1}, it suffices to find such a proper decomposition $\partial_t^n u(0, x) = U_{n, 0}(\varepsilon, x) + U_{n, 1}(\varepsilon, x)$, and then, obtain estimates for $A_n$ and $B_n$.
We first prove the case where $n = k$.

\textit{$\bullet$ Step 1 -- Decomposition of $\partial_t^ku(0, x)$}:
For $x \in \bR^d$, we define 
\begin{multline*}
U_{k, 0}(\varepsilon, x) :=
\\
\left\{
\begin{array}{lll}
\displaystyle - \frac{1}{\beta} \int_0^{\varepsilon}  \lambda^{-1-1/\beta} \int_{ \frac{1}{2}\lambda^{1/\beta}}^{ \lambda^{1/\beta}} I^{\alpha}f \left(s, x\right)    \zeta' \left( \frac{-\cdot}{\lambda^{1/\beta}} \right) \,\mathrm{d}s \, \mathrm{d}\lambda \, &\displaystyle\textrm{if} \,\,\, \varepsilon < T^{\beta},
\\
\displaystyle \partial_t^k u(0, x) \quad &\displaystyle\textrm{if} \,\,\, \varepsilon \ge T^{\beta}, 
\end{array}
\right.
\end{multline*}
and
\begin{equation*}
U_{k, 1}(\varepsilon, x) :=							
 \left\{
\begin{array}{lll}
\displaystyle \partial_t^k u^{(\varepsilon)}(0, x) \quad &\displaystyle\textrm{if} \quad \varepsilon < T^{\beta},
\\
\displaystyle 0 \quad &\displaystyle\textrm{if} \quad \varepsilon \ge T^{\beta}. 
\end{array}
\right.
\end{equation*}
Then by \eqref{equ0101_01}, we obtain a decomposition of $\partial_t^nu(0, x)$, that is,
\begin{equation}
\label{eq0129_01}
\partial_t^k u(0, x) = U_{k, 0}(\varepsilon, x) + U_{k, 1}(\varepsilon, x)
\end{equation}
for any $(\varepsilon, x) \in (0, \infty) \times \bR^d$.

\textit{$\bullet$ Step 2 -- Estimates of $A_k$ and $B_k$ ($n = k$)}: Let
\[
A_k = \int_0^{\infty} \varepsilon^{-1-q\theta_k} \| U_{k, 0}(\varepsilon, \cdot) \|_{L_{p, w_2}}^q \,\mathrm{d}\varepsilon 
= \int_0^{ T^{\beta}} \ldots + \int_{ T^{\beta}}^{\infty} \ldots := A_{k, 1} + A_{k, 2}.
\]
Since $U_0(\varepsilon, x) = \partial_t^ku(0, x)$ for $\varepsilon \ge  T^{\beta}$, it is clear that
\[
A_{k, 2} \le N(q, \theta) T^{- \beta q \theta_k} \| \partial_t^ku(0, \cdot) \|_{L_{p, w_2}}^q.
\]
 Note that for $\varepsilon <  T^{\beta}$, by Minkowski's inequality and Fubini's theorem, we have
\[
 \| U_{k, 0}(\varepsilon, \cdot) \|_{L_{p, w_2}} \le N \int_0^{\varepsilon}  \lambda^{-1-1/\beta} \int_{\frac{1}{2}\lambda^{1/\beta}}^{\lambda^{1/\beta}} \| I^{\alpha}f(s, \cdot)\|_{L_{p, w_2}}\,\mathrm{d}s \, \mathrm{d}\lambda
\]
\[
\le N \int_0^{\varepsilon^{1/\beta}} \left( \int_{s^{\beta}}^{(2s)^{\beta}} \lambda^{-1-1/\beta}\,\mathrm{d}\lambda \right)\| I^{\alpha}f(s, \cdot)\|_{L_{p, w_2}}\,\mathrm{d}s
\]
\[
\le N(\beta) \int_0^{\varepsilon^{1/\beta}} s^{-1}  \| I^{\alpha}f(s, \cdot)\|_{L_{p, w_2}}\,\mathrm{d}s.
\]
Then by the change of variable $\varepsilon^{1/\beta} \to \varepsilon$,
\[
A_{k, 1} \le N\int_0^{T^{\beta}} \varepsilon^{-1-q\theta_k} \left(\int_0^{\varepsilon^{1/\beta}} s^{-1} \| I^{\alpha}f(s, \cdot)\|_{L_{p, w_2}}  \,\mathrm{d}s \right)^q \,\mathrm{d}\varepsilon
\]
\[
= N \int_0^{T} \varepsilon^{-1 -\beta q \theta_k} \left(\int_0^{\varepsilon} s^{-1} \| I^{\alpha}f(s, \cdot)\|_{L_{p, w_2}}  \,\mathrm{d}s \right)^q \,\mathrm{d}\varepsilon
\]
\[
= N \int_0^{T} \varepsilon^{-q} \left(\int_0^{\varepsilon} s^{-1} \| I^{\alpha}f(s, \cdot)\|_{L_{p, w_2}}  \,\mathrm{d}s \right)^q \varepsilon^{q -1 -\beta q \theta_k}  \,\mathrm{d}\varepsilon
\]
\[
\le N \int_0^{T} \varepsilon^{-1 -\beta q \theta_k } \| I^{\alpha}f(\varepsilon, \cdot)\|_{L_{p, w_2}}^q \,\mathrm{d}\varepsilon,
\]
where the last inequality is due to \eqref{eq0224_03} of Lemma \ref{lem0224} ($\alpha = 1$ in there) with $q -1 -\beta q \theta < q-1$.
Since $-1 - \beta q \theta_k = -q\alpha + \mu$, we have
\[
A_{k, 1} \le N(\alpha, q, \mu, k)\int_0^{T} \varepsilon^{ - q \alpha + \mu} \| I^{\alpha}f(\varepsilon, \cdot)\|_{L_{p, w_2}}^q \,\mathrm{d}\varepsilon.
\]
Then by Minkowski's inequality and \eqref{eq0224_03} again, the above integral is bounded by a constant times
\[
 \int_0^{T} \| f(s, \cdot) \|_{L_{p, w_2}}^q s^{\mu}\,\mathrm{d}s,
\]
and hence,
\[
A_k^{1/q} \le N \| f \|_{L_{p, q, w}(T)} + NT^{-\theta_k(k + \alpha)} \| \partial_t^n u(0, \cdot) \|_{L_{p, w_2}},
\]
where $N = N(\alpha, q, \mu, k) > 0$.
For $B_k$, note that
\[
B_k = \int_0^{\infty} \varepsilon^{-1 + q(1-\theta_k)} \| U_{k, 1}(\varepsilon, \cdot) \|_{H_{p, w_2}^2}^q \,\mathrm{d}\varepsilon = \int_0^{T^{\beta}} \varepsilon^{-1 + q(1-\theta_k)} \| U_{k, 1}(\varepsilon, \cdot) \|_{H_{p, w_2}^2}^q \,\mathrm{d}\varepsilon.
\]
Since
\[
\partial_t^k u^{(\varepsilon)}(0, x) = \varepsilon^{-(k+1)/\beta} \int_{\frac{1}{2}\varepsilon^{1/\beta}}^{\varepsilon^{1/\beta}} u(s, x) \partial_t^k\eta \left(\frac{-s}{\varepsilon^{1/\beta}} \right)\,\mathrm{d}s,
\]
we have
\[
\| U_1(\varepsilon, \cdot) \|_{H_{p, w_2}^2} \le N(k) \varepsilon^{-(k+1)/\beta} \int_{\frac{1}{2}\varepsilon^{1/\beta}}^{\varepsilon^{1/\beta}} \| u(s, \cdot) \|_{H_{p, w_2}^2}\,\mathrm{d}s,
\]
and by the change of variable $\varepsilon^{1/\beta} \to \varepsilon$ and \eqref{eq0224_03},
\[
B_k \le \int_0^{T^{\beta}} \varepsilon^{-1 + q(1-\theta_k)} \| U_1(\varepsilon, \cdot) \|_{H_{p, w_2}^2}^q \,\mathrm{d}\varepsilon \le N \int_0^{T} \varepsilon^{-q + \mu} \left( \int_{\frac{1}{2}\varepsilon}^{\varepsilon}  \| u(s, \cdot) \|_{H_{p, w_2}^2}\,\mathrm{d}s \right)^{q} \,\mathrm{d}\varepsilon
\]
\[
\le N \int_0^{T} \varepsilon^{ \mu}   \| u(\varepsilon, \cdot) \|_{H_{p, w_2}^2}^q \,\mathrm{d}\varepsilon = N \| |u| + |Du| + |D^2u|  \|_{L_{p, q, w}(T)}^q,
\]
where $N = N(\alpha, q, \mu, k) > 0$.
By combining the estimates for $A_k$ and $B_k$, we have
\begin{multline*}
\| \partial_t^k u(0, \cdot) \|_{B_{p, q, w_2}^{2\theta_k}} \le N \left\|  \left|f\right| +  \left|u \right| + \left| Du \right| + \left| D^2u \right| \right\|_{L_{p, q, w}(T)} 
\\
+ NT^{-\theta_k(k + \alpha)} \| \partial_t^k u(0, \cdot) \|_{L_{p, w_2}},
\end{multline*}
where we call that $f = \partial_t^{k+\alpha}u$.
Since $\partial_t^ku \in \bH_{p, q, w}^{\alpha, 0}(T)$, by applying \eqref{eq0513_03} in Lemma \ref{lem0525_1} to $\partial_t^ku(0, \cdot)$, 
\begin{multline}
	\label{eq2306281606}
\| \partial_t^k u(0, \cdot) \|_{B_{p, q, w_2}^{2\theta_k}} \le  N \| |f| +  |u| + |Du| + |D^2u| \|_{L_{p, q, w}(T)}
\\ + NT^{-\alpha}\| \partial_t^ku \|_{L_{p, q, w}(T)},
\end{multline}
where $N = N(d, \alpha, p, q, [w_2]_{A_p}, \mu, k) > 0$ and the last term in \eqref{eq2306281606} vanishes as $T \to \infty$.

\textit{$\bullet$ Step 3 -- Remaining case ($0 \le n < k$)}: 
Recall that
\[
R_n(s, x) = \sum_{m = n+1}^k N(m, n) s^{m-n} \partial_t^mu(0, x)
\]
and note that
\begin{equation}
	\label{eq2308010237}
\int_0^{\varepsilon}  \lambda^{-1-1/\beta} \int_{ \frac{1}{2}\lambda^{1/\beta}}^{ \lambda^{1/\beta}} R_n(s, x)     \zeta' \left( \frac{-s}{\lambda^{1/\beta}} \right) \,\mathrm{d}s  \mathrm{d}\lambda = \sum_{m=n+1}^k N(\beta, m, n)\varepsilon^{(m-n)/\beta} \partial_t^mu(0, x).
\end{equation}
To employ an induction argument, assume that, for any $m \in \{n+1, n+2, \ldots, k\}$, we have $U_{m, 0}(\varepsilon, x)$ and $U_{m, 1}(\varepsilon, x)$ such that $\partial_t^mu(0, x) = U_{m, 0}(\varepsilon, x) + U_{m, 1}(\varepsilon, x)$ and
\begin{multline}
	\label{eq230803_01}
 \int_0^{\infty} \left( \varepsilon^{-\theta_m} \| U_{m, 0}(\varepsilon, \cdot) \|_{L_{p, w_2}} \right)^q \, \frac{\mathrm{d}\varepsilon}{\varepsilon} 
+  \int_0^{\infty} \left( \varepsilon^{1-\theta_m} \| U_{m, 1}(\varepsilon, \cdot) \|_{H_{p, w_2}^2} \right)^q \, \frac{\mathrm{d}\varepsilon}{\varepsilon} 
\\
\le N  \| u \|_{ \bH_{p, q, w}^{k+\alpha, 2}(T)}^q,
\end{multline}
where $N = N(d, \alpha, p, q, [w_2]_{A_p}, \mu, k, m, T) \to N(d, \alpha, p, q, [w_2]_{A_p}, \mu, k, m)$ as $T \to \infty$.
By the proof above (\textit{Step 1} and \textit{Step 2}) we know that \eqref{eq230803_01} holds for $n = k$.
We now prove \eqref{eq230803_01} for $n < k$.
For $\varepsilon < T^{\beta}$, we set
\[
U_{n, 0}(\varepsilon, x) := - \frac{1}{\beta} \int_0^{\varepsilon}  \lambda^{-1-1/\beta} \int_{ \frac{1}{2}\lambda^{1/\beta}}^{ \lambda^{1/\beta}} I^{k-n}I^{\alpha}f(s, x)    \zeta' \left( \frac{-s}{\lambda^{1/\beta}} \right) \,\mathrm{d}s \, \mathrm{d}\lambda
\]
\[
 + \sum_{m=n+1}^k N(\beta, m, n) \varepsilon^{(m-n)/\beta}U_{m, 0}(\varepsilon, x)
\]
and
\[
U_{n, 1}(\varepsilon, x) := \partial_t^n u^{(\varepsilon)}(0, x) + \sum_{m=n+1}^k N(\beta, m, n)\varepsilon^{(m-n)/\beta}U_{m, 1}(\varepsilon, x),
\]
where $N(\beta, m, n)$ is the same as in \eqref{eq2308010237}, and note that $I^{k-n} I^\alpha = I^\alpha I^{k-n}$ and $U_{m,0}(\varepsilon,x) + U_{m,1}(\varepsilon,x) = \partial_t^m u(0,x)$ by the induction assumption for $m \geq n+1$.
For $\varepsilon \ge T^{\beta}$, we set $U_{n, 0}(\varepsilon, x) := \partial_t^nu(0, x)$ and $U_{n, 1}(\varepsilon, x) = 0$.
Then by \eqref{equ0101_01} and \eqref{eq2308010237}, we have a decomposition
\[
\partial_t^n u(0, x) = U_{n, 0}(\varepsilon, x) + U_{n, 1}(\varepsilon, x).
\]
By following \textit{Step 1} and \textit{Step 2}, we obtain
\begin{equation}
	\label{eq2306281607}
\| \partial_t^n u(0, \cdot) \|_{B_{p, q, w_2}^{2\theta_n}} \le  N \| u \|_{\bH_{p, q, w}^{k+\alpha, 2}(T)},
\end{equation}
where $N = N(d, \alpha, p, q, [w_2]_{A_p}, \mu, k, n, T) \to N(d, \alpha, p, q, [w_2]_{A_p}, \mu, k, n)$ as $T \to \infty$.
In particular, by the same calculation as in \textit{Step 2}, we reach
\[
 \int_0^{T^{\beta}} \left( \varepsilon^{-\theta_n} \| U_{n, 0}(\varepsilon, \cdot) \|_{L_{p, w_2}} \right)^q \, \frac{\mathrm{d}\varepsilon}{\varepsilon} \le N \int_0^T \varepsilon^{-1-\beta q \theta_n} \left( I^{k-n}I^{\alpha}\|f(\varepsilon, \cdot)\|_{L_{p, w_2}} \right)^q \,\mathrm{d}\varepsilon
\]
\[
+ N \sum_{m = n+1}^k  \int_0^{\infty}  \varepsilon^{-q\theta_n - 1}  \varepsilon^{q(m-n)/\beta}\| U_{m, 0}(\varepsilon, \cdot) \|_{L_{p, w_2}}^q \, \mathrm{d}\varepsilon.
\]
By applying Lemma \ref{lem0224} (to $I^{\alpha}\|f(\varepsilon, \cdot)\|_{L_{p, w_2}}$) $(k-n)$-times and using a relation $-\theta_n + (m-n)/\beta = -\theta_m$, the right-hand side of the above inequality is bounded by a constant times
\[
\int_0^T \varepsilon^{-q\alpha + \mu} \left( I^{\alpha}\|f(\varepsilon, \cdot)\|_{L_{p, w_2}} \right)^q \,\mathrm{d}\varepsilon +  \sum_{m = n+1}^k  \int_0^{\infty}  \varepsilon^{-q\theta_m - 1} \| U_{m, 0}(\varepsilon, \cdot) \|_{L_{p, w_2}}^q \, \mathrm{d}\varepsilon,
\]
\[
\le  N\| u \|_{\bH_{p, q, w}^{k+\alpha, 2}(T)}^q,
\]
where the last inequality is due to Lemma \ref{lem0224} and the induction assumption \eqref{eq230803_01} for $m=n+1,\ldots,k$.
Similarly, we obtain
\[
\int_0^{T^{\beta}} \varepsilon^{-1 + q(1-\theta_n)} \| U_{n, 1}(\varepsilon, \cdot) \|_{H_{p, w_2}^2}^q \,\mathrm{d}\varepsilon \le   N\| u \|_{\bH_{p, q, w}^{k+\alpha, 2}(T)}^q.
\]
Finally, by \eqref{eq230604_1} in Remark \ref{rmk230604_1} we have
\[
T^{(1+\mu)/q}\| \partial_t^n u(0, \cdot) \|_{L_{p, w_2}} \le N(q, \mu) \left( \|  \partial_t^n u \|_{L_{p, q, w}(T)} + T\|  \partial_t^{n+1} u \|_{L_{p, q, w}(T)}  \right),
\]
and then, 
\[
 \int_{T^{\beta}}^{\infty} \varepsilon^{-1-q\theta_n} \| U_{n, 0}(\varepsilon, \cdot) \|_{L_{p, w_2}}^q \,\mathrm{d}\varepsilon +  \int_{T^{\beta}}^{\infty} \varepsilon^{-1-q(1-\theta_n)} \| U_{n, 1}(\varepsilon, \cdot) \|_{H_{p, w_2}^2}^q \,\mathrm{d}\varepsilon
\]
\[
\le N\| u \|_{\bH_{p, q, w}^{k+\alpha, 2}(T)}^q,
\]
since $U_{n, 1}(\varepsilon, x) = 0$ for $\varepsilon \ge T^{\beta}$.
The theorem is proved.
\end{proof}

\begin{proof}[\textbf{Proof of Theorem \ref{cor0111_01}}]
By Proposition \ref{prop0102_03}, it is enough to prove \eqref{trace0110_3} for $u \in C_0^{\infty}([0, T] \times \bR^d)$.
For a given sufficiently smooth $u \in \cH_{p, q, w}^{k+\alpha, 1}(T)  $, by direct computation one verifies that $v:= (1 - \Delta)^{-1/2}u \in  \bH_{p, q, w}^{k+\alpha, 2}(T)$. 
Then by applying Theorem \ref{thm0111_01} to $v$, we have \eqref{trace0110_1} with $v$ in place of $u$.
Since $(1 - \Delta)^{-1/2}$ is an isometry from $B_{p, q, w_2}^{2\theta-1}$ to $B_{p, q, w_2}^{2\theta}$ (Remark \ref{rmk_interpol}), it holds that for $n = 0, 1, \ldots, k$,
\[
\|\cT_n v \|_{B_{p, q, w_2}^{2\theta}} = \| (1 - \Delta)^{-1/2}\cT_n u \|_{B_{p, q, w_2}^{2\theta}} = \|\cT_n u \|_{B_{p, q, w_2}^{2\theta - 1}}.
\]
Similarly, it also holds that $\| \partial_t^k v \|_{\bH_{p, q, w}^{\alpha, 0}(T)} =  \|  \partial_t^k u \|_{\cH_{p, q, w}^{\alpha, -1}(T)}$, $\|\partial_t^n v \|_{L_{p, q, w}(T)} = \|  \partial_t^n u \|_{\bH_{p, q, w}^{-1}(T)} $ for $n < k$, and $\| |v| + |Dv| + |D^2v| \|_{L_{p, q, w}(T)} = \| |u| + |Du| \|_{L_{p, q, w}(T)}$.
From the estimate \eqref{trace0110_1} for $v \in \bH_{p, q, w}^{k+\alpha, 2}(T)$ with the above relations, we obtain \eqref{trace0110_3} for $u \in \cH_{p, q, w}^{k+\alpha, 1}(T)$.
The theorem is proved.
\end{proof}

\begin{remark}
		\label{rmk230416}
In Theorems \ref{thm0111_01} and \ref{cor0111_01}, we prove that the initial traces only when the spatial domain $\Omega$ is $\bR^d$.
This is because the characterization of $B_{p, q, w_2}^s$ as an interpolation of $H_{p, w_2}^{s'}$, and the isomorphism of the operator $(1 - \Delta)^{\nu/2}$ on $B_{p, q, w_2}^s$ (and on $H_{p, w_2}^{s'}$) are clear when $\Omega = \bR^d$.
However, by following the proof of Theorem \ref{thm0111_01} line by line, one can consider $(X_0, X_1)$ instead of $(L_{p, w_2}(\bR^d), H_{p, w_2}^2(\bR^d))$ where $X_0$ and $X_1$ are Banach spaces which are continuously embedded in the same Hausdorff topological vector space $Z$.
For instance, if we take $(X_0, X_1) = (L_{p, w_2}(\Omega), H_{p, w_2}^2(\Omega))$ and $(X_0, X_1) = (H_{p, w_2}^{-1}(\Omega), H_{p, w_2}^1(\Omega))$ for general $\Omega \subset \bR^d$, we have Theorem \ref{thm230416} below.
One may also take homogeneous Bessel potential spaces for $X_0$ and $X_1$, for which the initial trace spaces are also homogeneous Besov spaces.
\end{remark}

Note that in the following theorem, if $w=1$ and $\partial\Omega$ is sufficiently smooth, the initial trace spaces are well-known Besov spaces (see \cite[p.204]{MR0781540}):
\[
\left( L_{p, w_2} \left( \Omega \right), H_{p, w_2}^2 \left( \Omega \right) \right)_{\theta_n, q} = B_{p, q}^{2\theta_n}(\Omega), \quad \left( H_{p, w_2}^{-1} \left( \Omega \right), H_{p, w_2}^1 \left( \Omega \right) \right)_{\theta_n, q} = B_{p, q}^{2\theta_n - 1}(\Omega).
\]
\begin{theorem}
		\label{thm230416}		
Let $\Omega$ be a domain in $\bR^d$, $\alpha \in (0, 1)$, $p, q \in (1, \infty)$, and $w(t, x) = w_1(t)w_2(x) = t^{\mu}w_2(x)$, where $\mu \in (-1, q-1)$ and $w_2 \in A_p(\bR^d)$.
Also, let $k$ and $n$ be non-negative integers with $n \le k$ and $\theta_n:= (k+\alpha-n-\frac{1+\mu}{q})/(k+\alpha)$.
If $\alpha > (1+\mu)/q$ (so that $\theta_n \in (0,1)$ for all $n=0,1,\ldots,k$), for each $n \in \{0,1,\ldots,k\}$,
 we have the following:
\begin{enumerate}[(i)]

\item The operator $\cT_n: \bH_{p, q, w}^{k+\alpha, 2}(\Omega_T) \to \left( L_{p, w_2} \left( \Omega \right), H_{p, w_2}^2 \left( \Omega \right) \right)_{\theta_n, q}$ with $\cT_nu = \partial_t^nu(0, x)$ for $u \in \bH_{p, q, w}^{k+\alpha, 2}(\Omega_T) \cap C^{\infty}([0, T] \times \Omega)$ is bounded and satisfies
\begin{equation}
	\label{eq2306281612}
\|\cT_n u \|_{\left( L_{p, w_2} \left( \Omega \right), H_{p, w_2}^2 \left( \Omega \right) \right)_{\theta_n, q}} \le N \| u \|_{ \bH_{p, q, w}^{k+\alpha, 2}(\Omega_T)},
\end{equation}
where $N = N(\alpha, q, \mu, k, n, T) > 0$.

\item The operator $\cT_n: \cH_{p, q, w}^{k+\alpha, 1}(\Omega_T) \to \left( H_{p, w_2}^{-1} \left( \Omega \right), H_{p, w_2}^1 \left( \Omega \right) \right)_{\theta_n, q}$ with $\cT_nu = \partial_t^nu(0, x)$ for $u \in \cH_{p, q, w}^{k+\alpha, 1}(\Omega_T) \cap C^{\infty}([0, T] \times \Omega)$ is bounded and satisfies
\begin{equation}
	\label{eq2306281613}
\|\cT_n u \|_{\left( H_{p, w_2}^{-1} \left( \Omega \right), H_{p, w_2}^1 \left( \Omega \right) \right)_{\theta_n, q}} \le N \| u \|_{ \cH_{p, q, w}^{k+\alpha, 1}(\Omega_T)},
\end{equation}
where $N = N(\alpha, q, \mu, k, n, T) > 0$.
\end{enumerate}
Moreover, $N$ in \eqref{eq2306281612} and \eqref{eq2306281613} are uniformly bounded for $T \ge 1$.
\end{theorem}

Indeed, to obtain ($ii$), we need to estimate $\|U_0(\varepsilon, \cdot) \|_{H_{p, w_2}^{-1}(\Omega)}$, which follows from \eqref{equ0101_01}, \eqref{eq0129_01}, and
\[
-(k+\alpha) U_0(\varepsilon, x) = \int_0^{\varepsilon}  \lambda^{-1-1/\beta} \int_{\frac{1}{2}\lambda^{1/\beta}}^{\lambda^{1/\beta}} I^{\alpha}f(s, x) \,  I^{k-n}\left[ \zeta' \left( \frac{-\cdot}{\lambda^{1/\beta}} \right) \right]\,\mathrm{d}s \, \mathrm{d}\lambda
\]
\[
- D_i \int_0^{\varepsilon}  \lambda^{-1-1/\beta} \int_{\frac{1}{2}\lambda^{1/\beta}}^{\lambda^{1/\beta}} I^{\alpha}g_i(s, x) \, I^{\alpha}f(s, x) \,  I^{k-n}\left[ \zeta' \left( \frac{-\cdot}{\lambda^{1/\beta}} \right) \right]\,\mathrm{d}s \, \mathrm{d}\lambda,
\]
where $\partial_t^{k+\alpha}u = D_ig_i + f$ for some $g_i, f \in L_{p, q, w}(\Omega_T)$, $i=1, \ldots, d$.
To estimate $\|\partial_t^n u(0, \cdot) \|_{H_{p, w_2}^{-1}(\Omega)}$, we use Lemma \ref{lem1201_1}.

\begin{remark}
	\label{rmk_unweighted}
If $k = 0$ and there is no spatial weight, i.e., $w_2(x) = 1$, one can obtain the trace estimates \eqref{trace0110_1} and \eqref{trace0110_3} by calculating the well-known integral representations of the $B_{p, q}^{\nu}$-norm for $\nu > 0$.
Indeed, since there is no weight with respect to the $x$ variables (i.e., translation invariant with respect to $x$), one can use the following:
\begin{equation}
\label{eq2306291250}
\| f \|_{B_{p, q}^{\nu}} \simeq_{d, p} \|f\|_{L_p} + \left( \int_{\bR^d} |h|^{-\nu q} \| \Delta_h^2 f \|_{L_p}^q\,\frac{\mathrm{d}h}{|h|^d} \right)^{1/q}\quad (0 < \nu < 2),
\end{equation}
where $\Delta_h f (x)= f(x + h) - f(x) $ and $\Delta_h^2f(x) = \Delta_h\left(\Delta_h f \left( x \right) \right)$ for $x, h \in \bR^d$.
For instance, first suppose that $k=0$ and $u \in \bH_{p, q, w}^{\alpha, 2}(T)$ ($w(t, x) = t^{\mu}$) is sufficiently smooth. Note that, for any $s \in (0, T)$,
\[
\Delta_h^2 u(0, x) = -\left( u \left( s, x +  2h \right) - u \left( 0, x + 2 h \right) \right) + 2\left( u \left( s, x +  h \right) - u \left( 0, x + h \right) \right)
\]
\[
-\left( u \left( s, x \right) - u \left( 0, x  \right) \right) + \Delta_h^2 u(s, x).
\]
By multiplying $(t-s)^{-\alpha}$ and then taking the integral with respect to $s \in (0, t)$ on both sides, for any $t \in (0, T]$, we have
\[
| \Delta_h^2 u(0, x) | \le N(\alpha) \sum_{m = 0}^2 t^{\alpha - 1} \left| \int_0^t f(s, x + m h)\,\mathrm{d}s \right| + N(\alpha) t^{\alpha - 1} \left| \int_0^t \Delta_h^2u(s, x)\,\mathrm{d}s \right|,
\]
where $f = \partial_t^{\alpha}u$.
Then, $\| \Delta_h^2 u(0, \cdot ) \|_{L_p}$ is bounded by $N(\alpha)$ times
\[
 t^{\alpha - 1} \int_0^t \|f(s, \cdot)\|_{L_p} \,\mathrm{d}s +  t^{\alpha - 1} \int_0^t \|\Delta_h^2u(s, \cdot) \|_{L_p}\,\mathrm{d}s := A(t) + B(t)
\]
for $t \in (0, T]$.
Now we take $t = |h|^{2/\alpha}$ for $h \in \bR^d$ ($|h| \le T^{\alpha/2}$) and let $\theta = 1 - (1+\mu)/q\alpha > 0$.
By direct computation, for example, by performing a change of variables,
\[
\int_{ \{|h| \le T^{\alpha/2}\}} |h|^{-2\theta q} A(|h|^{2/\alpha})^q \,\frac{\mathrm{d}h}{|h|^d} = N \int_0^{T}  \left(t^{-1} \int_0^t \|f(s, \cdot)\|_{L_p}\,\mathrm{d}s \right)^q t^{\mu}\,\mathrm{d}t
\]
\[
\le N(d, \alpha, q, \mu, T)\int_0^{T} \|f(s, \cdot)\|_{L_p}^q t^{\mu}\,\mathrm{d}t,
\]
where the last inequality is due to Hardy's inequality.
Similarly, we also have
\[
\int_{ \{|h| \le T^{\alpha/2}\}} |h|^{-2\theta q} B(|h|^{2/\alpha})^q \,\frac{\mathrm{d}h}{|h|^d} \le N(d, \alpha, q, \mu, T)\int_0^{T} \|D^2 u(s, \cdot)\|_{L_p}^q t^{\mu}\,\mathrm{d}t,
\]
with the help of \eqref{eq0224_03} of Lemma \ref{lem0224}. It is clear that
\[
\int_{ \{|h| > T^{\alpha/2}\}} |h|^{-2\theta q} \| \Delta_h u(0, \cdot)\|_{L_p}^q \,\frac{\mathrm{d}h}{|h|^d} \le N(d, \alpha, q, \mu, T) \|u(0, \cdot)\|_{L_p}^q.
\]
Thus by Lemma \ref{lem0525_1}, we have
\[
\| u(0, \cdot) \|_{B_{p, q}^{2\theta}} \le N(d, \alpha, p, q, \mu, T) \| u \|_{\bH_{p, q, w}^{\alpha, 2}(T)}.
\]

Now we consider $k > 0$.
Let $u \in \bH_{p, q, w}^{k+\alpha, 2}(T)$ ($w(t, x) = t^{\mu}$) be sufficiently smooth.
We use \eqref{equ0101_01} where $1/\beta$ is replaced by $2/\beta$ (due to the scaling issue). Then by taking $\varepsilon = |h|$ in \eqref{equ0101_01}, for $\theta_n = ( k + \alpha - n - \frac{1+\mu}{q} ) / \left( k+\alpha \right)$ we obtain
\[
\| \partial_t^n u(0, \cdot) \|_{B_{p, q}^{2\theta_n}} \le N(d, \alpha, p, q, \mu, k, n, T) \| u \|_{\bH_{p, q, w}^{k + \alpha, 2}(T)}
\]
by estimating the norm in \eqref{eq2306291250} with the help of \eqref{eq0224_03} of Lemma \ref{lem0224}. We omit the details since it is similar to the estimation about $U_0$ and $U_1$ in the proof of Theorem \ref{thm0111_01}.
For the case $u \in \cH_{p, q, w}^{k+\alpha, 1}(T)$ ($w(t, x) = t^{\mu}$), we use $(1 - \Delta)^{-1/2}$; for example, see the proof of Theorem \ref{cor0111_01}.
\end{remark}

\subsection{Proof of Theorems \ref{thm_subdiff} and \ref{thm_supdiff}}

In this subsection, we only consider $k = 0$ or $1$, that is, $k + \alpha = \beta \in (0, 2)$. 
We start with the following lemmas.

\begin{lemma}
\label{lem0207}
Let $\mathfrak{T}$ be an operator defined on $\cS(\bR^d)$ whose Fourier multiplier $m$ is a bounded function on $\bR^d$, i.e., $\cF(\mathfrak{T} f)(\xi) = m(\xi)\cF(f)(\xi)$, $\xi \in \bR^d$. If there is a constant $ \mathfrak{D}(m) >0$ such that
\begin{equation}
	\label{eq0207_01}
	\sup_{r > 0} \left( r^{2|\gamma| - d} \int_{r < |\xi| < 2r} \left| D^{\gamma}m(\xi) \right|^2\,\mathrm{d}\xi \right)^{1/2} \le \mathfrak{D}(m)
\end{equation}
for any multi-index $\gamma$ with $|\gamma| \le d$, then for any $w_2 \in A_p$ with $p \in (1, \infty)$, we have
\begin{equation}
	\label{eq0207_02}
\| \mathfrak{T} f\|_{L_{p, w_2}} \le N(d, p, [w_2]_{A_p})\mathfrak{D}(m) \| f\|_{L_{p, w_2}}.
\end{equation}
\end{lemma}

Lemma \ref{lem0207} is a particular case of \cite[Theorem 1-(1)]{MR0542885} with $s=2$ and $n=l=d$ in there.
In particular, in \cite[Theorem 1-(1)]{MR0542885}, the authors do not specify how the estimate \eqref{eq0207_02} depends on $\mathfrak{D}(m)$, but by keeping track of the role of $\mathfrak{D}(m)$ in its proof, one verifies that it is of the above form.

\begin{lemma}
\label{lem0201}
Let $\alpha \in (0, 1)$ and $f \in L_{p, w_2}$ with $p \in (1, \infty)$, $w_2 \in A_p$. Then for any $\kappa \in (0, \alpha)$, we have
\begin{equation}
\label{eq0131_01_2}
\| \left( P_{\alpha} \ast \psi_j \right) \ast f (t, \cdot)\|_{L_{p, w_2}} \le N(d, \alpha, \kappa, [w_2]_{A_p}) \left( 2^{-2\kappa j/\alpha} t^{-\kappa} \wedge 1 \right) \| f\|_{L_{p, w_2}}
\end{equation}
for $j =   0, \pm 1, \pm 2, \ldots$, and
\begin{equation}
\label{eq0131_01}
\| \left( P_{\alpha} \ast \Psi \right) \ast f (t, \cdot)\|_{L_{p, w_2}} \le N(d, \alpha, [w_2]_{A_p})  \| f\|_{L_{p, w_2}}.
\end{equation}
\end{lemma}

\begin{proof}
For simplicity, set $P_{\alpha} = P$ and $P_{\alpha} \ast \psi_j = P_j$ for $j =   0, \pm 1, \pm 2, \ldots$. By using the integral representation \eqref{eq0209_02} in Remark \ref{rmk_mittag} with the fact that $\hat{P}_j = \hat{P}\, \hat{\psi_j}$ and $\operatorname{supp} \hat{\psi}_j \subset \{2^{j-1} \le |\xi| \le 2^j \}$, we verify that
\begin{equation}
\label{eq0209_01}
\left| D_{\xi}^{\gamma} \hat{P}_j (t, \xi) \right| \le N(d, \alpha, \kappa, \gamma)2^{-j|\gamma|}  \left(  2^{-2j/\alpha} t^{-1} \right)^{\kappa} 1_{2^{j-1} \le |\xi| \le 2^{j+1}}
\end{equation}
for any multi-index $\gamma$ and for any $\kappa \in (0, \alpha)$. 
Indeed, for $\gamma = 0$ and for any $\kappa \in (0, \alpha)$, we have
\[
\left| \hat{P}_j(t, \xi) \right| \le N 1_{2^{j-1} \le |\xi| \le 2^{j+1}} \int_0^{\infty} \frac{r^{\alpha - 1}}{r^{2\alpha} + 2r^{\alpha} \cos{\alpha \pi} + 1} e^{-rt |\xi|^{2/\alpha}}\,\mathrm{d}r 
\]
\[
\le N1_{2^{j-1} \le |\xi| \le 2^{j+1}} \int_0^{1} r^{\alpha - 1} e^{-rt |\xi|^{2/\alpha}} \left( rt |\xi|^{2/\alpha} \right)^{\kappa}\left( rt |\xi|^{2/\alpha} \right)^{-\kappa} \,\mathrm{d}r 
\]
\[
+ N1_{2^{j-1} \le |\xi| \le 2^{j+1}} \int_1^{\infty} r^{-\alpha - 1} e^{-rt |\xi|^{2/\alpha}}  \left( rt |\xi|^{2/\alpha} \right)^{\kappa}\left( rt |\xi|^{2/\alpha} \right)^{-\kappa} \,\mathrm{d}r
\]
\[
\le N 1_{2^{j-1} \le |\xi| \le 2^{j+1}} \left(  2^{-2j/\alpha} t^{-1} \right)^{\kappa}  \left( \int_0^{1} r^{\alpha - \kappa - 1} \,\mathrm{d}r + \int_1^{\infty} r^{-\alpha - \kappa - 1} \,\mathrm{d}r \right)
\]
\[
\le N(d, \alpha, \kappa) 1_{2^{j-1} \le |\xi| \le 2^{j+1}}  \left(  2^{-2j/\alpha} t^{-1} \right)^{\kappa},
\]
where we use the fact that $e^{-rt |\xi|^{2/\alpha}} \left( rt |\xi|^{2/\alpha} \right)^{\kappa+m}$ is uniformly bounded in $(r, t, \xi)$ for any $m \ge 0$ and the last inequality is due to $\kappa \in (0, \alpha)$.
The same argument holds for $\gamma$ such that $|\gamma| \neq 0$, with the help of $|D_{\xi}^{\gamma_1} \hat{\psi}_j| \le N(d) 2^{-j|\gamma_1|}$ for any multi-index $\gamma_1$.
Then
$$
\sup_{r > 0} \left( r^{2|\gamma| - d} \int_{r < |\xi| < 2r} \left| D_{\xi}^{\gamma}\hat{P}_j(t, \xi) \right|^2\,\mathrm{d}\xi \right)^{1/2}
$$
$$
 = N 2^{-j|\gamma| -2\kappa j/\alpha} t^{-\kappa}  \sup_{2^{j-2} < r < 2^{j+1}} \left(  r^{2|\gamma| - d} \int_{r < |\xi| < 2r} 1_{2^{j-1} \le |\xi| \le 2^{j+1}} \,\mathrm{d}\xi \right)^{1/2}
$$
$$
\le N \left( 2^{-2j/\alpha}t^{-1} \right)^{\kappa}
$$
for $N = N(d, \alpha, \kappa, \gamma) > 0$.

On the other hand, since $|D_{\xi}^{\gamma} \hat{P} (t, \xi)| \le N \| |\cdot|^{|\gamma|} P(t, \cdot) \|_{L_1}$ and $P(1, x)$ satisfies \eqref{eq0209_03} in Remark \ref{rmk_mittag}, by direct calculation we get
$$
\left| D_{\xi}^{\gamma} \hat{P}_j (t, \xi) \right| \le N \sum_{|\gamma_1| \le |\gamma|} 1_{2^{j-1} \le |\xi| \le 2^{j+1}} 2^{-j |\gamma_1|} t^{ \left|    \left( \gamma \right| - \left| \gamma_1 \right| \right) \alpha/2}
$$
for any multi-index $\gamma$, and then,
$$
\sup_{r > 0} \left( r^{2|\gamma| - d} \int_{r < |\xi| < 2r} \left| D_{\xi}^{\gamma}\hat{P}_j(t, \xi) \right|^2\,\mathrm{d}\xi \right)^{1/2} \le N \sum_{|\gamma_1| \le |\gamma|}  t^{ \left(    \left| \gamma \right| - \left| \gamma_1 \right| \right) \alpha/2} 2^{j\left( \left| \gamma \right| - \left| \gamma_1 \right| \right)},
$$
where $N = N(d, \alpha, \gamma) > 0$.
Observe that if $t\le 2^{-2j/\alpha}$, the last summation in the above inequality is bounded by $ \sum_{|\gamma_1| \le |\gamma|} 1 = N(\gamma)$. Therefore, for any multi-index $\gamma$ with $|\gamma| \le d$, we have \eqref{eq0207_01} with $\hat{P}_j$ in place of $m$, where
$$
\mathfrak{D}(\hat{P}_j) \le N(d, \alpha, \kappa) \left(  2^{-2j/\alpha}t^{-1} \right)^{\kappa} \wedge 1.
$$
This along with Lemma \ref{lem0207} proves \eqref{eq0131_01_2}.

To prove \eqref{eq0131_01}, recall that $\hat{\Psi} \in C_0^{\infty}(\bR^d)$, $\operatorname{supp} \hat{\Psi} \subset \{|\xi| \le 2 \}$ and $\hat{\Psi} = 1$ on $\{|\xi| \le 1 \}$.
It is easily seen that $\hat{\Psi}$ satisfies \eqref{eq0207_01} with $\mathfrak{D}(\hat{\Psi}) = N(d)$ and then by Lemma \ref{lem0207},
\[
\| \left( P_{\alpha} \ast \Psi \right) \ast f (t, \cdot)\|_{L_{p, w_2}} = \|  \Psi \ast \left( P_{\alpha} \ast f \right)  (t, \cdot)\|_{L_{p, w_2}} \le N(d) \|   P_{\alpha} \ast f  (t, \cdot)\|_{L_{p, w_2}}.
\]
Also, by a suitable dyadic decomposition, we have $P_{\alpha} \ast f \le N(d, \alpha) \cM f$ where $\cM$ is the Hardy-Littlewood maximal operator.
Then by Hardy-Littlewood maximal function theorem with $A_p$-weights, we obtain \eqref{eq0131_01}. The lemma is proved.
\end{proof}

We prove the key estimate of this subsection in the following proposition.

\begin{proposition}
\label{prop0201}
Let $p, q \in (1, \infty)$, $w = t^{\mu}w_2(x)$ with $\mu \in (-1, q-1)$, $w_2 \in A_p$. Suppose $\alpha \in \left( \left(1 + \mu \right)/q, 1\right)$. Then, for $f \in C_0^{\infty}(\bR^d)$, we have
\begin{equation}
\label{eq0131_04}
\| P_{\alpha} \ast f \|_{L_{p, q, w}(T)} \le N \| f \|_{B_{p, q, w_2}^{-\frac{2(1+\mu)}{q\alpha}}},
\end{equation}
where $N = N(d, \alpha, p, q, [w_2]_{A_p}, \mu, T) > 0$.
\end{proposition}

\begin{proof}
We set $P_{\alpha} = P$ for notational simplicity. Also, for $g \in \{ f, P\}$, we set $g_j = g \ast \psi_j$, $j = 1, 2, \ldots$, and $g_0 = g \ast \Psi$ (not $g \ast \psi_0$).
Observe that  $\hat{\Psi} = \hat{\Psi}(\hat{\Psi} + \hat{\psi}_1)$, $\hat{\psi}_1 = \hat{\psi}_1(\hat{\Psi} + \hat{\psi}_1 + \hat{\psi}_{2})$, and $\hat{\psi}_j = \hat{\psi}_j(\hat{\psi}_{j-1} + \hat{\psi}_j + \hat{\psi}_{j+1})$ for $j = 2, 3, \ldots$, by the construction of $\Psi$ and $\psi_j$. Then
\[
P \ast f = P_0\ast f_0 + P_1 \ast f_0 + \sum_{j=1}^{\infty} \sum_{i = j-1}^{j+1} P_i \ast f_j,
\]
and by Lemma \ref{lem0201}, for some $\kappa \in (0, \alpha)$ ($\kappa$ is to be specified below.)
\[
\| P(t, \cdot) \ast f(\cdot) \|_{L_{p, w_2}} \le \|  P_0 \left( t, \cdot \right) \ast f_0 \left( \cdot \right) \|_{L_{p, w_2}} + \| P_1(t, \cdot) \ast  f_0\left( \cdot \right) \|_{L_{p, w_2}}
\]
\[
+ \sum_{j=1}^{\infty} \sum_{i = j-1}^{j+1}  \| P_i(t, \cdot) \ast f_j(\cdot) \|_{L_{p, w_2}}
\]
\[
\le N  \left( \| |f_0| + |f_1| \|_{L_{p, w_2}} +  \sum_{j=1}^{\infty}  \left(   2^{-2\kappa j/\alpha} t^{-\kappa} \wedge 1  \right) \| f_j\|_{L_{p, w_2}}     \right),
\]
where $N = N(d, \alpha, \kappa, p, [w_2]_{A_p}) > 0$.
From the above inequalities, it suffices to estimate $\cI$ and $\mathcal{J}$ where
\[
\cI :=  \int_0^T \| |f_0| + |f_1| \|_{L_{p, w_2}}^q t^{\mu}\,\mathrm{d}t 
\]
and
\[
 \mathcal{J} :=  \int_0^T \left( \sum_{j=1}^{\infty}  \left(   2^{-2\kappa j/\alpha} t^{-\kappa} \wedge 1  \right) \| f_j\|_{L_{p, w_2}} \right)^q t^{\mu}\,\mathrm{d}t .
\]
It is clear that
\[
\cI \le N(\alpha, q, \mu, T) \left(  \| f_0 \|_{L_{p, w_2}}^q + \| f_1 \|_{L_{p, w_2}}^q \right),
\]
where $s = -2(1+\mu)/(q\alpha)$.
On the other hand, note that $\mathcal{J}$ is bounded by a positive constant times
\[
 \int_0^T \left( \sum_{j=1}^{\infty}  \| f_j\|_{L_{p, w_2}} 1_{t \le 2^{-2j/\alpha}} \right)^q t^{\mu}  \,\mathrm{d}t
 \]
 \[
   +  \int_0^T \left( \sum_{j=1}^{\infty}  2^{-2\kappa j/\alpha}t^{-\kappa} \| f_j\|_{L_{p, w_2}} 1_{t > 2^{-2j/\alpha}} \right)^q t^{\mu}  \,\mathrm{d}t
:=  \mathcal{J}_1 +  \mathcal{J}_2.
\]
Estimate of $ \mathcal{J}_1$:
Observe that, for $t > 2^{-2/\alpha}$, the integrand of $\mathcal{J}_1$ is always zero. For each $t \in (0, 2^{-2/\alpha}]$, there exits the largest positive integer $j_1 = j_1(t)$ such that $t \le 2^{-2j_1/\alpha}$.
Then by H\"older's inequality, for some $r <0$,
\[
 \left( \sum_{j=1}^{\infty}  \| f_j\|_{L_{p, w_2}} 1_{t \le 2^{-2j/\alpha}} \right)^q =   \left( \sum_{j=1}^{j_1(t)}   \| f_j\|_{L_{p, w_2}}  \right)^q 1_{t \le 2^{-2j/\alpha}}
 \]
 \[
\le  \left( \sum_{j=1}^{j_1(t)} \left( 2^{-r \frac{2j}{\alpha}} \right)^{q'}    \right)^{q/q'}
 \sum_{j=1}^{j_1(t)} \left( 2^{r \frac{2j}{\alpha}} \right)^{q}     \| f_j\|_{L_{p, w_2}}^q 1_{t \le 2^{-2j/\alpha}}
 \]
 \[
\le N   \left( 2^{-2j_1/\alpha} \right)^{qr}  \sum_{j=1}^{j_1} \left( 2^{r \frac{2j}{\alpha}}     \| f_j\|_{L_{p, w_2}}\right)^{q} 1_{t \le 2^{-2j/\alpha}}
\]
\[
 \le  N t^{qr}  \sum_{j=1}^{j_1} \left( 2^{r \frac{2j}{\alpha}} \right)^{q}     \| f_j\|_{L_{p, w_2}}^q 1_{t \le 2^{-2j/\alpha}},
\]
where $N = N(r, \alpha, q)$ and the last inequality is due to $t \le 2^{-2j_1/\alpha}$ and $r < 0$.
Then, by taking $r < 0$ such that $qr+ \mu > -1$, i.e., $-(1+\mu)/q < r < 0$ (recall that $-(1+\mu)/q < 0$ since $\mu > -1$),
\[
 \mathcal{J}_1 \le N \sum_{j=1}^{\infty} \left( \int_0^{2^{-2j/\alpha}} t^{qr+\mu}\,\mathrm{d}t \right) 2^{\frac{2j}{\alpha}qr} \| f_j\|_{L_{p, w_2}}^q  \le N \sum_{j=1}^{\infty}   2^{jqs}\| f_j\|_{L_{p, w_2}}^q,
\]
where $s = - 2(1+\mu)/(q\alpha)$ and $N = N(\alpha, q, \mu)$.

Estimate of $ \mathcal{J}_2$: By repeating the above process with $t > 2^{-2j/\alpha}$, 
where $j_2 = j_2(t)$ is the smallest positive integer such that $t > 2^{-2j_2/\alpha}$,
\begin{align*}
 &\left(\sum_{j=1}^{\infty}  2^{-2\kappa j/\alpha}t^{-\kappa} \| f_j\|_{L_{p, w_2}} 1_{t > 2^{-2j/\alpha}} \right)^q
 \\
  &\le t^{-\kappa q}\left( \sum_{j = j_2}^{\infty} \left( 2^{-\frac{2j}{\alpha}(\kappa+r')} \right)^{q'} \right)^{q/q'}  \sum_{j = j_2}^{\infty} 2^{\frac{2j}{\alpha}q r'} \| f_j \|_{L_{p, w_2}}^q 1_{t > 2^{-2j/\alpha}}
\\
 &= Nt^{-\kappa q} 2^{-\frac{2j_2}{\alpha}q(\kappa + r')} \sum_{j = j_2}^{\infty} 2^{\frac{2j}{\alpha}q r'} \| f_j \|_{L_{p, w_2}}^q 1_{t > 2^{-2j/\alpha}}
\end{align*}
for some $r' > -\kappa$ such that $q r' + \mu < -1$ and $N = N(\alpha,q,\mu)$.
Note that such $r'$ always exists since $\kappa < \alpha$ and $\alpha > (1+\mu)/q$. Then, from the fact that $t > 2^{-2j_2/\alpha}$,
\[
 \mathcal{J}_2 \le N \sum_{j=1}^{\infty} \left( \int_{2^{-2j/\alpha}}^{\infty} t^{qr + \mu}\,\mathrm{d}t \right) 2^{\frac{2j}{\alpha}qr} \| f_j \|_{L_{p, w_2}}^q = N \sum_{j=1}^{\infty} 2^{jqs}\| f_j\|_{L_{p, w_2}}^q,
\]
where $s = - 2(1+\mu)/(q\alpha)$ and $N = N(\alpha, q, \mu)$.
Thus, by taking $\kappa \in (0, \alpha)$ such that $\kappa > (1+\mu)/q $ (for example, $2\kappa = \alpha + (1+\mu)/q$) and combining the estimates for $\mathcal{I}$ and $\mathcal{J}$, we have
\[
\| P_{\alpha} \ast f \|_{L_{p, q, w}(T)}^q \le N \left( \| f_0 \|_{L_{p, w_2}}^q + \sum_{j=1}^{\infty}  2^{jqs}\| f_j\|_{L_{p, w_2}}^q \right),
\]
that is, \eqref{eq0131_04} with $N = N(d, \alpha, p, q, [w_2]_{A_p}, \mu, T)$. The proposition is proved.
\end{proof}

With the help of Proposition \ref{prop0201}, we prove Theorem \ref{thm_subdiff}, the extension theorem for $\bH_{p, q, w}^{\alpha, 2}(T)$ and $\cH_{p, q, w}^{\alpha, 1}(T)$.

\begin{proof}[\textbf{Proof of Theorem \ref{thm_subdiff}}]
($i$)
It is enough to find $U(t, x) \in \bH_{p, q, w}^{\alpha, 2}(T)$ satisfying \eqref{eq0202_03} for $u_0(x) \in C_0^{\infty}(\bR^d)$. Take $U$ as in \eqref{eq0131_03}, that is, $U(t, \cdot) = P_{\alpha}(t, \cdot) \ast u_0(\cdot)$. Then by applying Proposition \ref{prop0201} to $(1 - \Delta)u_0$, we have
\[
\| |U| + |DU| + |D^2U| \|_{L_{p, q, w}(T)} \le  N \| (1 - \Delta) u_0 \|_{B_{p, q, w_2}^{2\theta - 2}} =  N\| u_0 \|_{B_{p, q, w_2}^{2\theta}},
\]
where $N = N(d, \alpha, p, q, [w_2]_{A_p}, \mu, T)$. This implies \eqref{eq0202_03} since $U$ satisfies $\partial_t^{\alpha}U = \Delta U$ in $\bR^d_T$.

($ii$) Let $w_0 = (1-\Delta)^{-1/2}v_0 \in B_{p, q, w_2}^{2\theta}$. Then by ($i$), there is a solution $W \in \bH_{p, q, w}^{\alpha, 2}(T)$ to \eqref{eq0203_01} satisfying $W(0, \cdot) = w_0(\cdot)$ and 
\[
\| W \|_{\bH_{p, q, w}^{\alpha, 2}(T)} \le N \| w_0 \|_{B_{p, q, w_2}^{2\theta}} = N \| v_0 \|_{B_{p, q, w_2}^{2\theta-1}},
\]
where $N = N(d, \alpha, p, q, [w_2]_{A_p}, \mu, T)$. 
By taking $V = (1 - \Delta)^{1/2}W \in \cH_{p, q, w}^{\alpha, 1}(T)$, we obtain \eqref{eq0224_01}.
In particular, since $\alpha > (1+\mu)/q$, $U(0, \cdot) = u_0(\cdot)$ and $V(0, \cdot) = v_0(\cdot)$ are well-defined in the trace sense by Theorems \ref{thm0111_01} and \ref{cor0111_01}, respectively.
The theorem is proved.
\end{proof}

\begin{remark}
	\label{rmk_homogeneous}
Recall that in order for the initial trace $u(0, \cdot) \in L_{p, w_2}(\Omega)$ to make sense for $u \in \bH_{p, q, w}^{\alpha, 0}(\Omega_T)$, we need $\alpha \ge (1+\mu)/q$ with $\mu \in (-1, q-1)$.
On the other hand, if $u_0 \in L_{p, w_2}(\Omega)$, the initial value of the homogeneous solution $U = P_{\alpha}(t, \cdot) \ast u_0(\cdot)$ is well-defined regardless of the range of $\alpha \in (0, 1)$.
This is made possible by appropriately choosing the time integrability $q$ and the power $\mu$ of weight $w_1(t) = t^{\mu}$.
In fact, this is also one of the advantages of the weighted $L_{p, q}$-theory.
For example, we simply consider an unweighted space, i.e., $w(t, x) = 1$, and an initial value $u_0 \in L_{p}$.
For a given $\alpha \in (0, 1)$, we choose $q \in (1, \infty)$ such that $1/\alpha < q < 2/\alpha$, i.e., $1 < q\alpha < 2$.
Then, for $s = 1 - 2/(q\alpha) < 0$, we have
\[
\|u_0\|_{B_{p, q}^{s}} \le N \| u_0 \|_{L_{p}},
\]
which means that $u_0 \in L_{p} \subset B_{p, q}^{1 - 2/(q\alpha)}$. 
Together with the condition $\alpha > 1/q$, this ensures that the solution $U \in \cH_{p, q}^{\alpha, 1}(T)$ satisfies $U(0, \cdot) = u(\cdot)$ in the trace sense by Theorem \ref{cor0111_01} and Theorem \ref{thm_subdiff}-($ii$).
On the other hand, by direct calculation, one can see that
\[
\| U(t, \cdot) - u_0(\cdot) \|_{L_{p}(\Omega)} \to 0 \quad \textrm{as} \quad t \to 0.
\]
For another perspective, we refer the reader to \cite{MR3814402} and \cite{MR4225517}, which discuss the $L_{2}$-theory.
\end{remark}

The following is a version of Lemma \ref{lem0201} for the case of $\beta = 1+ \alpha \in (0, 1)$.

\begin{lemma}
\label{lem0220}
Let $\alpha \in (0, 1)$, $\beta = 1 + \alpha$ and $f \in L_{p, w_2}$ with $p \in (1, \infty)$, $w_2 \in A_p$. Then for any $\kappa_1 \in (0, \beta)$ and $\kappa_2 \in (0, \alpha)$, we have
\begin{equation}
\label{eq0220_01}
\| \left( P_{\beta} \ast \psi_j \right) \ast f (t, \cdot)\|_{L_{p, w_2}} \le N(d, \alpha, \kappa_1, [w_2]_{A_p}) \left( 2^{-2\kappa_1 j/\beta} t^{-\kappa_1} \wedge 1 \right) \| f\|_{L_{p, w_2}},
\end{equation}
\begin{equation}
\label{eq0220_02}
\| \left( \widetilde{P}_{\beta} \ast \psi_j \right) \ast f (t, \cdot)\|_{L_{p, w_2}} \le N(d, \alpha, \kappa_2, [w_2]_{A_p}) \left( 2^{-2j/\beta}2^{-2\kappa_2 j/\beta} t^{-\kappa_2} \wedge t \right) \| f\|_{L_{p, w_2}}
\end{equation}
for $j =   0, \pm 1, \pm 2, \cdots$, and
\begin{equation}
\label{eq0220_03}
\left\| \left( P_{\beta} \ast \Psi \right) \ast f (t, \cdot) \right\|_{L_{p, w_2}} \le N(d, \alpha, [w_2]_{A_p})  \| f\|_{L_{p, w_2}},
\end{equation}
\begin{equation}
\label{eq0220_04}
\left\| \left( \widetilde{P}_{\beta} \ast \Psi \right) \ast f (t, \cdot) \right\|_{L_{p, w_2}} \le N(d, \alpha, [w_2]_{A_p})  t \| f\|_{L_{p, w_2}}.
\end{equation}
\end{lemma}

\begin{proof}
The inequality \eqref{eq0220_01} can be obtained in the same way as in the proof of Lemma \ref{lem0201}.
The only difference is, since $\beta > 1$, we need to use the representation \eqref{eq0223_01} in Remark \ref{rmk_mittag} for two parametric Mittag-Leffler function $E_{\beta, 1}(= E_{\beta})$ instead of \eqref{eq0209_02}.
Note that when we investigate \eqref{eq0223_01}, the fact that $\cos(\delta/\beta) < 0$ for any $\delta \in (\frac{\pi \beta}{2}, \pi]$ is useful.
Similarly, we get \eqref{eq0220_02} with the help of the relation between the Fourier transform of $\widetilde{P}_{\beta}$ and two parametric Mittag-Leffler function $E_{\beta, 2}$, that is, \eqref{eq0223_02} in Remark \ref{rmk_mittag}.
We also get \eqref{eq0220_03} the exact same way as in the proof of \eqref{eq0131_01}.
Then \eqref{eq0220_04} directly follows from \eqref{eq0220_03}, since $\widetilde{P}_{\beta}(t, x) = \int_0^t P(s, x)\,\mathrm{d}s$. The lemma is proved.
\end{proof}

By following the proof of Theorem \ref{thm_subdiff} with the help of Lemma \ref{lem0220}, we prove Theorem \ref{thm_supdiff}.

\begin{proof}[\textbf{Proof of Theorem \ref{thm_supdiff}}]
We omit the proof since it is almost identical to the proof of Theorem \ref{thm_subdiff}, except we use Lemma \ref{lem0220} instead of Lemma \ref{lem0201}. We remark that we must have $\alpha > (1 + \mu)/q$ to be $\theta_1 \in (0, 1)$, but regarding $\theta_0$, we always have $\theta_0 \in (0, 1)$ for any $\alpha \in (0, 1)$ because $\mu < q-1$.
\end{proof}

\bibliographystyle{plain}

\def\cprime{$'$}

\end{document}